\documentclass{amsart}
\usepackage{etex}
\reserveinserts{28}
\usepackage{xypic}
\usepackage{tikz}
\usepackage{tikz-cd}

\usepackage{pst-node}

\usepackage{kbordermatrix}

\usepackage{subcaption}

\usepackage{MnSymbol}

\usepackage{latexsym}
\usepackage{stmaryrd}

\usepackage[capitalize,nameinlink,noabbrev,nosort]{cleveref}

\usepackage{pictex}
\usepackage{amsmath,amscd}

\usepackage{afterpage}
\usepackage{youngtab}
\usepackage{psfrag}
\usepackage[boxsize=.3 em]{ytableau}

\usepackage{verbatim}

\usepackage{graphics}

\newtheorem{thm}{Theorem}[section]

\newtheorem{theorem}[thm]{Theorem}
\newtheorem{cor}[thm]{Corollary}
\newtheorem{corollary}[thm]{Corollary}

\newtheorem{lem}[thm]{Lemma}
\newtheorem{lemma}[thm]{Lemma}

\newtheorem{prop}[thm]{Proposition}
\newtheorem{proposition}[thm]{Proposition}
\theoremstyle{definition}
\newtheorem{defn}[thm]{Definition}
\newtheorem{example}[thm]{Example}
\newtheorem{definition}[thm]{Definition}
\newtheorem{rem}[thm]{Remark}  
\newtheorem{remark}[thm]{Remark}

\numberwithin{equation}{section}

\newcommand{\Cone}{\operatorname{Cone}}

\newcommand{\Zones}{\operatorname{Zones}(\checkX)}
\newcommand{\RZones}{\operatorname{Trop}(\checkX)}

\newcommand{\LL}{\mathbf L}

\newcommand{\LW}{\mathbf W}

\newcommand{\Match}{\operatorname{Match}}

\newcommand{\Lattice}{\operatorname{Lattice}}

\newcommand{\conv}{\operatorname{Conv}}
\newcommand{\Conv}{\operatorname{Conv}}
\newcommand{\Mut}{\operatorname{Mut}}
\newcommand{\MutVar}{\operatorname{MutVar}}
\newcommand{\NO}{\Delta}

\newcommand{\R}{\mathbb R}

\newcommand{\To}{\longrightarrow}
\newcommand{\Max}{{\mathrm{max}}}
\newcommand{\prin}{{\mathrm{prin}}}
\newcommand{\RG}{r_G}
\newcommand{\Vol}{\operatorname{Volume}}
\newcommand{\gr}{\operatorname{gr}}
\newcommand{\oomega}{w}

\newcommand{\PDelta}{\Delta}

\newcommand{\Q}{\Gamma}

\newcommand{\pyo}{p_{\ydiagram{1}}}
\newcommand{\pyz}{p_{\ydiagram{2}}}
\newcommand{\pyoo}{p_{\ydiagram{1,1}}}
\newcommand{\pyzz}{p_{\ydiagram{2,2}}}
\newcommand{\pyt}{p_{\ydiagram{3}}}
\newcommand{\pytt}{p_{\ydiagram{3,3}}}

\newcommand{\Sh}{\mathcal{P}}

\newcommand{\PGhat}{\tilde{\mathcal{P}}_G \setminus \max}

\newcommand{\Shkn}{{\mathcal P_{k,n}}}
\newcommand{\rect}{\mathrm{rec}}
\newcommand{\Mono}{\mathrm{Mono}}
\newcommand{\wt}{\mathrm{wt}}

\def\O{\mathcal{O}}
\newcommand{\C}{\mathbb{C}}
\newcommand{\Z}{\mathbb{Z}}
\newcommand{\inv}{^{-1}}

\newcommand{\x}{\times}

\newcommand{\Dac}{D_{\operatorname{ac}}}
\newcommand{\Grec}{G_{\operatorname{rec}}}

\newcommand{\Hom}{{\operatorname{Hom}}}

\newcommand{\val}{{\operatorname{val}}}
\newcommand{\ValK}{{\operatorname{Val}_{\mathbf K}}}
\newcommand{\ValKhigh}{{\operatorname{Val}^{\mathbf K}}}
\newcommand{\ValL}{{\operatorname{Val}_{\mathbf L}}}
\renewcommand{\min}{{\operatorname{min}}}
\newcommand{\Pmin}{{\operatorname{m}}}
\newcommand{\maxdiag}{{\operatorname{MaxDiag}}}
\newcommand{\diag}{{\operatorname{Diag}}}

\newcommand{\PCG}{{\mathcal A\hskip -.05cm\operatorname{Coord}_{\checkX}(G)}}
\newcommand{\PC}{{\mathcal A\hskip -.05cm\operatorname{Coord}_{\checkX}}}

\newcommand{\wPC}{\widetilde{\mathcal A\hskip -.05cm\operatorname{Coord}}_{\checkX}}

\newcommand{\PCGA}{{\mathcal A\hskip -.05cm\operatorname{Coord}_{\X}}(G)}

\newcommand{\Proj}{\operatorname{Proj}}
\newcommand{\Trop}{\operatorname{Trop}}
\newcommand{\Sym}{\operatorname{Sym}}
\newcommand{\TBG}{{\mathcal X\hskip -.05cm\operatorname{Coord}_{\X}}(G)}
\newcommand{\TB}{{\mathcal X\hskip -.05cm\operatorname{Coord}_{\X}}}
\newcommand{\Network}{{\mathcal X\hskip -.05cm\operatorname{Coord}_{\X}}}
\newcommand{\PosSet}{\operatorname{PosSet}}

\newcommand{\AAA}{\mathbf{A}}
\newcommand{\BBB}{\mathbf{B}}
\newcommand{\CCC}{\mathbf{C}}
\newcommand{\DDD}{\mathbf{D}}
\newcommand{\EEE}{\mathbf{E}}
\newcommand{\FFF}{\mathbf{F}}

\newcommand{\mathbbX}{\mathbb X}
\newcommand{\checkX}{\check{\mathbb X}}
\newcommand{\X}{\mathbb X}
\newcommand{\opencheckX}{\check{\mathbb X}^\circ}
\newcommand{\openX}{{\mathbb X}^\circ}

\newcommand{\Xcirc}{{\mathbb X}^\circ}
\newcommand{\MonSeq}{\operatorname{MonSeq}}

\newcommand{\righttwist}{\overrightarrow{\tau}}
\newcommand{\lefttwist}{\overleftarrow{\tau}}
\newcommand{\leftdel}{\overleftarrow{\partial}}

\newcommand{\ac}{\operatorname{ac}}
\newcommand{\p}{{p}}

\def\CA{{\mathcal A}}

\def\CX{{\mathcal X}}

\newcommand{\mui}{{\mu_i}}
\newcommand{\muibox}{{\mu_i^{\ydiagram{1}}}}

\begin{document}

\title
{Newton-Okounkov bodies, cluster duality and mirror symmetry for Grassmannians}

\author{K. Rietsch}
\address{Department of Mathematics,
            King's College London,
            Strand, London
            WC2R 2LS
            UK
}
\email{konstanze.rietsch@kcl.ac.uk}%
\author{L. Williams}%
\address{Department of Mathematics,
            Harvard University,
            Cambridge, MA
            USA
}
\email{williams@math.harvard.edu}


\thanks{}

\begin{abstract}
In this article we use 
cluster
 structures 
and mirror symmetry to explicitly describe a natural class of Newton-Okounkov 
bodies for Grassmannians.
We consider the Grassmannian $\mathbbX=Gr_{n-k}(\C^n)$, as well 
as the mirror dual %
 \emph{Landau-Ginzburg model} $(\opencheckX, W:\opencheckX \to \C)$,
where $\opencheckX$ is the complement of a particular anti-canonical divisor
in a
Langlands dual Grassmannian $\checkX = Gr_k((\C^n)^*)$, and the superpotential
$W$ has a simple expression  in
terms of Pl\"ucker coordinates \cite{MarshRietsch}. Grassmannians
simultaneously have the structure of an $\mathcal{A}$-cluster variety and  an 
$\mathcal{X}$-cluster variety \cite{Scott, Postnikov}; roughly speaking, a cluster variety is 
obtained by gluing together a collection of 
tori along birational maps \cite{ca1, FG1}.
Given a plabic graph or, more generally, a cluster seed $G$,
we consider two associated coordinate systems: 
 a {\it network} or \emph{$\mathcal X$-cluster chart}
$\Phi_G:(\C^*)^{k(n-k)}\to \openX$ 
and a
{\it Pl\"ucker cluster} or \emph{$\mathcal A$-cluster chart}
$\Phi_G^{\vee}:(\C^*)^{k(n-k)}\to \opencheckX$.
Here $\openX$ 
and $\opencheckX$ 
 are the open positroid varieties in $\X$ and 
$\checkX$, respectively.
To each $\mathcal X$-cluster chart $\Phi_G$ 
and ample `boundary divisor' $D$ in $\mathbbX\setminus\Xcirc$, we associate 
a 
\emph{Newton-Okounkov  body} $\Delta_G(D)$ in $\R^{k(n-k)}$, 
which is defined as the convex hull of rational points; these points are obtained from the multi-degrees of 
leading terms of the Laurent polynomials $\Phi_G^*(f)$ for 
$f$ on $\X$ with poles bounded by some multiple of $D$.
On the other hand using the $\mathcal A$-cluster chart $\Phi_G^{\vee}$ on the mirror  side,
we obtain a set of rational polytopes 
-- described in terms of inequalities --
by writing the superpotential $W$ 
as a Laurent polynomial in the $\mathcal A$-cluster coordinates, and then ``tropicalising".  
Our first main result is that the Newton-Okounkov bodies $\Delta_G(D)$ and the polytopes obtained by tropicalisation on the mirror side 
coincide.  As an application, we construct degenerations of the Grassmannian to normal toric varieties corresponding to
(dilates of) these Newton-Okounkov bodies.  Our second main result is an explicit combinatorial formula in terms of Young 
diagrams, for the lattice
points of the Newton-Okounkov  bodies, in the case that the cluster seed $G$ corresponds
to a plabic graph.
This formula has an interpretation 
in terms of 
the quantum Schubert calculus of Grassmannians  \cite{FW}.
  \end{abstract}

\maketitle
\setcounter{tocdepth}{1}
\tableofcontents

\section{Introduction}

\subsection{}
Suppose that $\mathbb X=Gr_{n-k}(\C^n)$ is the Grassmannian of codimension $k$ planes in $\C^n$, embedded in $\mathbb P(\bigwedge^{n-k} \C^{n})$ via the Pl\"ucker embedding. Let 
 $N:=k(n-k)$ denote the dimension of $\X$. 
Grassmannians can be thought of as very close to toric varieties. Indeed, both Grassmannians and toric varieties are examples of spherical varieties.
Moreover the Grassmannian $\mathbbX$ has a distinguished anticanonical divisor $D_{\operatorname{ac}}=D_1+\dotsc+D_n$ made up of $n$ 
hyperplanes,
 which generalises the usual torus-invariant 
anticanonical divisor of $\mathbb {CP}^{n-1}$. We denote the complement of the divisor $\Dac$ by $\Xcirc$; this is a generalisation of the open torus-orbit in a toric variety. 
We now view the Grassmannian $\X$ as the compactification of $\Xcirc$ by the boundary divisors $D_1,\dotsc, D_n$. We consider ample divisors of the form $D=r_1D_1+\dotsc +r_nD_n$ in $\X$, and their associated finite-dimensional subspaces 
\[
L_{D}:= H^0(\X, \mathcal O(D))\subset \C(\mathbb X).
\]
Explicitly, $L_{D}$ is the space of rational functions on $\X$ that are regular on $\Xcirc$ and for which the order of pole along $D_i$ is bounded by $r_i$. By the Borel-Weil theorem, $L_D$ 
may be identified 
with the irreducible representation $V_{r\omega_{n-k}}$ of $GL_n(\C)$ where $r=\sum r_i$ and $\omega_{n-k}$ is the fundamental weight associated to $\X=Gr_{n-k}(\C^n)$.

In the toric setting one would associate to an ample divisor such as $D$ its moment polytope $P(D)$, see  \cite{Fulton}. This is a lattice polytope in $\mathfrak t_c^*$, the dual of the Lie algebra of the compact torus $T_c$ acting on the toric variety. It has the key property that its lattice points are in bijection with a basis of $L_D$, and the lattice points of the dilation $rP(D)$ are in bijection with a basis of~$L_{rD}$. 

There is a vast generalisation of this construction initiated by Okounkov, which applies in our setting of $\X=Gr_{n-k}(\C^n)$, and which can be used to associate to an ample divisor such as $D=\sum r_iD_i$ in $\X$ a convex body $\Delta(D)$ in $\R^N$, see 
\cite{Ok96, Ok03, LazarsfeldMustata, KavehKhovanskii08, KavehKhovanskii}. This so-called {\it Newton-Okounkov  body} $\Delta(D)$  again encodes the dimension of each $L_{rD}$ via the set of lattice points in the $r$-th dilation. 
In recent years Newton-Okounkov bodies have attracted a lot of attention;
they have  applications to toric degenerations and connections 
to integrable systems \cite{Anderson, HaradaKaveh}.  
However in general, Newton-Okounkov  bodies are 
quite difficult to compute: they are not necessarily rational polytopes, or even 
polytopes \cite{KLM}.

The main goal of this paper is to use mirror symmetry to describe the Newton-Okounkov  bodies of divisors $D$ as above, for a particular class of naturally occurring valuations.  We show that they are rational polytopes, by giving
formulas for the inequalities cutting them out.  We also give explicit formulas for their lattice points. We now describe our results in more detail.  

\subsection{}
We consider certain open embedded tori inside $\Xcirc$ called {\it network tori}. 
These tori $\mathbb T_G$ were introduced by Postnikov~\cite{Postnikov}, with their 
Pl\"ucker coordinates described succinctly by Talaska~\cite{Talaska}. 
They are associated to planar bicolored (plabic) graphs $G$, which have associated dual quivers $Q(G)$; the faces of $G$ (equivalently, the 
vertices of $Q(G)$) are naturally labeled by a collection $\mathcal P_G$
of  Young diagrams.
 The network tori form part of a cluster Poisson variety structure (also known as `$\mathcal X$-cluster 
structure'), and we also consider more general $\mathcal X$-cluster tori associated to quivers
but not necessarily coming from plabic graphs; we continue to denote the
tori, quivers, and vertices of the quivers by  $\mathbb T_G$, $Q(G)$, and $\mathcal P_G$. 
As part of the data such a torus has specific $\mathcal{X}$-cluster coordinates $\TBG$
 which are 
indexed by $\mathcal P_G$.
The data of the quiver together with the torus coordinates is 
called an \emph{$\mathcal X$-cluster seed} and denoted $\Sigma_G^{\mathcal{X}}$.  
As we show in~\cref{s:twist},  for a general $\mathcal X$-cluster seed $\Sigma_G^{\mathcal{X}}$ we also have an open embedding %
\[
\Phi_G:(\C^*)^{\mathcal P_G}\overset\sim\To \mathbb T_G\subset \Xcirc,
\]
where the notation $(\C^*)^{\mathcal P_G}$ refers to the torus with coordinates labeled by the unordered set $\mathcal P_G$.
Using the embedding $\Phi_G$ and a choice of ordering on $\mathcal P_G$, we define a lowest-order-term valuation 
\[
\val_G:\C(\X)\setminus\{0\}\to \Z^{\mathcal P_G}.
\]
The Newton-Okounkov  body for a  divisor $D$ with this choice of valuation is defined to be
\[
\Delta_G(D):=\overline{\operatorname{ConvexHull}\left(\bigcup_{r=1}^\infty\frac 1r\val_G(L_{rD})\right)}.
\]
Our first goal is to describe $\Delta_G(D)$ explicitly for a general $\mathcal X$-cluster seed, using mirror symmetry for $\X$. 

\subsection{} We recall the mirror Landau-Ginzburg model $(\opencheckX, W)$ for the Grassmannian $\X$ introduced in \cite{MarshRietsch}. Here $\opencheckX$ is the analogue of $\Xcirc=\X\setminus \Dac$, but inside a Langlands dual Grassmannian $\checkX$, and $W:\opencheckX \to\C$ is a regular function called the {\it superpotential}. The superpotential is given by an explicit formula in terms of Pl\"ucker coordinates as a sum of $n$ terms (and it depends on a single parameter $q$). We may think of $W$ as an element of $\C[\opencheckX][q]$.

For example, if $\X= Gr_3(\C^7)$ then the superpotential on $\opencheckX$ is given by the expression
\[
W=\frac{p_{\ydiagram{4,1}}}{p_{\ydiagram{4}}}+\frac{p_{\ydiagram{4,4,1}}}{p_{\ydiagram{4,4}}}+q\frac{p_{\ydiagram{3,3}}}{p_{\ydiagram{4,4,4}}}+\frac{p_{\ydiagram{4,3,3}}}{p_{\ydiagram{3,3,3}}}
+\frac{p_{\ydiagram{3,2,2}}}{p_{\ydiagram{2,2,2}}}+\frac{p_{\ydiagram{2,1,1}}}{p_{\ydiagram{1,1,1}}}+\frac{p_{\ydiagram{1}}}{p_\emptyset},\]
where the $p_\lambda$ are Pl\"ucker coordinates for $\checkX=Gr_4(\C^7)$; see Section~\ref{s:notation}
for an explanation of the notation. 

As another example, if  $\X=Gr_2(\C^5)$, then the superpotential on $\opencheckX$ is 
\begin{equation}\label{G25-super}
W=\frac{p_{\ydiagram {3,1}}}{p_{\ydiagram {3}}}+
q\frac{p_{\ydiagram {2}}}{p_{\ydiagram {3,3}}}+\frac{p_{\ydiagram {3,2}}}{p_{\ydiagram {2,2}}}+
 \frac{p_{\ydiagram {2,1}}}{p_{\ydiagram {1,1}}}+
\frac{p_{\ydiagram 1}}{p_\emptyset}.
\end{equation}
The $n$ summands of the superpotential individually give rise to functions which in this case are 
\begin{equation}\label{eq:Wi}
W_1=
\frac{p_{\ydiagram {3,1}}}{p_{\ydiagram {3}}},
\quad
W_2=
\frac{p_{\ydiagram {2}}}{p_{\ydiagram {3,3}}},
\quad
W_3=
\frac{p_{\ydiagram {3,2}}}{p_{\ydiagram {2,2}}},
\quad
W_4=
\frac{p_{\ydiagram {2,1}}}{p_{\ydiagram {1,1}}},
\quad
W_5=
\frac{p_{\ydiagram 1}}{p_\emptyset}.
\end{equation}
We will often  use the normalisation $p_{\emptyset}=1$ so that the Pl\"ucker coordinates are
 actual coordinates on~$\opencheckX$.

\subsection{} Besides the network tori $\mathbb T_G$,
there is a different 
collection of open tori $\mathbb T^{\vee}_G$ in  
$\opencheckX$ indexed by plabic graphs $G$, each one corresponding to a maximal algebraically independent set of Pl\"ucker 
coordinates \cite{Postnikov, Scott}. We call these collections of Pl\"ucker coordinates the {\it Pl\"ucker clusters} of $\checkX$. 
By \cite{Scott} they are part of an $\mathcal A$-cluster structure on $\C[\opencheckX]$ in the sense of 
Fomin and Zelevinsky~\cite{ca1}. As before we also consider more general $\mathcal A$-cluster tori associated to quivers
 not necessarily coming from plabic graphs; we continue to denote them by  $\mathbb T^{\vee}_G$.
As part of the data such a torus has specific $\mathcal{A}$-cluster coordinates $\PCG$ indexed by the vertices 
$\mathcal{P}_G$ of the quiver, which are Pl\"ucker coordinates when the quiver comes from 
a plabic graph.
The data of the quiver together with the torus coordinates is 
called an \emph{$\mathcal A$-cluster seed} and denoted $\check\Sigma_G^{\mathcal{A}}$.  
We think of the $\mathcal A$-cluster coordinates as encoding an open embedding
\[
\Phi^{\vee}_G:(\C^*)^{\mathcal P_G}\overset\sim\To \mathbb T^{\vee}_G\subset \opencheckX.
\]

Given an $\mathcal A$-cluster torus $\mathbb T^\vee_G$, we may restrict $W$ and 
each $W_i$ to the torus $\mathbb T^\vee_G$. The ring of regular functions on $\mathbb T^\vee_G$ is just the Laurent polynomial ring in the coordinates $\PCG=\{p_\mu\mid \mu\in\mathcal P_G\}$ of the $\mathcal A$-cluster (and the restriction of $W$ lies in this ring tensored with $\C[q]$). 
From the $\mathcal A$-cluster seed and the superpotential together we thus obtain Laurent polynomials 
\[ \LW_i^G=W_i|_{\mathbb T_G^\vee},\quad i=1,\dotsc, n, \quad\text{ and }\quad \LW^G=\sum_i q^{\delta_{i,n-k}}\LW^G_i.
\]
To these Laurent polynomials, together with a choice of integers $r_1,\dotsc, r_n$, we may associate a (possibly empty or unbounded) intersection of half-spaces
 $\Q_G(r_1,\dotsc, r_n)$ by a tropicalisation construction, see \cref{sec:superpolytope}. 
We describe this construction by giving an example. 

Let $\X=Gr_2(\C^5)$, with superpotential given by
\eqref{G25-super}.
If we write $W$ and the $W_i$ from \eqref{eq:Wi}
in terms of the Pl\"ucker cluster indexed by %
$\mathcal P_{G}=\left\{{\ydiagram{1},\ydiagram{1,1},\ydiagram{2}, \ydiagram{2,2}, \ydiagram{3}, \ydiagram{3,3}}\right\}$, we get
 the Laurent polynomial
\begin{equation}\label{e:superIntro}
\LW^G=\frac{\pyoo}{\pyo}+
\frac{\pyzz }{\pyo \ \pyz}+ \frac{\pytt }{\pyz \ \pyt}+ q\frac{\pyz}{\pytt}+\frac{\pyt}{\pyz}+\frac{\pytt \ \pyo}{\pyz \ \pyzz}+\frac{\pyz}{\pyo} +
\frac{\pyzz}{\pyo \ \pyoo}+{\pyo}, 
\end{equation}
as well as 
\begin{equation}\label{e:Li}
\LW^G_1=\frac{\pyoo}{\pyo}+
\frac{\pyzz }{\pyo \ \pyz}+ \frac{\pytt }{\pyz \ \pyt}, \qquad
\LW^G_2= \frac{\pyz}{\pytt},\qquad
\LW^G_3=\frac{\pyt}{\pyz}+\frac{\pytt \ \pyo}{\pyz \ \pyzz}, \qquad
\LW^G_4=\frac{\pyz}{\pyo} +
\frac{\pyzz}{\pyo \ \pyoo},
\qquad
\LW^G_5 = {\pyo}.
\end{equation}

Each Laurent polynomial $\LW^G_i$ gives rise to a piecewise-linear function $\Trop(\LW^G_i):\R^{\mathcal P_G}\to \R$ obtained by replacing multiplication by addition, division by subtraction, and addition by $\min$. 
For any choice of $r_1,\dotsc, r_5\in\Z$ we then define
$\Q_{G}(r_1,\dotsc, r_5)\subset \R^{\mathcal P_{G}}$  by the following
explicit inequalities in terms of
variables  $d=(d_{\ydiagram{1}}, d_{\ydiagram{1,1}},d_{\ydiagram{2}}, d_{\ydiagram{2,2}}, d_{\ydiagram{3}}, d_{\ydiagram{3,3}})\in\R^{\mathcal P_{G}}$:
\begin{align*}\label{e:QDIntro}
&\Trop(\LW^G_1)(d)+r_1=\min\left(d_{\ydiagram{1,1}}-d_{\ydiagram{1}}\, , d_{\ydiagram{2,2}}-d_{\ydiagram{1}}-d_{\ydiagram{2}}\, , d_{\ydiagram{2,2}}-d_{\ydiagram{2}}-d_{\ydiagram{3}}\right)+r_1  \ge 0 ,
\\
&\Trop(\LW^G_2)(d)+r_2=d_{\ydiagram{2}}-d_{\ydiagram{3,3}}+r_2 \ge 0 ,
\\
&\Trop(\LW^G_3)(d)+r_3=\min\left( d_{\ydiagram{3}}-d_{\ydiagram{2}}\, , d_{\ydiagram{3,3}}+d_{\ydiagram{1}}-d_{\ydiagram{2}}-d_{\ydiagram{3,3}}\right)+r_3\ge 0 ,
\\
&\Trop(\LW^G_4)(d)+r_4=\min\left( d_{\ydiagram{2}}-d_{\ydiagram{1}}\, , d_{\ydiagram{2,2}}-d_{\ydiagram{1}}-d_{\ydiagram{1,1}}\right)+r_4\ge 0 ,
\\
&\Trop(\LW^G_5)(d)+r_5=d_{\ydiagram{1}}+r_5 \ge 0.
\end{align*}

There is an important
 special case where $r=r_{n-k}\ge 0$ and $r_i=0$ for all other $i$. (In the running example
$n=5$ and $k=3$, so  $r=r_2$.) In this case the polytope  defined by the construction is also denoted $\Q_G^r$. The polytope $\Q_G^r$ can be expressed directly in terms of the superpotential $\LW^G=W|_{\mathbb T_G^\vee\x\C^*}$ as
\begin{equation}\label{e:QrGIntro}
\Q_G^r=\{d\in\R^{\mathcal P_G}\mid \Trop(\LW^G) (d,r)\ge 0\},
\end{equation}
see \cref{def:Tropicalisation} for the notation.
When $r=1$, we refer to this polytope as the {\it superpotential polytope} 
$\Q_G^1 = \Q_G$ for the seed $\check\Sigma_G^{\mathcal A}$. 

\subsection{} We now put the two sides together to state the first main theorem. Recall the original Grassmannian $\X=Gr_{n-k}(\C^n)$ with its anti-canonical divisor $\Dac=D_1+\dotsc +D_n$, its $\mathcal X$-cluster seeds, and the definition of the Newton-Okounkov  body.  

\begin{theorem}\label{t:mainIntro} Suppose $D$ is an ample divisor in $\X$ of the form $D=r_1D_1+\dotsc+r_n D_n$ and $\Sigma^{\mathcal X}_G$ is an $\mathcal X$-cluster seed in $\Xcirc$. The associated Newton-Okounkov body $\Delta_G(D)$ is a rational polytope and we have
\[
\Delta_G(D)=\Q_G(r_1,\dotsc, r_n),
\]
where $\Q_G(r_1,\dotsc, r_n)$ is the polytope constructed from the superpotential $W:\opencheckX\x\C^*_q\to \C$ and the $\mathcal A$-cluster seed $\check\Sigma^{\mathcal A}_G$ of $\opencheckX$. 
\end{theorem}
 When 
$D=D_{n-k}$ we also denote $\Delta_G(D_{n-k})$ simply by $\Delta_G$. The above result implies
that 
 \begin{equation}\label{e:mainQG1}
\Delta_G=\Q_G, 
 \end{equation}
where $\Q_G$ is the superpotential polytope from \eqref{e:QrGIntro}. 
This key special case is proved first and is the content of \cref{thm:main}.
To prove \cref{thm:main}, we show that for a distinguished choice of $G$ (indexing the 
``rectangles" cluster), both $\Delta_G$ and $\Q_G$ coincide with a
Gelfand-Tsetlin polytope.  We then ``lift" $\Q_G$ to generalized Puiseux series
and show that when the seed $G$ changes via a mutation, $\Q_G$ is transformed
via a piecewise linear ``tropicalized mutation".   We also show that 
when we mutate $G$, $\Delta_G$ is transformed via the same tropicalized mutation:
our proof on this side uses deep properties of the 
\emph{theta basis} of \cite{GHKK}, including the Fock-Goncharov conjecture
that elements of the theta basis are \emph{pointed}, see \cref{th:pointed}. In the case where $\Q_G$ is an integral polytope we prove that $\Q_G=\NO_G$ without using \cite{GHKK}, see \cref{t:intcase}.

If we choose a network torus coming from a plabic graph $G$, 
then the associated Laurent expansion $\LW^G$ of $W$ can be read off from $G$ using a formula of Marsh and Scott \cite{MarshScott}. 
We thus obtain an explicit formula in terms of perfect matchings
for the inequalities defining the Newton-Okounkov  body, see \cref{s:clusterexpansion}.

It follows from our results that $\Delta_G$ is a rational polytope.  
In \cref{degeneration} we build on this fact to show that from each seed $\Sigma_G^{\mathcal X}$
we obtain a flat degeneration of $\X$ to the toric variety associated to the dual fan constructed from the polytope $\Delta_G$.
Note however that $\Delta_G$ 
 is not in general integral; of the 34 polytopes $\Delta_G$ associated to 
plabic graphs for $Gr_3(\C^6)$, precisely two are non-integral, see 
\cref{sec:Milena}.
In each of those cases, there is a unique non-integral vertex which 
corresponds to the twist of a Pl\"ucker coordinate.  Since the first version of this
paper appeared on the arXiv, the polytopes arising from 
$Gr_3(\C^6)$ have been
 studied in \cite{Hering}.

In \cref{s:generalD} we prove \cref{t:mainIntro} in  the general $D=\sum r_iD_i$ case  by relating $\Delta_G(D)$ to $\Delta_G(D_{n-k})$ and $\Q_G(r_1,\dotsc, r_n)$ to $\Q_G$ and deducing the general result
from \cref{thm:main}.

\subsection{}
Our second main result
concerns an explicit description of the lattice points of the Newton-Okounkov body $\Delta_G=\Delta_G(D_{n-k})$ 
when $G$ is a plabic graph. Recall that the homogeneous coordinate ring of $\X$ is generated by Pl\"ucker coordinates 
which are
naturally indexed by the set  $\mathcal P_{k,n}$ of Young diagrams fitting inside an $(n-k)\x k$ rectangle. We denote these Pl\"ucker coordinates by $P_{\lambda}$ with $\lambda\in\mathcal P_{k,n}$.
Note that the upper-case $P_\lambda$ (Pl\"ucker coordinate of $\X$) should not be confused with the lower-case $p_\lambda$ (Pl\"ucker coordinate of $\checkX$).
The largest of the Young diagrams in $\mathcal P_{k,n}$ is the entire $(n-k)\x k$ rectangle, and we denote
its corresponding Pl\"ucker coordinate by $P_\Max$. %
The set $
\{P_{\lambda}/P_{\Max}\mid \lambda\in\mathcal P_{k,n} \}$
is a natural basis for $H^0(\X, \mathcal O(D_{n-k}))$.

The following result
says that the valuations $\val_G(P_\lambda/P_{\Max})$ are precisely the $n\choose k$ lattice points of the Newton-Okounkov body $\Delta_G$, and gives an explicit formula for them.

\begin{theorem}[\cref{l:injection}]\label{t:ValuationsFormulaIntro} 
Let $G$ be any reduced plabic graph giving a network torus for $\openX$.  
Then the Newton-Okounkov  body $\Delta_G$
has ${n \choose k}$ lattice points 
$\{\val_G(P_{\lambda}/P_{\max}) \ \vert \ \lambda \in \mathcal{P}_{k,n}\}
 \subseteq \Z^{\mathcal{P}_G}$,
 with coordinates given by 
\begin{equation*}
\val_G(P_{\lambda}/P_{\max})_{\mu} = \maxdiag (\mu \setminus \lambda)
\end{equation*}
for any partition $\mu\in \mathcal{P}_G$. 
 Here $\maxdiag (\mu \setminus \lambda)$ denotes the maximal number of boxes in a slope $-1$ diagonal in the skew partition $\mu\setminus\lambda$, see \cref{def:maxdiag}. 
\end{theorem}

Note that the right hand side of the formula depends neither on the plabic graph $G$ nor on the Grassmannian, that is, on $k$ or $n$. We illustrate the function $\maxdiag$ with an example:
\[ 
\maxdiag\left(\begin{array}{c} {\ydiagram{7,7,4,4,3,1}} \underset{\backslash}{\qquad} {\ydiagram {6,5,2,2,2,2}}\\ \qquad   \end{array}\,\right)=\maxdiag\left(\begin{array}{c}{\ydiagram{6+1,5+2,2+2,2+2,2+1}}\end{array}\right)=2.
\]
Also note that
 if $\mu\subseteq \lambda$ then necessarily $\maxdiag(\mu\setminus\lambda)=0$, so the theorem implies that the $\mu$-coordinate of $\val_G(P_\lambda/P_\Max)$ vanishes. Indeed, if $\lambda=\Max$ then the formula says that all coordinates of the valuation vanish, which recovers the fact that the constant function $1$ has valuation $0$. 

Interestingly, the function
$\maxdiag (\mu \setminus \lambda)$ in \cref{t:ValuationsFormulaIntro} has an interpretation in 
quantum cohomology: by a result of Fulton and Woodward \cite{FW}, it is equal to the
smallest degree $d$ such that $q^d$ appears in the Schubert expansion of the product of two Schubert
classes $\sigma_{\mu}\star \sigma_{\lambda^{c}}$
in the quantum cohomology ring  $QH^*(\mathbb X)$. 
We also prove a parallel result in \cref{sec:highvaluation} which says that if we consider the 
\emph{highest-order-term} valuation $\val^G$ instead of the lowest-order-term valuation $\val_G$,
then $\val^G(P_{\lambda}/P_{\max})_{\mu}$  is equal to the largest degree $d$ such that 
$q^d$ appears in the Schubert expansion of 
 $\sigma_{\mu} \star \sigma_{\lambda^{c}}$. This degree was described in \cite{PostnikovDuke}, see also  \cite{Yong}.

While our proof of \cref{t:ValuationsFormulaIntro}
does not rely on Theorem~\ref{t:mainIntro}, both proofs  use the general philosophy of mirror symmetry. We think of the valuation $\val_G(P_\lambda/P_\Max)$ as an element of the character lattice of the $\mathcal{X}$-cluster 
network torus $\mathbb T_G$.  Then we reinterpret this {\it character lattice} as the {\it cocharacter lattice} of the dual torus $\mathbb T_G^\vee$. We consider the dual torus to be naturally an $\mathcal{A}$-cluster torus in a Langlands dual Grassmannian $\checkX$, using the cluster algebra structure of~\cite{ca1, Scott}. 
Then we show that $\val_G(P_\lambda/P_\Max)$ represents a \emph{tropical point} of $\checkX$ with regard to this cluster structure. The formula in Theorem~\ref{t:ValuationsFormulaIntro} is obtained by the explicit construction of an element of $\checkX(\R_{>0}((t)))$ which represents this tropical point.

\subsection{} %
 We note that  tropicalisation in the Langlands dual world is well-known to play a fundamental role in the parameterization of basis elements of representations of a reductive algebraic 
group $\mathcal{G}$; this goes back to Lusztig and his work on the canonical basis \cite{Lus:CanonBasis, Lus:Quantum}. 
The particular construction of the polytope $\Q_G $ we use here is related to the construction of Berenstein and Kazhdan in their theory of geometric crystals \cite{BK:GeometricCrystalsII}. The cluster charts we use are specific to Grassmannians, but we note that  there is an isomorphism, \cite[Theorem~4.9]{MarshRietsch}, between the superpotential $W:\openX\to \C$ and the  function used in \cite{BK:GeometricCrystalsII} in the maximal parabolic setting. The function from \cite{BK:GeometricCrystalsII}  also agrees with the Lie-theoretic superpotential associated to 
$\X=\mathcal{G}/\mathcal{P}$ in~\cite{Rietsch}. 

 On the Newton-Okounkov side, it is interesting to note the related work of Kaveh~\cite{Kaveh} in the case of the full flag variety $\mathcal G/\mathcal B$ which describes Newton-Okounkov convex bodies associated to particular highest-order term valuations on $\C(\mathcal G/\mathcal B)$ 
 and recovers string polytopes. Combining this result with Berenstein and Kazhdan's construction of string polytopes via geometric crystals provides a similar picture to ours in the  full flag variety case of two `dual' constructions of the same polytope, and may be interpreted as an instance of mirror symmetry.  However the proofs are very different and the representation theory arguments of \cite{Kaveh} do not extend to our setting. We note also a recent paper of Judd  \cite{Judd}, which adds detail to this picture in the case of  $SL_n/B$.

The connection between the lattice points of the tropicalized superpotential polytopes and the theta basis of the dual cluster algebra, which enters into our first main theorem, appears as an instance of the cluster duality conjectures between cluster $\mathcal X$-varieties and cluster $\mathcal A$-varieties developed by Fock and Goncharov~\cite{FG,FG1}. 
For Theorem~\ref{t:mainIntro} we make use of the deep properties of the theta
basis of Gross, Hacking, Keel and Kontsevich \cite{GHKK} for a cluster $\mathcal X$-variety, see Section~\ref{s:theta}. The duality theory of cluster algebras has also been explored and applied in other works such as  \cite{GoncharovShen,GoncharovShen2,Magee}.  

For a Grassmannian $Gr_2(\C^n)$, the plabic graphs are in bijection with triangulations of an $n$-gon, 
and in this case polytopes isomorphic  to ours were obtained earlier by Nohara and Ueda. These polytopes were shown in \cite{nohara_ueda} to be integral (unlike in the general case), and were used to construct toric degenerations of the Grassmannian $Gr_2(\C^n)$, see also \cite{Hering}.


\subsection{} This project originated out of the observation that Gelfand-Tsetlin polytopes appear to be naturally associated, but by very different constructions, both to the Grassmannian $\X$ and to its mirror, using a transcendence basis as input data. It also arose out of the wish to better understand the superpotential for Grassmannians from \cite{MarshRietsch}. As far as we know this is the first time these ideas from mirror symmetry have been brought to bear on the problem of constructing Newton-Okounkov  bodies.

Since the first version of this paper was posted to the arXiv in 2017, several other 
related works have appeared, including \cite{BFMC}, which discusses
a general framework for toric degenerations of cluster varieties,
and \cite{ShenWeng}, which discusses cyclic sieving and cluster duality
for Grassmannians.

\vskip .2cm

\noindent{\bf Acknowledgements:~}
The first author thanks M.~Kashiwara for drawing her attention to the theory of geometric crystals \cite{Kashiwara:PC}.   
The authors would also like to thank Man Wei Cheung, Sean Keel, Mark Gross,  Tim Magee, and Travis Mandel for helpful conversations about the theta basis.
They would also like to thank Mohammad Akhtar, Dave Anderson, Chris Fraser, Steven Karp, Ian Le,
Alex Postnikov, and Kazushi Ueda for useful discussions.
They are grateful to Peter Littelmann for helpful comments as well as Xin Fang and Ghislain Fourier.  
And they would like to thank Milena Hering and Martina Lanini  and collaborators Lara Bossinger, Xin Fang, and Ghislain Fourier  for 
bringing an important example to their attention. 
Finally, we would like to thank three very diligent referees (including the one who gave us $93$ itemized comments), whose feedback
has led to some substantial improvements in the exposition.
This material is based upon work supported by the Simons foundation,
a Rose-Hills Investigator award, as well as the National Science
Foundation under agreement No. DMS-1128155 and No. DMS-1600447.  
Any opinions, findings
and conclusions or recommendations expressed in this material are those of 
the authors and do not necessarily reflect the views of the National
Science Foundation.

\section{Notation for Grassmannians}\label{s:notation}
\color{black} 
\subsection{The Grassmannian 
$\mathbbX$}

Let $\mathbbX$ be the Grassmannian of $(n-k)$-planes in $\C^n$. We will denote its dimension  by $N=k(n-k)$. 
An element of $\mathbbX$
can be represented as the column-span of a full-rank $n\times (n-k)$ matrix modulo right
multiplication by nonsingular $(n-k)\times (n-k)$ matrices.  
Let $\binom{[n]}{n-k}$ be the set of all $(n-k)$-element subsets of 
$[n]:=\{1,\dots,n\}$.
For $J\in \binom{[n]}{n-k}$, let $P_J(A)$
denote the maximal minor of an $n\times (n-k)$ matrix 
$A$ located in the row set $J$.
The map $A\mapsto (P_J(A))$, where $J$ ranges over $\binom{[n]}{n-k}$,
induces the {\it Pl\"ucker embedding\/} $\mathbbX\hookrightarrow \mathbb{P}^{\binom{n}{n-k}-1}$, and the $P_J$ are called \emph{Pl\"ucker coordinates}.

We also think of $\mathbbX$ as a homogeneous space for the group $GL_n(\C)$, acting on the left.
 We fix the standard pinning of $GL_n(\C)$ consisting of upper and lower-triangular Borel subgroups $B_+, B_-$, maximal torus $T$ in the intersection, and simple root subgroups $x_i(t)$ and $y_i(t)$ given by exponentiating the standard upper and lower-triangular Chevalley generators $e_i, f_i$ with $i=1,\dotsc, n-1$. We denote the Lie algebra of $T$ by $\mathfrak h$ and we have fundamental weights $\omega_i\in \mathfrak h^*$ as well as simple roots $\alpha_i\in \mathfrak h^*$.  
 
 For $\mathbbX=Gr_{n-k}(\C^n)$ there is a natural identification between $H^2(\mathbbX,\C)$ and the subspace of $\mathfrak h^*$  spanned by $\omega_{n-k}$, under which $\omega_{n-k}$ is identified with the 
 hyperplane class of $\X$ in the Pl\"ucker embedding.

\subsection{The mirror dual Grassmannian~$\checkX$}
\label{s:mirrordualGr}
Let $(\C^n)^*$ denote a vector space of row vectors.
We then let
 $\checkX=Gr_{k}((\C^n)^*)$ 
be the `mirror dual' Grassmannian of $k$-planes in the vector space $(\C^n)^*$. 
 An element of $\checkX$ can be represented as the row-span of a 
 full-rank $k\x n$ matrix $M$. This new Grassmannian is considered to be a homogeneous space via a {\it right} action by a rank $n$ general linear group. To be precise, the group acting on $\checkX$ is the Langlands dual group to the general linear group acting on $\X$, and we denote it $GL_n^\vee(\C)$ to keep track of this duality.\footnote{ The Langlands duality we mean here is a generalisation to complex reductive algebraic groups of duality of tori, in which two Langlands dual groups have dual maximal tori and roots and coroots are interchanged.}  For this group we use 
all the same notations as introduced in the preceding paragraph for $GL_n(\C)$, but with an added superscript $^{\vee}$. To illustrate the duality, the dominant character 
$r\omega_{n-k}$ of $GL_n(\C)$ that corresponded to a line bundle on $\mathbbX$ can be considered 
as representing a one-parameter subgroup of the maximal torus $T^\vee$ of $GL_n^\vee(\C)$. It determines an element of $T^\vee(\C((t)))$ (where $t$ is the parameter), and this element maps to
$t^r$ under $\alpha_{n-k}^\vee$. 

Note that the Pl\"ucker coordinates of $\checkX$ are naturally parameterized by $\binom{[n]}{k}$; for every $k$-subset $I$ in $[n]$ the Pl\"ucker coordinate $p_I$ is associated to the $k\x k$ minor of $M$ with column set given by~$I$.

\color{black}
\subsection{Young diagrams}\label{Young}
It is convenient to index Pl\"ucker coordinates of both 
$\mathbbX$ and $\checkX$ using Young diagrams. 
Recall that $\Shkn$ denotes the set of Young diagrams fitting in an 
$(n-k)\x k $ rectangle.   There is a natural bijection between 
$\Shkn$ and 
${[n] \choose n-k}$, defined as follows.  Let $\mu$ be an element of 
$\Shkn$, justified so that its top-left corner coincides with 
the top-left corner of the 
$(n-k) \times k$ rectangle.  The south-east border of $\mu$ is then
cut out by a path from the 
northeast to southwest corner of the rectangle, which consists of $k$
west steps and $(n-k)$ south steps.  After labeling the $n$ steps by 
the numbers $\{1,\dots,n\}$, we map $\mu$ to the labels of the 
south steps.  This gives a bijection 
from $\Shkn$ to
${[n] \choose n-k}$.  If we use the labels of the west steps instead,
we get a bijection
from $\Shkn$ to
${[n] \choose k}$.
Therefore the elements of $\Shkn$ 
index the Pl\"ucker coordinates $P_\mu$ on $\mathbbX$ and 
simultaneously the Pl\"ucker coordinates on $\checkX$, which we denote
by $p_\mu$.

For $0 \leq i \leq n-1$, set $J_i:=[i+1,i+k]$, 
interpreted cyclically as a subset of  $[1,n]$. 
We let $\mu_i$ denote the Young diagram with west steps given by $J_i$. 
Then when $i\le n-k$, we have that $\mu_i$ is the rectangular $i \x k$ Young diagram, and when $i\ge n-k$, it is the rectangular $(n-k)\x (n-i)$ 
Young diagram.  
Note that  $\mu_{n-k}$ is the whole $(n-k) \times k$ rectangle, so we also 
write $\Max:=\mu_{n-k}$.

\subsection{The open positroid strata $\openX$ and $\checkX^{\circ}$}\label{s:positroid}

We use the special Young diagrams $\mu_i$ 
from \cref{Young} to 
 define a distinguished 
 anticanonical divisor $D_{\ac}=\bigcup_{i=1}^n D_i$ where $D_i=\{P_{\mu_i}=0\}$
in $\mathbb X$, and similarly 
  an 
 anticanonical divisor $\check D_{\ac}=\bigcup_{i=1}^n\{p_{\mu_i}=0\}$
in $\checkX$.

\begin{definition}\label{d:openX}
We define $\Xcirc$ to be the complement of the divisor 
$D_{\ac}=\bigcup_{i=1}^n \{P_{\mu_i}=0\}$, 
\begin{equation*}
\Xcirc:=\mathbb X \setminus  D_{\ac} =\{x\in\mathbb X  \ | \ P_{\mu_i}(x)\ne 0 \ \forall i\in[n] \}.
\end{equation*}

	And we define $\checkX^{\circ}$ to be the complement of the divisor 
$\check D_{\ac}=\bigcup_{i=1}^n \{p_{\mu_i}=0\}$, 
\begin{equation*}
	\checkX^{\circ}:=\checkX\setminus \check D_{\ac} =\{x\in\checkX  \ | \ p_{\mu_i}(x)\ne 0 \ \forall i\in[n] \}.
\end{equation*}
  \end{definition}
These varieties come up in \cite{GSV:Grass} and \cite{KLS:positroid}.

\section{Plabic graphs for Grassmannians}\label{sec:moves}

In this section 
we review Postnikov's notion of \emph{plabic graphs} \cite{Postnikov}, which 
we will then use to define network charts and cluster charts for the Grassmannian.

\begin{definition}
A {\it plabic (or planar bicolored) graph\/}
is an undirected graph $G$ drawn inside a disk
(considered modulo homotopy)
with $n$ {\it boundary vertices\/} on the boundary of the disk,
labeled $1,\dots,n$ in clockwise order, as well as some
colored {\it internal vertices\/}.
These  internal vertices
are strictly inside the disk and are
colored in black and white. 
We will always assume that $G$ is bipartite, and that 
each boundary vertex $i$ is adjacent to one white vertex and no other vertices.
\end{definition}

See Figure \ref{G25} for an example of a plabic graph.
\begin{figure}[h]
\centering
\includegraphics[height=1in]{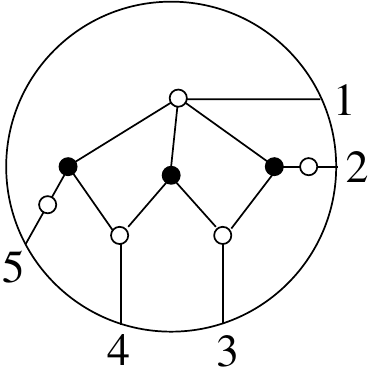}
\caption{A plabic graph}
\label{G25}
\end{figure}

There is a natural set of local transformations (moves) of plabic graphs, which we now describe.
Note that we will always assume that a plabic graph $G$ has no isolated 
components (i.e. every connected component contains at least
one boundary vertex).  We will also assume that $G$ is \emph{leafless}, 
i.e.\ if $G$ has an 
internal vertex of degree $1$, then that vertex must be adjacent to a boundary
vertex.

(M1) SQUARE MOVE (Urban renewal).  If a plabic graph has a square formed by
four trivalent vertices whose colors alternate,
then we can switch the
colors of these four vertices (and add some degree $2$ vertices to preserve
the bipartiteness of the graph).

(M2) CONTRACTING/EXPANDING A VERTEX.
Any degree $2$ internal vertex not adjacent to the boundary can be deleted,
and the two adjacent vertices merged.
This operation can also be reversed.  Note that this operation can always be used
to change an arbitrary
 square face of $G$ into a square face whose four vertices are all trivalent.

(M3) MIDDLE VERTEX INSERTION/REMOVAL.
We can always remove or add degree $2$ vertices at will, subject to the 
  condition that the graph remains bipartite.

See \cref{M1} for depictions of these three moves.

\begin{figure}[h]
\centering
\includegraphics[height=.5in]{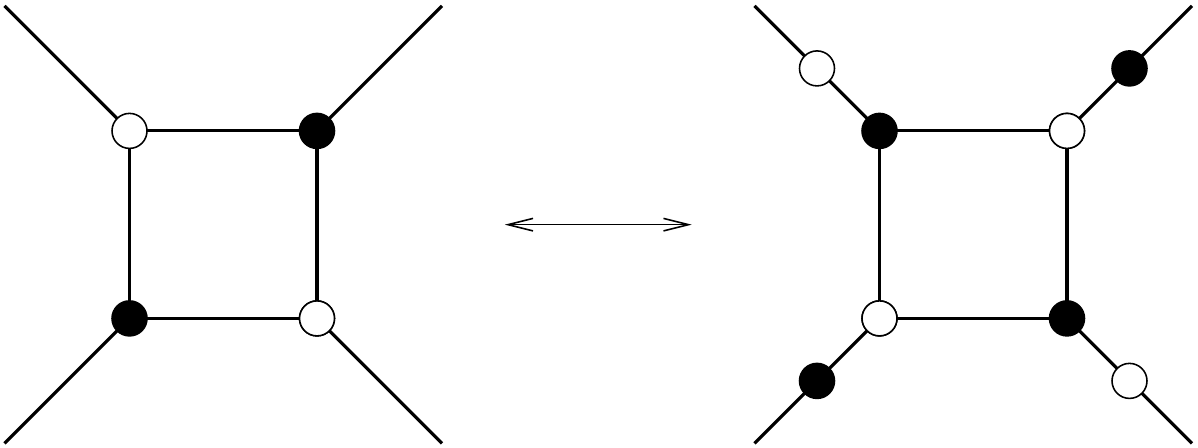}
\hspace{.5in}
\raisebox{6pt}{\includegraphics[height=.4in]{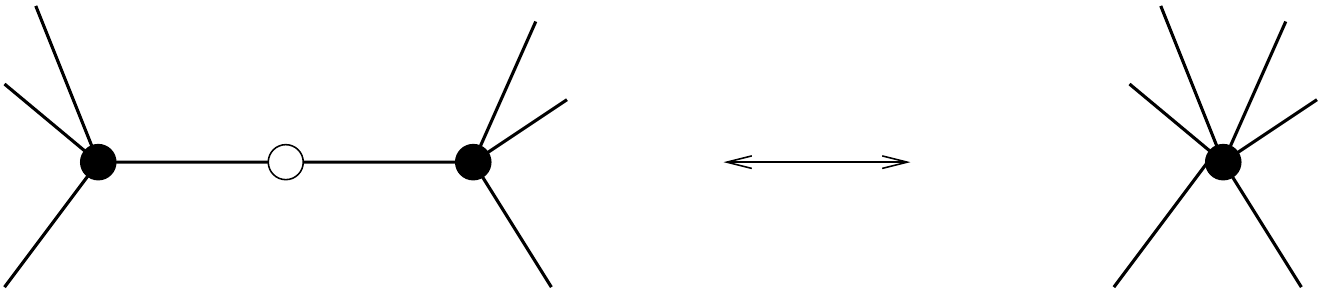}}
\hspace{.5in}
\raisebox{16pt}{\includegraphics[height=.07in]{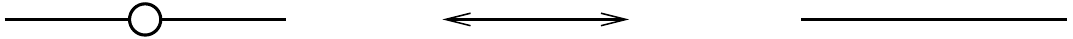}}
\caption{%
A square move, an edge 
contraction/expansion, and a vertex insertion/removal.}
\label{M1}
\end{figure}

(R1) PARALLEL EDGE REDUCTION.  If a plabic graph contains
two trivalent vertices of different colors connected
by a pair of parallel edges, then we can remove these
vertices and edges, and glue the remaining pair of edges together.

\begin{figure}[h]
\centering
\includegraphics[height=.25in]{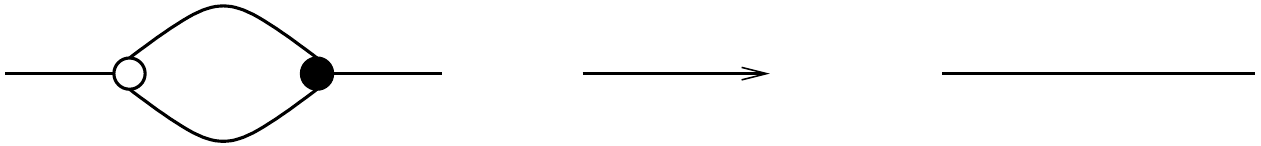}
\caption{Parallel edge reduction}
\label{R1}
\end{figure}

\begin{definition}
Two plabic graphs are called \emph{move-equivalent} if they can be obtained
from each other by moves (M1)-(M3).  The \emph{move-equivalence class}
of a given plabic graph $G$ is the set of all plabic graphs which are move-equivalent
to $G$.
A leafless plabic graph without isolated components
is called \emph{reduced} if there is no graph in its move-equivalence
class to which we can apply (R1).
\end{definition}

\begin{definition}\label{def:rules}
Let $G$ be a reduced plabic graph 
with boundary vertices $1,\dots, n$.
The \emph{trip} $T_i$ from $i$ is the path
obtained by starting from $i$ and traveling along
edges of $G$ according to the rule that each time we reach an internal black vertex we turn
(maximally) right, and each time we reach an internal white vertex we turn (maximally) left.
This trip ends at some boundary vertex ${\pi(i)}$.
In this way we associate a \emph{trip permutation}
$\pi_G=(\pi(1),\dots,\pi(n))$ to each reduced plabic graph $G$,
and we say that $G$ has \emph{type $\pi_G$}.
	(Because one can reverse each trip, it is clear that $\pi_G$ is a permutation.)
\end{definition}

As an example, the trip permutation associated to the
reduced plabic graph in Figure \ref{G25} is $(3,4,5,1,2)$.

\begin{remark}\label{rem:moves}
Let $\pi_{k,n} = 
(n-k+1, n-k+2,\dots, n, 1, 2, \dots, n-k)$.  
In this paper we will 
be particularly concerned with reduced plabic graphs whose trip permutation
is $\pi_{k,n}$.  Note that the trip permutation of a plabic graph is preserved 
by the local moves (M1)-(M3), but not by (R1). For reduced plabic graphs the converse holds, namely 
it follows from \cite[Theorem 13.4]{Postnikov} 
that any two reduced plabic graphs with trip permutation $\pi_{k,n}$ are 
move-equivalent.
\end{remark}

Next we use the trips to label each face of a reduced plabic graph 
by a partition.

\begin{definition}\label{def:facelabels}
Let $G$ be a reduced plabic graph of type $\pi_{k,n}$.  Note that 
each trip $T_i$ partitions the disk containing $G$ into two parts:
the part on the left of $T_i$, and the part on the right.  Place
an $i$ in each face of $G$ which is to the left of $T_i$.  After doing
this for all $1 \leq i \leq n$, each face will contain an $(n-k)$-element
subset of $\{1,2,\dots,n\}$.  Finally we realise that
$(n-k)$-element subset as the south steps of a corresponding Young diagram 
in $\Shkn$. We let $\widetilde{\mathcal P}_G$ denote the set of Young diagrams inside $\Shkn$ associated in this way to $G$.   Note that the boundary regions of 
	$\widetilde{\mathcal P}_G$  are labeled by the Young diagrams
	 $\mu_i$ for $0 \leq i \leq n-1$ (see \cref{Young}); in particular $\widetilde{\mathcal P}_G$  contains $\Max$ and $\emptyset$.  We set 
${\mathcal P}_G:=\widetilde{\mathcal P}_G\setminus\{\emptyset\}$.
Each reduced plabic graph $G$ of type $\pi_{k,n}$ will have 
	precisely $N+1$ faces, where $N=k(n-k)$ \cite[Theorem 12.7]{Postnikov}.
\end{definition}

The left of Figure \ref{G25-partitions} shows the labeling
of each face
of our running example by a Young diagram in $\Shkn$
(here $k=3$ and $n=5$).
\begin{figure}[h]
\centering
\includegraphics[height=1in]{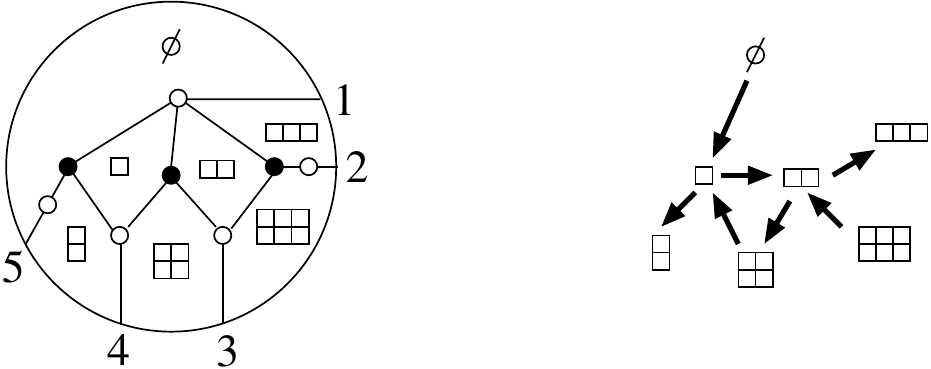}
\caption{A plabic graph $G$ with trip permutation $\pi_{3,5}$, with 
faces labeled by Young diagrams in $\Sh_{3,5}$, and the corresponding quiver $Q(G)$.  Here $\mathcal P_G=\left\{\ydiagram{3,3},\ydiagram{2,2},\ydiagram{1,1},\ydiagram{3},\ydiagram{2},\ydiagram{1}\right\}$.}
\label{G25-partitions}
\end{figure}

We next describe quivers and quiver mutation, and how they relate to moves on plabic graphs.
Quiver mutation was first defined by Fomin and Zelevinsky \cite{ca1}
 in order to define cluster algebras.

\begin{definition}[Quiver]\label{quiver}
A \emph{quiver} $Q$ is a directed graph; we will assume that $Q$ has no 
loops or $2$-cycles.
If there are $i$ arrows from vertex $\lambda$ to $\mu$, then 
we will set $b_{\lambda \mu} = i$ and $b_{\mu \lambda} = -i$.
Each vertex is designated either  \emph{mutable} or \emph{frozen}.
The skew-symmetric matrix $B = (b_{\lambda \mu})$ is called the \emph{exchange matrix} of $Q$.
\end{definition}

\begin{definition}[Quiver Mutation]
Let $\lambda$ be a mutable vertex of quiver $Q$.  The quiver mutation 
$\Mut_\lambda$ transforms $Q$ into a new quiver $Q' = \Mut_\lambda(Q)$ via a sequence of three steps:
\begin{enumerate}
\item For each oriented two path $\mu \to \lambda \to \nu$, add a new arrow $\mu \to \nu$
(unless $\mu$ and $\nu$ are both frozen, in which case do nothing).
\item Reverse the direction of all arrows incident to the vertex $\lambda$.
\item Repeatedly remove oriented $2$-cycles until unable to do so.
\end{enumerate}
	If $B$ is the exchange matrix of $Q$, then we let $\Mut_{\lambda}(B)$ denote the 
	exchange matrix of $\Mut_{\lambda}(Q)$.

\end{definition}

We say that two quivers $Q$ and $Q'$ are \emph{mutation equivalent} if $Q$
can be transformed into a quiver isomorphic to $Q'$ by a sequence of mutations.

\begin{definition}
Let $G$ be a reduced plabic graph.  We associate a quiver $Q(G)$ as follows.  The vertices of 
$Q(G)$ are labeled by the faces of $G$.  We say that a vertex of $Q(G)$ is \emph{frozen}
if the 
corresponding face is incident to the boundary of the disk, and is \emph{mutable} otherwise.
For each edge $e$ in $G$ which separates two faces, at least one of which is mutable, 
we introduce an arrow connecting the faces;
 this arrow is oriented so that it ``sees the white endpoint of $e$ to the left and the 
black endpoint to the right'' as it crosses over $e$.  We then remove oriented $2$-cycles
from the resulting quiver, one by one, to get $Q(G)$. See \cref{G25-partitions}.
\end{definition}

The following lemma is straightforward, and is implicit in \cite{Scott}.

\begin{lemma}\label{lem:mutG}
If $G$ and $G'$ are related via a square move at a face,
then $Q(G)$ and $Q(G')$ are related via mutation at the corresponding vertex.
\end{lemma}

Because of \cref{lem:mutG}, we will subsequently refer to 
``mutating" at a nonboundary face of $G$, meaning that we mutate
at the corresponding vertex of quiver $Q(G)$.
Note that in general the quiver $Q(G)$ admits mutations at vertices
which do not correspond to moves of plabic graphs.  For example, $G$ might have a 
hexagonal face, all of whose vertices are trivalent; 
in that case, $Q(G)$ admits a mutation at the corresponding
vertex, but there is no move of plabic graphs which corresponds 
to this mutation.  

In \cref{sec:cluster} and \cref{sec:poschart},
we will explain how to associate to each plabic graph $G$
a \emph{network chart} and a \emph{cluster chart} in $\Xcirc$, and similarly in $\opencheckX$.

\section{The rectangles plabic graph}\label{sec:Grectangles}
We define a particular reduced plabic graph $G_{k,n}^{\rect}$ with trip permutation
$\pi_{k,n}$ which will play a central role in our proofs.  This is a reduced plabic graph whose internal faces are 
arranged into an $(n-k) \times k$ grid pattern, as shown in Figure~\ref{G-rectangles}.  
(It is easy to check that the plabic graph $G_{k,n}^{\rect}$ is reduced, 
using e.g. 
\cite[Theorem 10.5]{KW}.)
When one uses Definition~\ref{def:facelabels} to label
faces by Young diagrams, one obtains the labeling of faces by rectangles
which is shown in the figure.  The generalization of this figure for 
arbitrary $k$ and $n$ is straightforward.
Note that the plabic graph from \cref{G25-partitions}
is $G_{3,5}^{\rect}.$

\begin{figure}[h]
\centering
\includegraphics[height=1.8in]{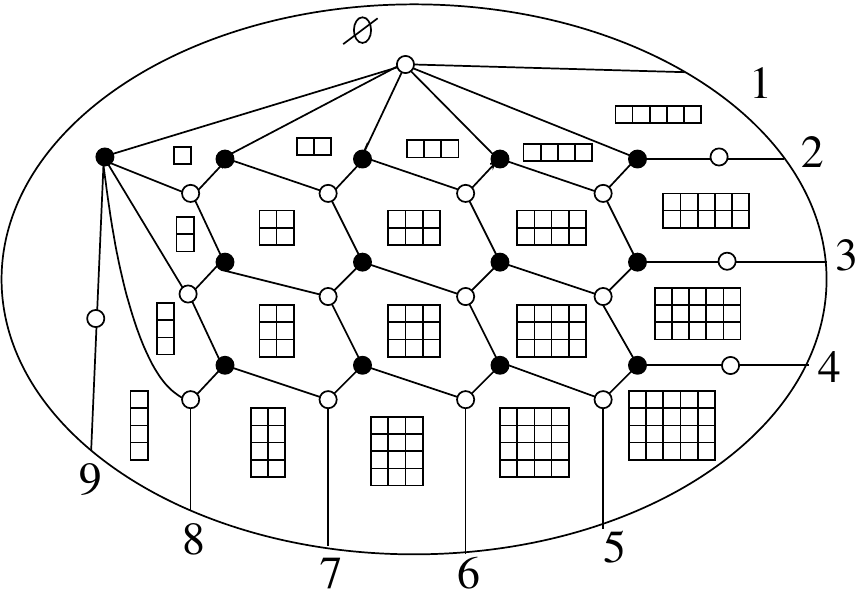}
\caption{The plabic graph $G_{5,9}^{\rect}$ with trip permutation $\pi_{5,9}$, with 
faces labeled by  $\Sh_{5,9}$.}
\label{G-rectangles}
\end{figure}

\section{Cluster charts 
from plabic graphs}\label{sec:cluster}

In this section we
fix a reduced plabic graph $G$ of type $\pi_{k,n}$ and
use it to construct a cluster chart 
for each of the open positroid varieties $\openX$ and $\checkX^\circ$ from Definition~\ref{d:openX}. Our exposition will for the most part focus on $\checkX^\circ$. References for this construction are~\cite{Scott, Postnikov}, see also \cite[Section~7]{MarshRietsch}.

Recall from Definition~\ref{def:facelabels}
that we have a labeling of each face of $G$ by some Young diagram 
in $\widetilde{\mathcal P}_G \subset \Shkn$.  
We now interpret each 
Young diagram in $\Shkn$ as a $k$-element subset of 
$\{1,2,\dots, n\}$ via its west steps, see Section~\ref{s:notation}. It follows from 
\cite{Scott}  that the collection 
\begin{equation}
\wPC(G):= \{p_{\mu} \ \vert \ \mu \in \widetilde{\mathcal{P}}_G\} 
\end{equation} of 
Pl\"ucker coordinates indexed by the faces of $G$
is a \emph{cluster} for the \emph{cluster algebra}  
associated to the 
homogeneous coordinate ring of $\checkX$.  
In particular, these Pl\"ucker coordinates are called \emph{cluster variables} and are
algebraically independent; moreover,  \emph{any} 
Pl\"ucker coordinate for $\checkX$ can be written as a 
positive Laurent polynomial in the variables from ${\wPC}(G)$.

Among the elements of $\wPC(G)$
there are always $n$ Pl\"ucker coordinates 
$\{p_{\mu_i} \ \vert \ 0 \leq i \leq n-1\}$, called 
\emph{frozen variables}.  They are present in each $\wPC(G)$
because each reduced plabic graph of type $\pi_{k,n}$ has $n$ boundary
regions which are labeled by the Young diagrams $\mu_i$ defined in \cref{Young}.

Let 
\[
\PCG:=\left\{\frac{p_{\mu}}{p_{\emptyset}}\ |\ p_\mu\in\widetilde{\PC}(G)\setminus\{p_{\emptyset}\}\right\}\subset \C(\mathbbX).
\]
If we choose  the normalization of 
Pl\"ucker coordinates on $\opencheckX$ such that $p_{\emptyset} = p_{\{1,\dots,k\}} =1$, 
we get a map
\begin{equation}
\label{eq:clusterchart}
\Phi_G^{\vee}=\Phi_{G, \mathcal A}^{\vee}: (\C^*)^{{\mathcal P}_G}\to 
	\opencheckX 
	\subset \checkX
\end{equation}
which we call a \emph{cluster chart} for $\opencheckX$, which satisfies $p_{\nu}(\Phi^\vee_G((t_\mu)_{\mu}))=t_\nu$ for $\nu\in{\mathcal P}_G$. 
Here $\mathcal P_G$ is as in Definition~\ref{def:facelabels}. 
When it is clear that we are setting $p_\emptyset=1$ we may write 
\begin{equation}\label{e:PCG}
\PCG:=\left\{p_{\mu}\ |\ \mu\in{\mathcal P}_G\right\}.
\end{equation}

\begin{definition}[Cluster torus $\mathbb T^\vee_G$]  \label{rem2:convention}
Define the open dense torus $\mathbb T^\vee_G $ in $\opencheckX$ as the image  of the cluster chart $\Phi^\vee_G$,
\[
\mathbb T^\vee_G
:=\Phi_G^\vee((\C^*)^{\mathcal{P}_G})=\{x\in\checkX\mid \p_{\mu}(x)\ne 0 \text { for all $\mu\in \mathcal P_G$}\}.
\] 
We call $\mathbb T^\vee_G$ the {\it cluster torus} in $\opencheckX$ associated to $G$. 
\end{definition}

\begin{defn}[Positive transcendence bases] \label{d:TranscBasis}
We say that a transcendence basis $\mathcal T$ for the field of rational functions on a Grassmannian 
is \emph{positive}
if each Pl\"ucker coordinate 
is a rational function in the elements of $\mathcal T$
with coefficients which are all nonnegative.  
\end{defn}

\begin{remark} The $\p_\mu\in\PCG$ restrict to coordinates on the open torus $\mathbb T^\vee_G$ in $\checkX$.
Therefore we can think of $\PCG$ as a transcendence basis of $\C(\checkX)$.
Moreover by iterating the Pl\"ucker relations, we can express any Pl\"ucker coordinate as a 
rational function in the elements of $\PCG$ with coefficients which are all nonnegative,
so $\PCG$ is a positive transcendence basis.
\end{remark}

\begin{example}
We continue our example from Figure \ref{G25-partitions}.
The Pl\"ucker coordinates labeling the faces of $G$ 
are
$\wPC(G) = 
\{p_{\{1,2,3\}}, p_{\{1,2,4\}}, p_{\{1,3,4\}}, p_{\{2,3,4\}},
p_{\{1,2,5\}}, p_{\{1,4,5\}}, p_{\{3,4,5\}}\}.$
\end{example} 
We next describe cluster $\mathcal{A}$-mutation, and how it relates to the clusters associated to 
plabic graphs $G$.

\begin{definition}\label{Aseed}
Let $Q$ be a quiver with vertices $V$ and associated exchange matrix $B$.
We associate a \emph{cluster variable} $a_{\mu}$ to each vertex $\mu \in V$.
If $\lambda$ is a mutable vertex of $Q$,
then we define a new set of 
variables  $\MutVar_{\lambda}^{\mathcal{A}}(\{a_{\mu}\}) := \{a'_{\mu}\}$  where
$a'_{\mu} = a_{\mu}$ if $\mu \neq \lambda$, and otherwise, $a'_{\lambda}$ is determined
by the equation
\begin{equation}\label{e:AclusterMut}
a_{\lambda} a'_{\lambda} = 
\prod_{b_{\mu \lambda} > 0} a_{\mu}^{b_{\mu \lambda}} + 
\prod_{b_{\mu \lambda} < 0} a_{\mu}^{-b_{\mu \lambda}}.
\end{equation}
We say that 
$(\Mut_{\lambda}(Q), \{a'_{\mu}\})$ is obtained from 
 $(Q, \{a_{\mu}\})$ by \emph{$\mathcal{A}$-seed mutation}
 in direction $\lambda$, and we refer to the ordered pairs 
$(\Mut_{\lambda}(Q), \{a'_{\mu}\})$  and
 $(Q, \{a_{\mu}\})$ as \emph{labeled $\mathcal{A}$-seeds}.
	We say that two labeled $\mathcal{A}$-seeds are 
	$\mathcal{A}$-mutation equivalent if one can be obtained
	from the other by a sequence of $\mathcal{A}$-seed mutations.
\end{definition}

Using the terminology of \cref{Aseed}, each reduced plabic graph $G$ gives rise to a 
labeled $\mathcal{A}$-seed $(Q(G), \wPC(G))$.
\cref{lem:Amove} below, which is easy to check, shows that our labeling of faces of each 
plabic graph by a Pl\"ucker coordinate 
is compatible with the $\mathcal{A}$-mutation.
More specifically, performing a square move 
on a plabic graph corresponds to a three-term Pl\"ucker relation.
Therefore  whenever two plabic graphs are connected by moves, the corresponding
$\mathcal{A}$-seeds are $\mathcal{A}$-mutation equivalent.

\begin{lemma}\label{lem:Amove}
Let $G$ be a reduced plabic graph with cluster 
variables $\wPC(G) := \{p_{\mu}\ |\ \mu\in \widetilde{\mathcal P}_G \}$,
and let $\nu_1$ be a square face of $G$ formed by four trivalent vertices, see 
\cref{labeled-square}.
Let $G'$ be obtained from $G$ by performing a square move at the face $\nu_1$,
and $\wPC(G')$ be the corresponding cluster variables.  
Then 
$\wPC(G') = 
 \MutVar_{\nu_1}^{\mathcal{A}}(\{p_{\mu}\})$.  In particular, 
the Pl\"ucker coordinates labeling the faces of $G$ and $G'$ satisfy
the \emph{three-term Pl\"ucker relation} $$p_{\nu_1} p_{\nu'_1} = 
p_{\nu_2} p_{\nu_4} + p_{\nu_3} p_{\nu_5}.$$
\end{lemma}

\begin{figure}[h]
\centering
\includegraphics[height=.65in]{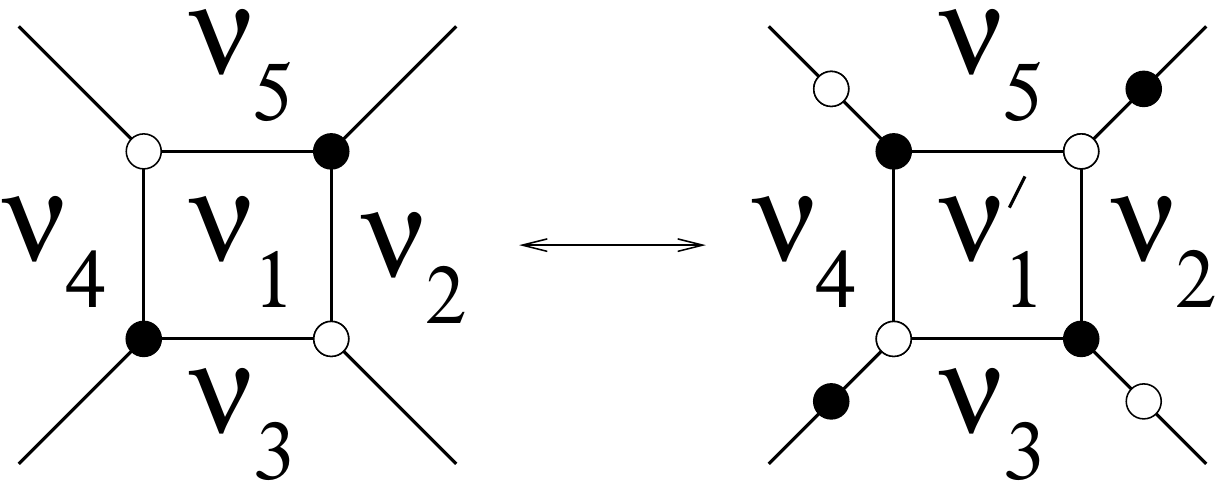}
\caption{}
\label{labeled-square}
\end{figure}

\begin{remark}\label{rem:Acluster}
By \cref{lem:Amove} and \cref{rem:moves}, all $\mathcal{A}$-seeds coming from plabic graphs $G$ of 
type $\pi_{k,n}$ are $\mathcal{A}$-mutation equivalent.
We can mutate the quiver $Q(G)$ at a vertex which does not correspond to a square face of $G$; 
 however, this will 
lead to quivers that no longer correspond to plabic graphs.
Nevertheless, 
	we can consider an arbitrary labeled $\mathcal{A}$-seed
	$(Q, \{a_{\mu}\})$ which is $\mathcal{A}$-mutation equivalent 
	to an $\mathcal{A}$-seed coming from a reduced plabic graph
 of type $\pi_{k,n}$;
  we say that 
	$(Q, \{a_{\mu}\})$ also has type $\pi_{k,n}$.
In this 
case we still have a cluster chart %
for $\opencheckX$ which
is obtained from the cluster chart $\Phi_G^{\vee}$
of \cref{eq:clusterchart}
	 by composing the $\mathcal{A}$-seed 
	mutations of 
\cref{e:AclusterMut}, and it will have a corresponding cluster torus.
Abusing notation, we will continue to index such $\mathcal{A}$-seeds, cluster charts, and cluster tori by $G$ (rather than 
$(Q, \{a_{\mu}\})$, but 
will take care to indicate when we are working with an arbitrary $\mathcal{A}$-seed rather than one coming 
from a plabic graph.
	
	We will sometimes use the notation $\Phi^{\vee}_{G,\mathcal{A}}$ for $\Phi^{\vee}_G$
	 and $\mathbb{T}_{G,\mathcal{A}}$ for $\mathbb{T}_G$
	to emphasize the $\mathcal{A}$-cluster structure.
	Note that the formulas 
\eqref{e:AclusterMut}
	naturally define a birational map
	$\mathcal{M}_{\mathcal{A},\lambda}: \mathbb{T}_{G,\mathcal{A}} \dashrightarrow \mathbb{T}_{G',\mathcal{A}}$.
\end{remark}

\begin{rem}[The case of $\openX$] The plabic graph $G$ which determines a seed of an $\mathcal A$-cluster structure on $\opencheckX$ also determines a seed of an $\mathcal A$-cluster structure on $\openX$. Namely we set 
\[
\PCGA=\left\{\frac{P_\mu}{P_{\Max}}\mid \mu\in\widetilde{\mathcal P}_{G}\setminus\{\Max\}\right\}.
\]
If we choose  the normalization of 
Pl\"ucker coordinates on $\openX$ such that $P_{\emptyset} = P_{\{1,\dots,n-k\}} =1$, 
we get a map
\begin{equation}
\label{eq:clusterchart2}
\Phi_G=\Phi_{G, \mathcal A}: (\C^*)^{{\mathcal P}_G}\to 
	\openX 
	\subset \X
\end{equation}
which we call a \emph{cluster chart} for $\openX$, which satisfies $P_{\nu}(\Phi_G((t_\mu)_{\mu}))=t_\nu$ for $\nu\in{\mathcal P}_G$. 

Again quiver mutation in general gives rise to many more seeds than these. But these seeds still correspond to torus charts in $\openX$ and we use the same notation  $\Phi_{G,\mathcal A}$ also for these more general charts.
\end{rem}

\section{Network charts %
from plabic graphs}
\label{sec:poschart}

In this section we will explain how to use a 
reduced plabic graph $G$ of type $\pi_{k,n}$ to 
construct   
a network chart for $\mathbbX^\circ$, the open positroid variety in $ \X=Gr_{n-k}(\C^n)$.
Network charts were originally 
introduced in \cite{Postnikov, Talaska} as a way to 
parameterize the \emph{positive part} of the Grassmannian. 
There is a notion of mutation for network charts, which was described in the
Grassmannian setting by Postnikov \cite[Section 12]{Postnikov}.
More generally, the notion of mutation can be defined for arbitrary quivers;
it is called \emph{mutation of $y$-patterns} in \cite[(2.3)]{ca4}
and cluster $\mathcal{X}$-%
mutation by Fock and Goncharov \cite[Equation 13]{FG}.
In this article we will not restrict ourselves to network charts from plabic graphs,
but will consider more general 
network charts associated to quivers $Q$ 
mutation equivalent to $Q(G)$, see Section~\ref{s:twist}.

\begin{defn}\label{def:posGrass}
The 
\emph{totally positive part}  
$\mathbbX(\mathbf \R_{>0})$ 
of the  Grassmannian $\mathbbX$ is the 
subset of the real Grassmannian $Gr_{n-k}(\R^n)$  consisting of the elements for which all Pl\"ucker coordinates are in $\R_{>0}$. 
\end{defn}
This definition is equivalent to Lusztig's  original definition \cite{Lusztig3} of the totally positive part of a generalized partial flag variety $G/P$  applied in the Grassmannian case.
(One proof of the equivalence of definitions comes from 
 \cite{TalaskaWilliams}, which related 
the Marsh-Rietsch parameterizations of cells \cite{MR}
of Lusztig's totally non-negative Grassmannian  to the 
parameterizations of cells coming from network charts.)

Network charts are defined using \emph{perfect orientations}
 and \emph{flows} in plabic graphs.
\begin{definition}\label{rem:normalization}
A {\it perfect orientation\/} $\O$ of a plabic graph $G$ is a
choice of orientation of each of its edges such that each
black internal vertex $u$ is incident to exactly one edge
directed away from $u$; and each white internal vertex $v$ is incident
to exactly one edge directed towards $v$.
A plabic graph is called {\it perfectly orientable\/} if it admits a perfect orientation.
The {\it source set\/} $I_\O \subset [n]$ of a perfect orientation $\O$ is the set of all $i$ which 
are sources of $\O$ (considered as a directed graph). Similarly, if $j \in \overline{I}_{\O} := [n] - I_{\O}$, then $j$ is a sink of $\O$.
If $G$ has type  $\pi_{k,n}$, then 
each perfect orientation of $G$ will have 
a source set 
	of size  $n-k$ \cite[Lemma 9.4 and discussion after Remark 11.6]{Postnikov}.
\end{definition}

\begin{figure}[h]
\centering
\includegraphics[height=1in]{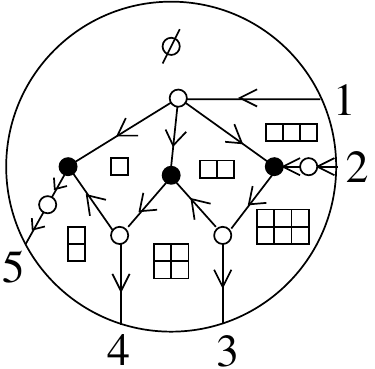}
\caption{The perfect orientation $\O$ of a 
plabic graph $G$ of type $\pi_{3,5}$ with source set 
$I_\O = \{1,2\}$.}
\label{G25-orientation}
\end{figure}

The following lemma appeared in \cite{PSWv1}.\footnote{The
published version of \cite{PSWv1}, namely \cite{PSW}, did not include the lemma, because it turned out
to be unnecessary.}

\begin{lem}[{\cite[Lemma 3.2 and its proof]{PSWv1}}]\label{PSW-lemma} Each reduced plabic graph $G$ has an acyclic perfect orientation $\O$.  Moreover, we may choose $\O$ so that 
the set of boundary sources $I$ is the index set for the lexicographically minimal non-vanishing Pl\"ucker coordinate
on the corresponding cell. 
(In particular, if $G$ is of type $\pi_{k,n}$, then we can choose $\O$ so that $I = \{1,\dots,n-k\}$.)
Then given another reduced plabic graph $G'$ which is move-equivalent to $G$, 
we can transform $\O$ into a perfect orientation $\O'$ for $G'$, such that $\O'$ is also acyclic with boundary 
sources $I$, using oriented versions of the moves (M1), (M2), (M3).  Up to rotational symmetry, we will only 
need to use the oriented 
version of the move (M1) shown in Figure \ref{oriented-square-move}.
\end{lem}
\begin{figure}[h]
\centering
\includegraphics[height=.6in]{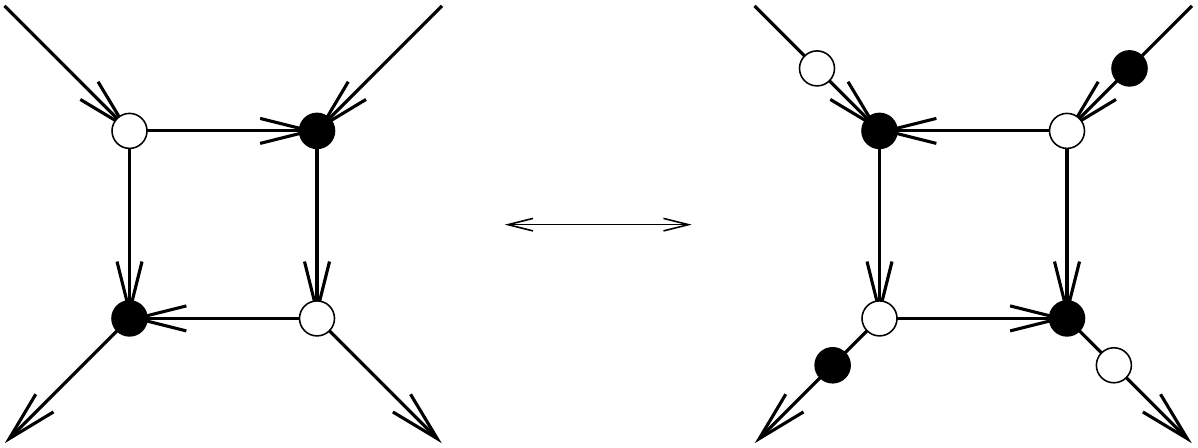}
\caption{Oriented square move}
\label{oriented-square-move}
\end{figure}

\begin{remark}\label{acyclic}
By Lemma \ref{PSW-lemma}, a reduced plabic 
graph $G$ of type $\pi_{k,n}$ always has 
  an acyclic
 perfect orientation $\mathcal O$ with 
source set $I_{\mathcal O}=\{1,\dotsc, n-k\}$, as in 
\cref{G25-orientation}.  Moreover it follows from \cite[Lemma 4.5]{PSW} that this
is the unique perfect orientation %
with  
source set $\{1,\dotsc, n-k\}$.
From now on we will always 
choose our perfect orientation to be acyclic with source set $\{1,\dotsc,n-k\}$;
we prefer this choice because 
then the variable
$x_{\emptyset}$ never appears in the expressions for flow polynomials, 
and 
we always have 
$P_{\max}=1$.
\end{remark}

Recall from Definition~\ref{def:facelabels}
that we  label each face of $G$ by a Young diagram 
in $\widetilde{\mathcal{P}}_G \subset \Shkn$.  
Let \begin{equation}\label{e:TB}
\widetilde{\TB}(G) := \{x_{\mu}\ |\ \mu\in \widetilde{\mathcal P}_G \}
\end{equation}
be a set of  
parameters which are indexed by the Young diagrams $\mu$ labeling
faces of $G$.  
Since one of the faces of $G$ 
is labeled by the empty partition, $\emptyset$, we also set
\begin{equation}
\TBG:= \{x_\mu \ |\ \mu\in \mathcal P_G\}=\widetilde{\TB}(G) \setminus \{x_{\emptyset}\}.
\end{equation}

A \emph{flow} $F$ from $I_{\O}$ to a set $J$ of boundary vertices
with 
$|J|=|I_{\O}|$ 
is a collection of paths
in $\O$, all pairwise vertex-disjoint, 
such that the sources of these paths are $I_{\O} - (I_{\O} \cap J)$
and the destinations are $J - (I_{\O} \cap J)$.

Note that each path %
$w$ in $\O$ partitions the faces of $G$ into those 
which are on the left and those which are on the right of the 
walk. %
We define the \emph{weight} $\wt(w)$ of 
each such path %
to be the product of 
parameters $x_{\mu}$, where $\mu$ ranges over all face labels 
to the left of the path.  And we define the 
\emph{weight} $\wt(F)$ of a flow $F$ to be the product of the weights of 
all paths in the flow.

Fix a perfect orientation $\O$ of a reduced plabic graph $G$
of type $\pi_{k,n}$.  
Given $J \in {[n] \choose n-k}$, 
we define the {\it flow polynomial}
\begin{equation}\label{eq:Plucker}
P_J^G = \sum_F \wt(F),
\end{equation}
where $F$ ranges over all flows from $I_{\O}$ to $J$.

\begin{example}
We continue with our running example
from Figure \ref{G25-orientation}. There are two flows $F$
from $I_{\O}$ to $\{2,4\}$, and 
$P^G_{\{2,4\}} = x_{\ydiagram{3}} x_{\ydiagram{2,2}} x_{\ydiagram{3,3}}
+x_{\ydiagram{2}} x_{\ydiagram{3}} x_{\ydiagram{2,2}} x_{\ydiagram{3,3}}$.
There is one flow from $I_{\O}$ to $\{3,4\}$, and
$P^G_{\{3,4\}} = x_{\ydiagram{2}} x_{\ydiagram{3}} x_{\ydiagram{2,2}}
x_{\ydiagram{3,3}}^2.$
\end{example}

We now describe the network chart for $\openX$ associated to a plabic graph $G$.
The result concerning the totally positive Grassmannian
below comes from \cite[Theorem 12.7]{Postnikov}, while the extension to $\openX$ comes from 
\cite{TalaskaWilliams} (see also \cite{MullerSpeyer}).

\begin{theorem}[{\cite[Theorem 12.7]{Postnikov}}]\label{network_param}
Let $G$ be a reduced plabic graph of type $\pi_{k,n}$, and choose
an ayclic perfect orientation $\O$ with source set $I_{\O}=\{1,\dotsc,n-k\}$.
Let $A$ be the $(n-k)\times n$ matrix with rows indexed by $I_{\O}$ whose $(i,j)$-entry equals
$$
(-1)^{|\{i'\in [n-k]: i < {i'} < j\}|}\sum_{p:i\to j}\wt(w),
$$
where the sum is over all paths $w$ in $\O$ from $i$ to  $j$. Then the map $\Phi_G$ sending 
$(x_{\mu})_{\mu \in \mathcal{P}_G}
\in (\C^*)^{\mathcal{P}_G}$ to the element of $\mathbb X$ represented by $A$ 
is an injective map onto a dense open subset of $\openX$.    The restriction of $\Phi_G$ to 
$(\R_{>0})^{\mathcal{P}_G}$ gives a parameterization of the totally positive Grassmannian
$\mathbbX(\mathbf \R_{>0})$.  We call the map $\Phi_G$ a \emph{network chart} for $\openX$.
\end{theorem}

\begin{example}\label{ex:matrix}
For example, the graph and orientation in 
\cref{G25-orientation}
 gives for the matrix $A$
$$
\Phi_G((x_{\mu})_{\mu \in \mathcal{P}_G}) = 
\kbordermatrix{
& 1 & 2 & 3 & 4 & 5 \cr
1 & 1 & 0 & -x_{\ydiagram{3}} x_{\ydiagram{3,3}} & -x_{\ydiagram{3}} x_{\ydiagram{3,3}} x_{\ydiagram{2,2}} (1+x_{\ydiagram{2}}) & -x_{\ydiagram{3}} x_{\ydiagram{3,3}} x_{\ydiagram{2,2}} x_{\ydiagram{1,1}} (1+x_{\ydiagram{2}} + x_{\ydiagram{2}} x_{\ydiagram{1}})  \cr
2 & 0 & 1  & x_{\ydiagram{3,3}} &  x_{\ydiagram{3,3}} x_{\ydiagram{2,2}}  & x_{\ydiagram{3,3}} x_{\ydiagram{2,2}} x_{\ydiagram{1,1}}
}.
$$
\end{example}

The following result gives a 
formula for the Pl\"ucker coordinates 
of points in the image of $\Phi_G$.
In our setting, the result is essentially the 
Lindstrom-Gessel-Viennot Lemma.  However
 \cref{thm:Talaska}
can be generalized to  
arbitrary perfect orientations of a reduced plabic graph,
see  \cite[Theorem 1.1]{Talaska}.
\begin{thm}
\label{thm:Talaska}
Let $G$ be as in \cref{network_param} and let $J \in {[n] \choose n-k}$.
Then the Pl\"ucker coordinate $P_\lambda$ of $\Phi_G((x_{\mu})_{\mu \in \mathcal{P}_G})$, i.e. the minor with column set $J$ of the matrix
$A$, %
is equal to the flow polynomial
$P_J^G$ from \eqref{eq:Plucker}.
\end{thm}

\begin{definition}[Network torus $\mathbb T_G$]  \label{d:networktorus}
Define the open dense torus $\mathbb T_G $ in $\openX$ to be the image  of the network chart $\Phi_G$, namely
$\mathbb T_G
:=\Phi_G((\C^*)^{\mathcal{P}_G})$.
We call $\mathbb T_G$ the {\it network torus} in $\openX$ associated to $G$. 
\end{definition}

\begin{example}\label{Ex:Transc}
Since the image of $\Phi_G$ lands in 
	$\X^\circ$ (see \cite[Section 1.1]{MullerSpeyer}), %
	we can view the parameters $\TBG$ as rational functions on $\mathbbX$ which restrict to coordinates on the open torus $\mathbb T_G$. Therefore we can think of $\TBG$ as a transcendence basis of $\C(\mathbbX)$.
Moreover it is clearly positive in the sense of Definition~\ref{d:TranscBasis}.
\end{example}

\begin{example}\label{ex:A-example}
We continue with our running example
from Figure \ref{G25-orientation} and \cref{ex:matrix}.
 The formulas for the 
Pl\"ucker coordinates of $\Phi_G((x_\mu)_{\mu\in \mathcal P_G})$ are:
\begin{align*}
P_{\{1,2\}} &= 1, & &
P_{\{1,3\}} = x_{\ydiagram{3,3}},\\
P_{\{1,4\}} &= x_{\ydiagram{2,2}} x_{\ydiagram{3,3}}, & &
P_{\{1,5\}} = 
x_{\ydiagram{1,1}} x_{\ydiagram{2,2}} x_{\ydiagram{3,3}}, \\
P_{\{2,3\}} &= x_{\ydiagram{3}} x_{\ydiagram{3,3}}, & &
P_{\{2,4\}} = x_{\ydiagram{3}} x_{\ydiagram{2,2}} x_{\ydiagram{3,3}}(1+
x_{\ydiagram{2}}),\\
P_{\{2,5\}} &= x_{\ydiagram{3}} x_{\ydiagram{1,1}} x_{\ydiagram{2,2}}
x_{\ydiagram{3,3}}(1+x_{\ydiagram{2}} +x_{\ydiagram{1}} x_{\ydiagram{2}}), &&
P_{\{3,4\}} = x_{\ydiagram{2}} x_{\ydiagram{3}} x_{\ydiagram{2,2}}
x_{\ydiagram{3,3}}^2,\\
P_{\{3,5\}} &= x_{\ydiagram{2}} x_{\ydiagram{3}} x_{\ydiagram{1,1}}
x_{\ydiagram{2,2}} x_{\ydiagram{3,3}}^2 (1+x_{\ydiagram{1}}),  &&
P_{\{4,5\}} = x_{\ydiagram{1}} x_{\ydiagram{2}} x_{\ydiagram{3}}
x_{\ydiagram{1,1}} x_{\ydiagram{2,2}}^2 x_{\ydiagram{3,3}}^2.
\end{align*}
One may obtain these Pl\"ucker coordinates either directly from the matrix in 
\cref{ex:matrix} or by computing the flow polynomials from \cref{G25-orientation}.
Note that $x_\emptyset$ does not appear in the flow polynomials 
since the region labeled by $\emptyset$ is to the right of every path from $I_\O$ to $[n]\setminus I_\O$. 
One may 
invert the map $\Phi_G$ and express the $x_\mu$ as rational functions in the Pl\"ucker coordinates, thus describing $\TBG$ as a subset of $\C(\mathbb X)$. 
\end{example}

\begin{definition}[Strongly minimal, strongly maximal,
 and pointed]\label{def:minimal}
We say that a Laurent monomial 
$\prod_\mu x_\mu^{a_\mu}$ appearing
in a Laurent polynomial $P$ is \emph{strongly minimal}
(respectively, \emph{strongly maximal}) in $P$  if 
for every other Laurent monomial $\prod_\mu x_\mu^{b_\mu}$ occurring in $P$,
we have $a_\mu\le b_{\mu}$ (respectively, $a_\mu \geq b_{\mu}$) for all $\mu$.  

If $P$ has a strongly minimal Laurent
monomial with coefficient $1$, then we say  $P$ is \emph{pointed}.
Consider 
a plabic graph $G$ and perfect orientation with source set $\{1,\dotsc, n-k\}$. Recall that the flow polynomial 
$P_J$ %
is a sum over  flows from $\{1,\dotsc, n-k\}$ to $J$. 
We call a flow from $\{1,\dotsc, n-k\}$ to $J$ \emph{strongly minimal} 
(respectively, \emph{strongly maximal})
if it has a  strongly minimal (respectively, strongly 
maximal) weight monomial in $P_J$. 
\end{definition}

\begin{remark}
In Example \ref{ex:A-example}, 
each flow polynomial $P_{\{i,j\}}$ 
has a strongly minimal and a strongly maximal term.
This is true in general; see \cref{prop:strongminimal}.
\end{remark}

We next describe cluster $\mathcal{X}$-mutation, and how it relates to network parameters.

\begin{definition}\label{Xseed}
Let $Q$ be a quiver with vertices $V$,  associated exchange matrix $B$
(see \cref{quiver}), and with a 
parameter $x_{\mu}$ associated to each vertex $\mu \in V$.
If $\lambda$ is a mutable vertex of $Q$,
then we define a new set of 
parameters  $\MutVar_{\lambda}^{\mathcal{X}}(\{x_{\mu}\}) := \{x'_{\mu}\}$  where
\begin{equation}\label{e:XclusterMut}
x'_{\mu} = \begin{cases}
       \ \frac{1}{x_{\lambda}} & %
      \text{if }\mu = \lambda, \\
   \ x_{\mu}(1+x_{\lambda})^{b_{\lambda \mu}}  
  &\text{if there are }b_{\lambda \mu}
\text{ arrows from }\lambda \text{ to } \mu \text{ in }Q, \\
  \ \frac{x_{\mu}}{(1+x_{\lambda}^{-1})^{b_{\mu \lambda}}} &%
 \text{if there are }b_{\mu \lambda}
\text{ arrows from } \mu \text{ to } \lambda \text{ in }Q,\\
  \ x_{\mu} & %
 \text{ otherwise.}
\end{cases}
\end{equation}
We say that 
$(\Mut_{\lambda}(Q), \{x'_{\mu}\})$ is obtained from 
 $(Q, \{x_{\mu}\})$ by \emph{$\mathcal{X}$-seed mutation}
 in direction $\lambda$, and we refer to the ordered pairs 
$(\Mut_{\lambda}(Q), \{x'_{\mu}\})$  and
 $(Q, \{x_{\mu}\})$ as \emph{labeled $\mathcal{X}$-seeds}.
	Note that if we apply the $\mathcal{X}$-seed mutation in direction $\lambda$ to 
$(\Mut_{\lambda}(Q), \{x'_{\mu}\})$, we obtain  
 $(Q, \{x_{\mu}\})$ again.
	
	We say that two labeled $\mathcal{X}$-seeds are 
	\emph{$\mathcal{X}$-mutation equivalent} if one can be obtained
	from the other by a sequence of $\mathcal{X}$-seed mutations.

	If $f$ is a rational expression in the parameters $\{x_{\mu}\}$, 
	we use $\Mut_{\mathcal{X},\lambda}(f)$ to denote the new expression for $f$
obtained by rewriting it in terms of the  $\{x'_{\mu}\}$.
\end{definition}

Using the terminology of \cref{Xseed}, each reduced plabic graph $G$ gives rise to a 
labeled $\mathcal{X}$-seed $(Q(G), \widetilde{\Network}(G))$.
The following lemma, which is easy to check, shows that our flow polynomial 
expressions for Pl\"ucker coordinates are compatible with the $\mathcal{X}$-mutation.
In other words, whenever two plabic graphs are connected by moves, the corresponding
$\mathcal{X}$-seeds are $\mathcal{X}$-mutation equivalent.

\begin{lemma}\label{lem:mutXface}
Let $G$ be a reduced plabic graph with network parameters
$\widetilde{\TB}(G) := \{x_{\mu}\ |\ \mu\in \widetilde{\mathcal P}_G \}$,
associated quiver $Q(G)$, and with a fixed perfect orientation e.g. as in Remark~\ref{acyclic}.  
Let $\lambda$ be a square face of $G$ formed by four trivalent vertices.
Let $G'$ be another reduced plabic graph with associated data, obtained from $G$ by performing an oriented square move at 
$\lambda$, see
\cref{oriented-square-move}.
	Then for each $J\in {[n] \choose n-k}$,
	the mutation $\Mut_{\mathcal{X},\lambda}(P_J^G(\{x_{\mu}\}))$ of the flow polynomial $P_J^G$ 
is equal to the flow polynomial $P_J^{G'}(\{x'_{\mu}\})$ expressed in the network 
parameters of 
$G'$.
\end{lemma}

\begin{example}
We continue with our running example from 
\cref{G25-partitions}, 
  \cref{G25-orientation}, and 
	\cref{ex:A-example}.
Then we have that 
	$P^G_{\{1,3\}} = x_{\ydiagram{3,3}}.$  
	If we perform an oriented square move on $G$ at the vertex 
	$\ydiagram{2}$, we obtain a new perfectly oriented plabic graph $G'$ (with network parameters labeled as before but with a prime).
Using $G'$ we obtain 
	$P^{G'}_{\{1,3\}} = x'_{\ydiagram{3,3}}+x'_{\ydiagram{3,3}}x'_{\ydiagram{2}}$.  
	On the other hand, applying the $\mathcal{X}$-seed mutation to the network parameters of $G$ gives 
	$x'_{\ydiagram{2}} = \frac{1}{x_{\ydiagram{2}}}$ and $x'_{\ydiagram{3,3}} = \frac{x_{\ydiagram{3,3}}}{(1+x_{\ydiagram{2}}^{-1})}$.
	Substituting these expressions into 
	$P^{G'}_{\{1,3\}}$ gives back 
	$x_{\ydiagram{3,3}} = P^G_{\{1,3\}}.$  
\end{example}

\begin{remark}\label{rem:Xcluster}
By \cref{lem:mutXface} and \cref{rem:moves}, all $\mathcal{X}$-seeds coming from plabic graphs $G$ of 
type $\pi_{k,n}$ are $\mathcal{X}$-mutation equivalent.
	We can mutate the quiver $Q(G)$ at a vertex which does not correspond to a square face of $G$; 
	however, this will 
lead to quivers that no longer correspond to plabic graphs.
Nevertheless, 
	we can consider an arbitrary labeled $\mathcal{X}$-seed
	$(Q, \{x_{\mu}\})$ which is $\mathcal{X}$-mutation equivalent 
	to an $\mathcal{X}$-seed coming from a reduced plabic graph
 of type $\pi_{k,n}$; we say that 
	$(Q, \{x_{\mu}\})$ also has type $\pi_{k,n}$.
  In this 
case we still have a (generalized) network chart, also called $\mathcal X$-cluster chart, which
is obtained from the network chart $\Phi_G$
	of \cref{network_param} by composing the $\mathcal{X}$-seed 
	mutations of 
\cref{e:XclusterMut}; and there is 
a corresponding network or $\mathcal X$-cluster  torus, see also \cref{s:twist}.
Abusing notation, we will continue to index such $\mathcal{X}$-seeds, network charts, and network tori by $G$ (rather than 
$(Q, \{x_{\mu}\})$, but 
will take care to indicate when we are working with an arbitrary $\mathcal{X}$-seed rather than one coming 
	from a plabic graph.  
	
	We will sometimes use the notation $\Phi_{G,\mathcal{X}}$ for $\Phi_G$
	and $\mathbb{T}_{G,\mathcal{X}}$ for $\mathbb{T}_G$
	to emphasize the $\mathcal{X}$-cluster structure. 
	Note that the formulas 
\eqref{e:XclusterMut}
	naturally define a birational map
	$\mathcal{M}_{\mathcal{X},\lambda}: \mathbb{T}_{G,\mathcal{X}} \dashrightarrow \mathbb{T}_{G',\mathcal{X}}$.
	The mutation 
	$\Mut_{\mathcal{X},\lambda}$ 
	of rational functions 
	defined earlier is just the 
	pullback of $\mathcal{M}_{\mathcal{X},\lambda}^{-1}$.

Recall from \cref{acyclic} that our conventions guarantee that $x_{\emptyset}$ never appears in the expressions
for flow polynomials (which are Pl\"ucker coordinates).  Since $\emptyset$ labels a frozen vertex of our quiver $Q(G)$,
when we perform arbitrary mutations on $Q(G)$ (possibly leaving the setting of plabic graphs),
our expressions for Pl\"ucker coordinates will continue to be independent of the 
parameter associated to the frozen vertex $\emptyset$.  
\end{remark}

\section{The twist map and general $\mathcal X$-cluster tori}\label{s:twist}

In this section we define the \emph{twist map} on 
$\openX$, and, following Marsh-Scott \cite{MarshScott} and Muller-Speyer \cite{MullerSpeyer}, we 
explain how it connects network and cluster parameterizations coming from the same plabic graph $G$. We then use the twist map to deduce that the regular function on a network torus $\mathbb T_G$ coming from a Pl\"ucker coordinate stays regular after an arbitrary sequence of $\mathcal X$-cluster mutations. Thus we will see that the $\mathcal X$-cluster tori embed into $\Xcirc$ where they glue together.

The \emph{twist map} is an automorphism of $\openX$ which allows one to relate 
cluster charts and network charts.  It was first defined in the context of double Bruhat
cells by Berenstein, Fomin, and Zelevinsky \cite{BFZ, BZ}, and subsequently defined
for $\openX$ by Marsh and Scott \cite{MarshScott}.  
  Shortly thereafter it was defined 
for all positroid varieties (including $\openX$) by Muller and Speyer \cite{MullerSpeyer}, using a slightly 
different convention; they also relate their version of the twist to 
the one from \cite{BFZ}, see \cite[Section A.4]{MullerSpeyer} and references
therein.  We follow the conventions and terminology
of \cite{MullerSpeyer} in this paper.

\begin{definition}
Let $A$ denote an $(n-k) \times n$ matrix representing an element of 
$\openX$.  Let $A_i$ denote the $i$th column of $A$, with indices
taken cyclically; that is, $A_{i+n} = A_i$.  
Let $\langle -,-\rangle$ denote the standard Euclidean inner product on $\C^{n-k}$.

The \emph{left twist} of $A$ is the 
	$(n-k) \times n$ matrix $\lefttwist(A)$ such that, for all $i$, the $i$th column
$\lefttwist(A)_i$ satisfies 
$$\langle \lefttwist(A)_i \ \vert \ A_i \rangle = 1,\text{ and }$$ 
$$\langle \lefttwist(A)_i \ \vert \ A_j \rangle = 0 \text{ if }A_j
\text{ is not in the span of }\{A_{j+1}, A_{j+2},\dots, A_{i-1}, A_{i}\}.$$ 
\end{definition}

\begin{theorem} [{\cite[Theorem 6.7 and Corollary 6.8]{MullerSpeyer}}]
The map $\lefttwist$ is a 
regular automorphism of $\openX$. 
\end{theorem}

The inverse of $\lefttwist$, though we will not need it here, is called the right twist. 
The following theorem is a version of \cite[Theorem 7.1]{MullerSpeyer}. (It is also closely related to \cite[Theorem 1.1]{MarshScott}.) However, in \cite{MullerSpeyer}, the network tori 
were parameterized in terms of variables associated to {edges} rather than faces of $G$, so the notation looks different.

\begin{theorem}[{\cite[Theorem 7.1]{MullerSpeyer}}]
\label{thm:twist} 
There is an isomorphism 
$\leftdel=\leftdel_G$  of tori
	such that the following diagram commutes. 

	\[
\begin{tikzcd}
	\mathbb (\C^*)^{\PGhat} \arrow{r}{\leftdel} \arrow{d}{\Phi_{G,\mathcal A}} & \arrow{d}{\Phi_{G,\mathcal X}}(\C^*)^{\tilde{\mathcal{P}}_G\setminus\emptyset} \\
\openX \arrow{r}{\lefttwist} & \openX \\
\end{tikzcd}
\]

\end{theorem}

The left twist is closely related to the exchange matrix. 

\begin{prop}[{\cite[Corollary 5.11]{MullerSpeyer}\cite{Muller:PC}}]
	\label{p:lefttwist} 
	Let $G$ be a reduced plabic graph of type $\pi_{k,n}$, and $B$ the associated exchange matrix. Then there exists an adjusted exchange matrix $\tilde B=B+M$, where $M\in\Z^{\mathcal P_G\x\mathcal P_G}$ has the property that $M_{\mu,\nu}=0$ unless both $\mu$ and $\nu$ 
index frozen variables, such that the left twist is given by 
\[
(\leftdel)^*(x_\mu)= 
	\prod_{\nu\in \mathcal P_G} P_{\nu}^{\tilde B_{\mu,\nu}} \text{ for }\mu\in {\tilde{\mathcal{P}}_G\setminus \emptyset}  
,
\]
in terms of the $\mathcal X$- and $\mathcal A$-cluster charts associated to $G$.
In particular the pullback of the network parameter $x_\mu$, when $\mu$ is  mutable, is encoded in the original exchange matrix. 
\end{prop}

For mutable $\mu$ this proposition is simply \cite[Corollary 5.11]{MullerSpeyer}, restated using the exchange matrix. The adjustment required for frozen $\mu$ (choice of $M$) is technical and was left out from the paper \cite{MullerSpeyer} on those grounds, \cite{Muller:PC}.

\begin{example}
        Let us continue our running example using the plabic graph $G$ from
\cref{G25-partitions}.
We can express the element
	$\Phi_{G,\mathcal{A}}((P_{\mu})_{\mu\in \PGhat})$ as the matrix
$$A = \kbordermatrix{
& 1 & 2 & 3 & 4 & 5 \cr
	1 & 1 & 0 & -P_{\ydiagram{2,2}} 
	&  -(P_{\ydiagram{2,2}} P_{\emptyset} + P_{\ydiagram{2}} P_{\ydiagram{1,1}})/{P_{\ydiagram{1}}} 
	& -P_{\ydiagram{2}} \cr
	2 & 0 & 1  & ({P_{\ydiagram{1}} + P_{\ydiagram{3}} P_{\ydiagram{2,2}}})/{P_{\ydiagram{2}}} &  
	({P_{\ydiagram{1}} P_{\emptyset} + P_{\ydiagram{3}} P_{\ydiagram{2,2}} P_{\emptyset} + P_{\ydiagram{3}} P_{\ydiagram{2}} P_{\ydiagram{1,1}}})/{P_{\ydiagram{1}} P_{\ydiagram{2}}}
	& P_{\ydiagram{3}}
}. $$
	To get this matrix, we simply express the entries of a (row-reduced) matrix representing 
	an element of $\openX$ in terms of the set of Pl\"ucker coordinates  in the
	$\mathcal{A}$-cluster of $G$.

If we apply the left twist to $A$, we obtain the matrix 
	$$\lefttwist(A) 
	= \kbordermatrix{
& 1 & 2 & 3 & 4 & 5 \cr
	1 & 1 & 0 & -1/P_{\ydiagram{2,2}} 
	&  -(P_{\ydiagram{1}}+P_{\ydiagram{3}}P_{\ydiagram{2,2}})/P_{\ydiagram{2}}P_{\ydiagram{1,1}}
	&  
	-({P_{\ydiagram{1}} P_{\emptyset} + P_{\ydiagram{3}} P_{\ydiagram{2,2}} P_{\emptyset} + P_{\ydiagram{3}} P_{\ydiagram{2}} P_{\ydiagram{1,1}}})/{P_{\emptyset} P_{\ydiagram{1}} P_{\ydiagram{2}}}
	\cr
	2 & P_{\ydiagram{2}}/P_{\ydiagram{3}}  & 1  & 0 &  
	-P_{\ydiagram{2,2}}/P_{\ydiagram{1,1}}	
	& -(P_{\ydiagram{2,2}} P_{\emptyset} + P_{\ydiagram{2}} P_{\ydiagram{1,1}})/P_{\emptyset} P_{\ydiagram{1}}
}. $$

	Meanwhile, the adjusted exchange matrix $\tilde{B}$ is given by \cref{table:BM}.
Using 
	\cref{p:lefttwist}, we compute 
	\begin{align*}
		(\leftdel)^*(x_{\ydiagram{1}})&= 
		\frac{P_{\ydiagram{1,1}}P_{\ydiagram{2}}}{P_{\emptyset} P_{\ydiagram{2,2}}} &&
		(\leftdel)^*(x_{\ydiagram{2}})= \frac{P_{\ydiagram{2,2}} P_{\ydiagram{3}}}{P_{\ydiagram{1}}} &&
		(\leftdel)^*(x_{\ydiagram{3}})= \frac{P_{\ydiagram{3}}}{P_{\ydiagram{2}}} \\
		(\leftdel)^*(x_{\ydiagram{3,3}})&= \frac{P_{\ydiagram{2}}}{P_{\ydiagram{2,2}} P_{\ydiagram{3}}} &&
		(\leftdel)^*(x_{\ydiagram{2,2}})= \frac{P_{\ydiagram{2,2}}P_{\ydiagram{1}}}{P_{\ydiagram{2}} P_{\ydiagram{1,1}}}  &&
		(\leftdel)^*(x_{\ydiagram{1,1}})= \frac{P_{\ydiagram{1,1}}}{P_{\ydiagram{1}}}
	\end{align*}
If we now substitute the 
expressions for the $(\leftdel)^*(x_\mu)$ into the matrix from 
\cref{ex:matrix}, we obtain a matrix
	whose Pl\"ucker coordinates agree with those of 
	$\lefttwist(A)$, illustrating that the diagram from 
\cref{thm:twist} commutes.

\begin{center}
 \renewcommand{\arraystretch}{1.5}
    \begin{table}[h]
	    \begin{tabular}{| c | p{.7cm} | p{.7cm} | p{.7cm} | p{.7cm} | p{.7cm} | p{.7cm} | p{.7cm}| }
            \hline
        & $\ydiagram{1}$ & $\ydiagram{2}$ & $\ydiagram{3}$ 
	     & $\ydiagram{3,3}$ & $\ydiagram{2,2}$ & $\ydiagram{1,1}$ & $\emptyset$ \\
                \hline
		$\ydiagram{1}$ & $0$ & $1$ & $0$ & $0$ & $-1$ & $1$ & $-1$\\ \hline
		$\ydiagram{2}$ & $-1$ & $0$ & $1$ & $-1$ & $1$ & $0$& $0$ \\ \hline
		    $\ydiagram{3}$ & $0$ & $-1$ & $\mathbf{1}$ & $\mathbf{0}$ & $\mathbf{0}$ & $\mathbf{0}$ & $\mathbf{0}$ \\ \hline
		    $\ydiagram{3,3}$ & $0$ & $1$ & $\mathbf{-1}$ & $\mathbf{0}$ & $\mathbf{-1}$ & $\mathbf{0}$ & $\mathbf{0}$\\ \hline
		    $\ydiagram{2,2}$ & $1$ & $-1$ & $\mathbf{0}$ & $\mathbf{0}$ & $\mathbf{1}$ & $\mathbf{-1}$ & $\mathbf{0}$\\ \hline
		    $\ydiagram{1,1}$ & $-1$ & $0$ & $\mathbf{0}$ & $\mathbf{0}$ & $\mathbf{0}$ & $\mathbf{1}$ & $\mathbf{0}$ \\ \hline
		    $\emptyset$ & $1$ & $0$ & $\mathbf{0}$ & $\mathbf{0}$ & $\mathbf{0}$ & $\mathbf{0}$ & $\mathbf{0}$\\ \hline
                \end{tabular}
\vspace{0.2cm}
		\caption{The adjusted exchange matrix $\tilde{B} = B+M$.  The entries $M_{\mu, \nu}$ where both $\mu$ and $\nu$ index 
		frozen variables are displayed in bold.}
                \label{table:BM}
        \end{table}
\end{center}


\end{example}

Let $\mathcal X$ and $\mathcal A$ denote the spaces obtained by gluing together all of the $\mathcal X$-cluster tori, respectively, the $\mathcal A$-cluster tori, for varying seeds, using the rational maps given by mutation. From the work of Scott~\cite{Scott} we know that we have an embedding $\mathcal A\hookrightarrow \Xcirc$. Our goal is to prove the analogous result for $\mathcal X$. 

Let $\mathcal X^{\operatorname{net}}$ be the union of the network tori (associated to plabic graphs) glued together via the mutation maps. Recall that the network parameterizations define an embedding $\mathcal X^{\operatorname{net}}\hookrightarrow \Xcirc$.

\begin{prop}\label{p:XLaurent}
The map $\mathcal X^{\operatorname{net}}\hookrightarrow \Xcirc$ extends to an embedding
$
\mathcal X\hookrightarrow \Xcirc.
$
	Moreover, any Pl\"ucker coordinate $P_{\lambda}$, when expressed in terms of a general $\mathcal X$-cluster $G$, is 
	a Laurent polynomial in $\TB(G)$.
\end{prop}

\begin{remark} The second statement of \cref{p:XLaurent} is obvious when $G$ is a plabic graph (indeed, the network 
	expansions are even polynomial).  However, it is 
	not obvious when $G$ is a general $\mathcal X$-cluster because there is no 
	``Laurent phenomenon" for $\mathcal X$-cluster varieties; 
 the mutation formulas of 
\cref{e:XclusterMut}
	 are rational but not Laurent.
\end{remark}

We recall a result about twists, generalizing a construction from \cite{GSV03} and \cite{FG}, which applies in our setting as follows.

\begin{prop}[{\cite[Proposition~4.7]{HWilliams:KMcluster}}]\label{p:HWilliams}
Fix a seed $G$ with exchange matrix $B$. Suppose $M\in\Z^{\mathcal P_G\x\mathcal P_G}$ satisfies that $M_{\mu,\nu}=0$ unless both $\mu$ and $\nu$ 
index frozen variables. Let $\tilde B=B+M$. Let us denote by $\{X_\mu\}$ the $\mathcal X$-cluster variables associated to $G$, and by $\{A_\mu\}$ the $\mathcal A$-cluster variables associated to $G$. Consider the map $p_M^G$ from the $\mathcal X$-cluster torus $\mathbb T_{G,\mathcal{X}}$ 
to the $\mathcal A$-cluster torus $\mathbb T_{G,\mathcal{A}}$
associated to $G$ defined by the formula
\[
(p_M^G)^*(X_\mu)= \prod_{\nu\in \mathcal P_G} A_{\nu}^{\tilde B_{\mu,\nu}}.
\]
This map is compatible with mutation and extends to a regular map $p_M:\mathcal A\to \mathcal X$. 
In particular, whenever $G$ and $G'$ are adjacent seeds related by mutation at $\nu$, 
we have a commutative diagram
 \[
\begin{tikzcd}
	\mathbb T_{G,\mathcal{A}} \arrow{r}{p_M^G} \arrow[dashed]{d}{\mathcal{M}_{\mathcal{A}, \nu}} & \arrow[dashed]{d}{\mathcal{M}_{\mathcal{X}, \nu}}\mathbb T_{G,\mathcal{X}} \\
\mathbb T_{G',\mathcal A} \arrow{r}{p_M^{G'}} & \mathbb T_{G', \mathcal X} \\
\end{tikzcd}
\]
where $p_M^{G'}$ is defined in terms of the matrix 
$\Mut_{\nu}(B)+M$.
\end{prop} 

We are now in a position to prove~\cref{p:XLaurent}.
 
\begin
{proof} [{Proof of~\cref{p:XLaurent}}]

The map $\mathcal X^{\operatorname{net}}\hookrightarrow \Xcirc$ can be extended to a rational map $\mathcal X\to \Xcirc$ using mutation. By the combination of \cref{p:HWilliams} and \cref{p:lefttwist} we have the commutative diagram

 \[
\begin{tikzcd}
	\CA \arrow{r}{p_M} \arrow[hookrightarrow]{d} & \CX \arrow[dashed]{d} \\
\Xcirc \arrow{r}{\lefttwist} & \Xcirc \\
\end{tikzcd}.
\]

Here the left hand vertical map is the embedding of \cite{Scott} (which is equivalent to the assertion 
that Pl\"ucker coordinates 
are Laurent polynomials in the variables of any $\mathcal{A}$-cluster), while the right hand vertical map is so far only known to be rational. 
By \cref{p:lefttwist}, 
we have that on a cluster torus associated to a plabic graph $G$, the map $p_M^G$ is 
given by $\leftdel = [\tilde{B}_{\mu, \nu}]$, and 
by \cref{thm:twist} it 
 is invertible. Since mutation preserves the rank of a matrix
\cite[Lemma 3.2]{BFZ2}, the global map $p_M:\mathcal A\to \mathcal X$ is also invertible. Now the diagram implies that the vertical map on the right must be an embedding, just
like the map on the left.   This implies that the Pl\"ucker coordinates are regular functions on each of the $\mathcal{X}$-cluster tori, and thus given by Laurent polynomials.
\end{proof}

\section{The Newton-Okounkov  body $\NO_{G}(D)$}\label{sec:NO}

In this section we define the {\it Newton-Okounkov  body} $\NO_{G}(D)$
 associated to an ample divisor in $\mathbbX$ of the form $D=r_1 D_1+\cdots +r_n D_n$, see \cref{s:positroid},  along with a choice of transcendence basis $\TBG$ of $\C(\mathbbX)$, see Definition~\ref{d:networktorus} and Example~\ref{Ex:Transc}. The theory of Newton-Okounkov bodies was  developed  by Kaveh and Khovanskii, and Lazarsfeld and Mustata, see
\cite{KavehKhovanskii,KK2, LazarsfeldMustata},  building on Okounkov's original construction \cite{Ok96, Okounkov:symplectic,Ok03} which was inspired also by 
a formula for moment polytopes due to Brion~\cite{Brion:moment}. Our exposition below mainly follows \cite{KavehKhovanskii}. 
 A key property of a Newton-Okounkov  body associated to a divisor $D$ is that its Euclidean volume encodes the volume of $D$, i.e.\ the asymptotics of $\dim(H^0(\X,\mathcal O(rD)))$ as $r\to\infty$. 
In our setting we will see that the lattice points of $\NO_G(rD)$ count the dimension of the space of sections $H^0(\X, \mathcal O(rD))$  also for all finite $r$. 

Fix a reduced plabic graph $G$ or a labeled $\mathcal{X}$-seed $\Sigma^{\mathcal X}_G$ of type $\pi_{k,n}$.
To define the Newton-Okounkov  body $\NO_G(D)$  we first construct a valuation $\val_G$ on $\C(\mathbbX)$ from the transcendence basis $\TBG$.

\begin{defn}[The valuation $\val_G$] \label{de:val}
Given a general $\mathcal{X}$-seed  $\Sigma^{\mathcal X}_G$ of type $\pi_{k,n}$, 
we fix a total order $<$ on the parameters $x_{\mu} \in \TBG$. This order extends 
 to a term order on monomials in the parameters $\TBG$ 
which is lexicographic with respect to $<$. For example if $x_\mu<x_\nu$ then $x_{\mu}^{a_1} x_{\nu}^{a_2}<x_{\mu}^{b_1} x_{\nu}^{b_2}
$ if either $a_1<b_1$, or if $a_1=b_1$ and $a_2<b_2$.
We use the multidegree of the lowest degree summand to define a valuation
\begin{equation} \label{eq:valuation}
\val_{G}:\C(\mathbbX)\setminus\{0\}\to \Z^{{\mathcal P}_G}.
\end{equation}
Explicitly, let $f$ be  a polynomial in the  Pl\"ucker coordinates 
for $\mathbbX$. We use  \cref{thm:Talaska}, Definition~\ref{d:networktorus}, and \cref{p:XLaurent} 
to write $f$ uniquely as a Laurent polynomial in $\TBG$.
We then choose the  lexicographically minimal term $\prod_{\mu\in{\mathcal P}_G}x_\mu^{a_\mu}$ and define $\val_G(f)$ to be the associated exponent vector $(a_\mu)_\mu\in \Z^{{\mathcal P}_G}$.
In general for
$(f/g) \in \C(\mathbbX) \setminus\{0\}$ (here $f,g$ are polynomials
in the Pl\"ucker coordinates), the valuation is defined by
$\val_G(f/g) = \val_G(f) - \val_G(g)$.
Note however 
	that we will only  be applying $\val_G$ to functions whose $\mathcal X$-cluster expansions are 
	Laurent.  
\end{defn}

\begin{defn}[The Newton-Okounkov body $\NO_G(D)$]\label{def:NObody}
Let $D\subset \mathbbX$ be a divisor in the complement of $\openX$, that is we have $D=\sum r_iD_i$, compare with~\cref{s:positroid}. Denote by $L_{rD}$, the subspace of $\C(\X)$ given by
\[
L_{rD}:=H^0(\mathbbX,\mathcal O(rD)). 
\]
By abuse of notation we write $\val_G(L)$ for $\val_G(L \setminus \{0\})$. We define the {\it Newton-Okounkov  body} associated to $\val_G$ and the divisor $D$ by 
\begin{equation}\label{e:NOviaNOGr}
\NO_G(D)=
\overline{\operatorname{ConvexHull}
	\left(\bigcup_{r=1}^{\infty} \frac{1}{r} 
\val_G(L_{rD})\right)}.
\end{equation}
If we choose $D=D_{n-k}$,
we will refer to $\NO_G(D)$ simply as $\Delta_G$.
\end{defn}
\begin{defn}
For any subset $\mathcal S$ of $\R^{\mathcal P_G}$ we denote its subset of lattice points by $\Lattice(\mathcal S):=\mathcal S\cap\Z^{\mathcal P_G}$.
\end{defn}

\begin{rem}[Toy example]\label{e:polytope} 
Suppose $\PDelta \subset \R^m$ is a convex $m$-dimensional polytope. Associated to $\PDelta$ consider the set $\Lattice(r\PDelta)$ of lattice points in the dilation $r\PDelta$. Then we observe that
\begin{equation}
\PDelta=
\overline{\operatorname{ConvexHull}
\left(\bigcup_r \frac{1}{r} 
\Lattice(r\PDelta)\right)}.
\end{equation}
In particular if for a polytope $\PDelta\in\R^{\mathcal P_G}$ the lattice points $\Lattice(r\PDelta)$ coincide with  $\val_G(L_{rD})$ from \cref{def:NObody}, then it immediately follows that $\PDelta$ is the Newton-Okounkov  body $\NO_G(D)$. 
\end{rem}

\begin{rem}[The special case  of $D_{n-k}$] 
We will often choose our divisor $D$ in $\mathbbX$ to be
$D_{n-k}=
\{P_{\Max}=0\}$. 
We note that explicitly $H^0(\mathbbX,\mathcal O(rD_{n-k}))$ is the linear subspace of $\C(\mathbbX)$ described as follows
\begin{equation}\label{e:projnormal}
H^0(\mathbbX,\mathcal O(rD_{n-k})) =L_{r}:=\left<\frac{M}
{(P_{\Max})^r} \ \vert \  
M\in\mathcal{M}_{r}%
\right>,
\end{equation} 
where $\mathcal{M}_r$ is the set of all degree $r$ monomials in the 
Pl\"ucker coordinates. 
Recall that $H^0(\mathbbX,\mathcal O(rD_{n-k}))$  is naturally an irreducible representation of $GL_n(\C)$, namely it is isomorphic to $V_{r\omega_{n-k}}^*$.
	The identity \eqref{e:projnormal} says that $\X$ is {\it projectively normal} and follows from representation theory, see \cite[Section 2]{GrossWallach}. Namely, restriction of sections gives a nonzero equivariant map of $GL_n(\C)$-representations, $H^0(\mathbb P(\bigwedge^{n-k}\C^n),\mathcal O(r)) \to 
 H^0(\mathbbX,\mathcal O(rD_{n-k})) $, which must be surjective since its target is irreducible.
 
 For simplicity of notation we will usually write $\val_G(M)$ for $\val_G(M/P_\Max^r)$. 
Thus we write $\val_G(P_\lambda)$ instead of $\val_G(P_\lambda/P_\Max)$ and talk about the valuation of a Pl\"ucker coordinate.
\end{rem}

Starting from the divisor $D_{n-k}$ we introduce a set of lattice polytopes $\conv_G^r$.

\begin{defn}[The polytope $\conv_G^r$]
\label{d:NOrG} For each reduced plabic graph $G$ of type $\pi_{k,n}$ and related valuation $\val_G$ we define lattice polytopes $\conv_G^r$ in $\R^{\mathcal P_{G}}$ by
\[
\conv_G^r:=\operatorname{ConvexHull}\left(\val_G(L_r)\right),
\]
for $L_r$ as in \eqref{e:projnormal}. 
When $r=1$, we also write $\conv_G:=\conv_G^1$.
\end{defn}

The lattice polytope $\conv_G$ (resp. $\conv_G^r$) 
is what $\val_G$ associates to the divisor $D$ (resp. $rD$) directly, without taking account of the asymptotic behaviour of the powers of $\mathcal O(D)$. 
Since we will fix $D=D_{n-k}$ 
when considering the polytopes $\Conv_G^r$, we don't indicate the dependence on $D$
in the notation $\conv_G^r$.

\begin{remark}
Note that we used a total order $<$ on the parameters in order to 
define $\val_G$, and different choices give slightly differing valuation maps. However 
$\NO_G$ and the polytopes ${\conv}_G^r$, will turn out not to depend on our choice of total order, and that choice will not enter into our proofs. 
\end{remark}

\begin{rem}[Valuations associated to flags] \label{r:flagval} The valuations used in Okounkov's original construction come from flags of subvarieties $X\supset X_1\supset \cdots X_{N-1}\supset X_N=\{pt\}$, see also \cite[Section 1.1]{LazarsfeldMustata}. Our valuations $\val_G$ definitely do not all come from flags. 
For example in the case of the rectangles cluster, if the ordering on the $x_\mu$ is not compatible with inclusion of Young diagrams, then our valuation cannot come from a flag. 
In general, our definition can be interpreted as choosing, via a network chart, a birational isomorphism of $\mathbbX$ with $\C^N$, and then taking a standard flag of linear subspaces in $\C^N$.  
\end{rem}

We immediately point out some fundamental properties of the sets $\val_G(L_{r})$ defining our polytopes $\conv_G^r$. The first property is a version of the key lemma from \cite{Ok96}. It says, in the terminology of \cite{KavehKhovanskii} 
(see \cref{d:1dimleaves}),  
that the valuation $\val_G$ has one-dimensional leaves.

\begin{lem}[Version of {\cite[Lemma from Section 2.2]{Ok96}} ]
\label{l:okounkovlemma}
Consider $\C(\mathbbX)$ with the valuation $\val_G$ from Definition~\ref{de:val}.  For  any finite-dimensional linear subspace $L$ of $\C(\mathbbX)$,
 the cardinality of the image $\val_G(L)$ equals the dimension of $L$. In particular, the cardinality of the set $\val_G(L_r)$ equals
 the dimension of the vector space $L_r$ from \eqref{e:projnormal}, namely it is the dimension of the representation $V_{r\omega_{n-k}}$ of $GL_n(\C)$. 
\end{lem} 

The proof uses the valuation and the total order on $\Z^{\mathcal P_G}$ to define in the natural way a filtration 
\[
L=(L)_{\ge \mathbf a_1}\supset (L)_{\ge \mathbf a_2}\supset 
(L)_{\ge \mathbf a_3}\supset\cdots \supset 
(L)_{\ge \mathbf a_m}\supset \{0\},
\]
of $L$ indexed by $\val_G(L)=\{\mathbf a_1,\dotsc, \mathbf a_m\}$, where $L_{\ge a}=\{f\in L\mid \val_G(f)\ge \mathbf a\}$ and similarly with $\ge$ replaced by $>$. The result follows by observing that successive quotients $(L)_{\ge \mathbf a}/(L)_{> \mathbf a}$ are isomorphic to $\C$ by the isomorphism which takes the coefficient of the leading term.

\begin{example}\label{ex:NO}
We now take $r=1$ and compute the polytope ${\conv}_G$ associated to 
Example~\ref{ex:A-example}.  Computing the valuation 
of each Pl\"ucker coordinate we get the result shown in Table~\ref{table:Dynkin}.
Therefore  ${\conv}_G$ is the convex hull of the set of points 
$\left\{(\right.0,0,0,0,0,0),
(1,0,0,0,0,0),
(1,1,0,0,0,0),
(1,1,1,0,0,0),$ 
$(1,0,0,1,0,0),\\
(1,1,0,1,0,0),
(1,1,1,1,0,0),
(2,1,0,1,1,0),
(2,1,1,1,1,0),
(2,2,1,1,1,1\left .)\right\}$.

It will follow from results in \cref{s:QGT} that in this example, 
$\conv_G = \Delta_G$.
\end{example}

\begin{center}
 \renewcommand{\arraystretch}{1.5}
    \begin{table}[h]
     \begin{tabular}{| c || p{.7cm} | p{.7cm} | p{.7cm} | p{.7cm} | p{.7cm} | p{.7cm} |}
            \hline
       Pl\"ucker & $\ydiagram{3,3}$ & $\ydiagram{2,2}$ & $\ydiagram{1,1}$ 
          & $\ydiagram{3}$ & $\ydiagram{2}$ & $\ydiagram{1}$ \\
                \hline
                \hline
                $P_{1,2}$ & $0$ & $0$ & $0$ & $0$ & $0$ & $0$ \\ \hline
                $P_{1,3}$ & $1$ & $0$ & $0$ & $0$ & $0$ & $0$ \\ \hline
                $P_{1,4}$ & $1$ & $1$ & $0$ & $0$ & $0$ & $0$ \\ \hline
                $P_{1,5}$ & $1$ & $1$ & $1$ & $0$ & $0$ & $0$ \\ \hline
                $P_{2,3}$ & $1$ & $0$ & $0$ & $1$ & $0$ & $0$ \\ \hline
                $P_{2,4}$ & $1$ & $1$ & $0$ & $1$ & $0$ & $0$ \\ \hline
                $P_{2,5}$ & $1$ & $1$ & $1$ & $1$ & $0$ & $0$ \\ \hline
                $P_{3,4}$ & $2$ & $1$ & $0$ & $1$ & $1$ & $0$ \\ \hline
                $P_{3,5}$ & $2$ & $1$ & $1$ & $1$ & $1$ & $0$ \\ \hline
                $P_{4,5}$ & $2$ & $2$ & $1$ & $1$ & $1$ & $1$ \\ \hline
                \end{tabular}
\vspace{0.2cm}
                \caption{The valuations $\val_G(P_J)$  
of the Pl\"ucker coordinates}
                \label{table:Dynkin}
        \end{table}
\end{center}
                \vspace{-.5cm}

\section{A non-integral example of $\Delta_G$ for $Gr_3(\C^6)$}\label{sec:Milena}

We say that two plabic graphs are \emph{equivalent modulo (M2) and (M3)}
if they can be related by any sequence of moves of the form (M2) and 
(M3) as defined in \cref{sec:moves}.  
For $Gr_3(\C^6)$, there are precisely $34$ equivalence classes 
of plabic graphs of type $\pi_{3,6}$ modulo (M2) and (M3).
Milena Hering pointed out to us an example of such a plabic graph
$G^1$ such that 
$\Delta_{G^1}$ is non-integral.  We then did a computer check with Polymake and found that among the $34$
equivalence classes, only two give rise to non-integral Newton-Okounkov polytopes:
the graph $G^1$ as well as the closely related graph $G^2$ shown in \cref{fig:Milena}.  The other $32$ equivalence classes give rise to 
integral Newton-Okounkov polytopes.
\begin{figure}[h]
\centering
\includegraphics[height=1.3in]{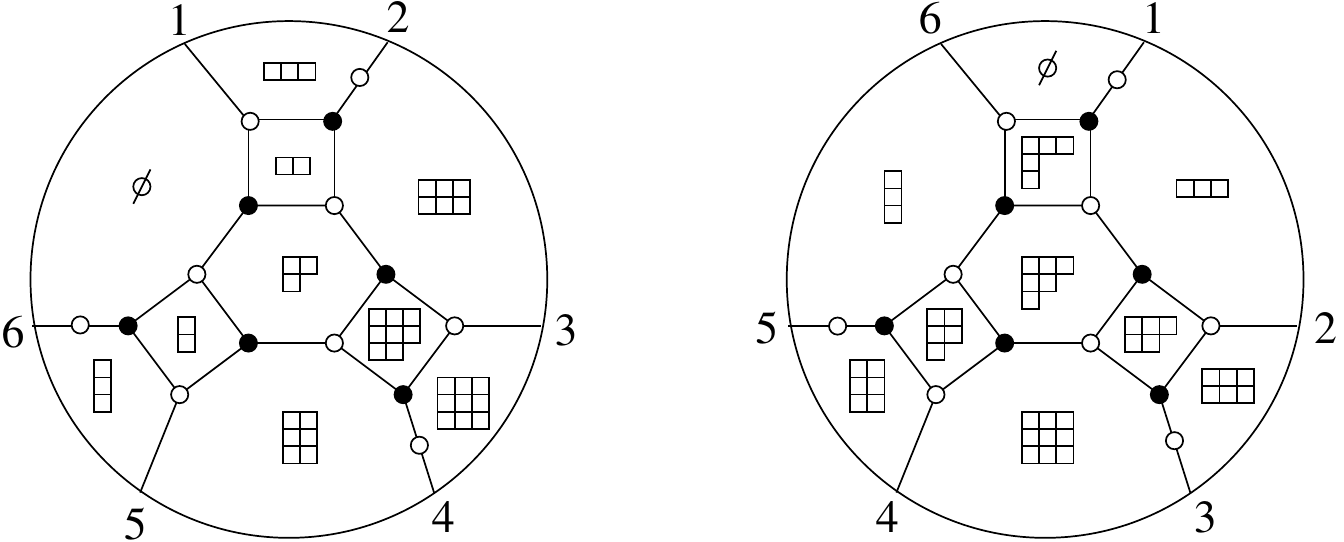}
\caption{The plabic graphs $G^1$ and $G^2$ such that 
$\Delta_{G^1}$ and $\Delta_{G^2}$ are not integral.}
\label{fig:Milena}
\end{figure}
Here  we computed the vertices of the Newton-Okounkov polytopes $\Delta_G$ by giving Polymake the inequality description
of $\Q_G$ (see Definition~\ref{l:B-modelPolytope}).  Then \cref{thm:main},  proved later in this paper,
says that  $\Delta_G = \Q_G$.

For example, the vertices of the polytope $\Delta_{G^1}$ are the valuations of the Pl\"ucker coordinates
together with one single non-integral vertex with coordinates as in \cref{nonintegral}.
 \renewcommand{\arraystretch}{2.3}
    \begin{table}[h]
     \begin{tabular}{|p{.7cm} | p{.7cm} | p{.7cm} | p{.7cm} | p{.7cm} | p{.7cm} | p{.7cm} | p{.7cm} | p{.7cm}|}
            \hline
         $\ydiagram{3,3,3}$ & $\ydiagram{3,3,2}$ & $\ydiagram{2,2,2}$ 
          & $\ydiagram{1,1,1}$ & $\ydiagram{3,3}$ & $\ydiagram{2,1}$ & $\ydiagram{1,1}$ & $\ydiagram{3}$ & $\ydiagram{2}$ \\
                \hline
                $\frac{3}{2}$ & $\frac{3}{2}$ & $1$ & $\frac{1}{2}$ & $1$ & $\frac{1}{2}$ & $\frac{1}{2}$ & $\frac{1}{2}$ & $\frac{1}{2}$\\ \hline
                \end{tabular}
	       \caption{}
	       \label{nonintegral}
        \end{table}

It is straightforward to check that this non-integral vertex represents half the valuation 
of the flow polynomial for the element 
$f= (P_{124} P_{356} - P_{123} P_{456})/P_{\Max}^2 \in L_2$.  
This element (and plabic graph) also appear in \cite[Section A.3]{MullerSpeyer},
where the authors observe that up to column rescaling, $f$ is the 
twist of the Pl\"ucker coordinate $P_{246}$.  (Their conventions for 
labeling faces of plabic graphs are slightly different from ours.)

For $Gr_3(\C^7)$, there are $259$ equivalence classes of plabic graphs of type
$\pi_{3,7}$ modulo (M2) and (M3).  Of the corresponding Newton-Okounkov polytopes, precisely
$216$ are integral and $43$ are non-integral \cite{Maplecode}.

\section{The superpotential and its associated polytopes} \label{s:ConvrQr}

\subsection{The superpotential $W$} \label{s:superpotential}
Following \cite{MarshRietsch}, we define the superpotential mirror dual to $\mathbbX$. We refer to \cite[Section~6]
{MarshRietsch} for more detail. Recall definitions from Sections~\ref{s:notation} and \ref{sec:cluster} relating to Pl\"ucker coordinates of $\checkX$.

\begin{definition}\label{def:superpotential}

Let $\mu_i^{\square}$ be the Young diagram associated to 
the $k$-element subset of horizontal steps
$J_i^+:=[i+1,i+k-1]\cup \{i+k+1\}$, where the index $i$ is always interpreted modulo $n$.  
Then for $i \ne n-k$,  
the Young diagram ${\mu}_i^\square$ %
turns out to be the unique diagram in $\Shkn$ obtained by adding a single  box to $\mu_i$.
And for $i=n-k$, the Young diagram  ${\mu}_{n-k}^\square$ associated to 
$J_{n-k}^+$ is the rectangular $(n-k-1)\x (k-1)$ Young diagram obtained from $\mu_{n-k}$ by removing a rim hook.

We define the \emph{superpotential} dual to the Grassmannian $\X$ to be the regular function $W:\opencheckX\x \C^* \to \C$ given by
\begin{equation}\label{e:Wq}
W= \sum_{i=1}^{n}q^{\delta_{i,n-k}}\frac{p_{{\mu}_i^\square}}{p_{\mu_i}},\end{equation}
where $q$ is the coordinate on the $\C^*$~factor and $\delta_{i,j}$ is the  Kronecker delta function. 
For $i=1,\dotsc, n$ we also define $W_i\in\C[\opencheckX]$ by
\begin{equation}\label{e:W_i}
W_i:= \frac{p_{\mu_i^\square}}{p_{\mu_i}}=\frac{p_{J_i^+}}{p_{J_i}},
\end{equation}
so that
	$W= q W_{n-k} + \sum_{i\neq n-k} W_i$.  
\end{definition}

\begin{example} For $k=3$ and $n=5$ we have $\mathbbX=Gr_2(\C^5)$ and $\checkX=Gr_3((\C^5)^*)$. The anticanonical divisor $\check D_{\ac}$ is given by 
\[
\check D_{\ac}= \{\p_{\ydiagram{3}}=0\}\cup\{\p_{\ydiagram{3,3}}=0\}\cup\{\p_{\ydiagram{2,2}}=0\} \cup \{\p_{\ydiagram{1,1}}=0\}\cup
\{\p_\emptyset=0\},
\]
compare with~\cref{s:positroid}, and
\[
W=\frac{p_{\ydiagram{3,1}}}{p_{\ydiagram{3}}}+q\frac{p_{\ydiagram{2}}}{p_{\ydiagram{3,3}}}+\frac{p_{\ydiagram{3,2}}}{p_{\ydiagram{2,2}}}+\frac{p_{\ydiagram{2,1}}}{p_{\ydiagram{1,1}}}+\frac{p_{\ydiagram{1}}}{p_\emptyset}.
\]
\end{example}

\begin{definition}[Universally positive]\label{d:universallypositive}
We say that a Laurent polynomial is \emph{positive} if all of its coefficients are in $\R_{>0}$. An element $h\in\C[\opencheckX]$ is called {\it universally positive}  (for the $\mathcal A$-cluster structure) if for every $\mathcal A$-cluster seed  $\check\Sigma_G^{\mathcal A}$ the expansion $\mathbf h^G$ of $h$ in $\PCG$ is a positive Laurent polynomial. Similarly $f\in\C[\opencheckX \times \C^*]$ is called universally positive if its expansion $\mathbf f^G$ in the variables $\PCG\cup\{q\}$ is given by a positive Laurent polynomial for every seed $\check\Sigma_G^{\mathcal A}$. 
\end{definition}

\begin{remark}\label{rem:cluster}
 Recall from Section~\ref{sec:cluster} the
$\mathcal A$-cluster algebra structure on
the homogeneous coordinate ring of the Grassmannian.  In the formula
 \eqref{e:W_i} for $W_i$, the numerator is a 
Pl\"ucker coordinate (and hence a cluster variable), 
and the denominator is a frozen variable.
Therefore by the positivity of the Laurent phenomenon \cite{LeeSchiffler, GHKK}, $W_i$ is an example of a universally positive element of $\C[\opencheckX]$. 
Similarly, the superpotential $W$ comes from the cluster algebra with $q$ adjoined and is universally positive in the extended sense. 
Proposition \ref{prop:MR} below gives the cluster expansion of $W$ 
in terms of 
the rectangles cluster. 
\end{remark}

\begin{prop}[{\cite[Proposition 6.10]{MarshRietsch}}]\label{prop:MR}
If we let $i \times j$ denote the Young diagram which is a rectangle 
with $i$ rows and $j$ columns, then on the subset of $\opencheckX$ where all $p_{i\x j}\ne 0$, the superpotential $W$ equals
\begin{equation}\label{eq:superpotential}
W = \frac{p_{1 \times 1}}{p_{\emptyset}} + \sum_{i=2}^{n-k} \sum_{j=1}^k
\frac{p_{i \times j} \ p_{(i-2) \times (j-1)}}{p_{(i-1)\times (j-1)} \
p_{(i-1)\times j}} + 
q \frac{p_{(n-k-1) \times (k-1)}}{p_{(n-k) \times k}} + 
\sum_{i=1}^{n-k} \sum_{j=2}^k \frac{ p_{i \times j} \ p_{(i-1)\times (j-2)}}
{p_{(i-1)\times (j-1)} \ p_{i \times (j-1)}}.
\end{equation}
Here of course if $i$ or $j$ equals $0$, then
$p_{i \times j} = p_{\emptyset}$. We may furthermore set  $p_{\emptyset}=1$ on $\opencheckX$. 
\end{prop}

The Laurent polynomial \eqref{eq:superpotential} can be encoded in a diagram (shown in Figure \ref{fig:superpotential}
 for $k=3$ and $n=5$), see \cite[Section 6.3]{MarshRietsch}.
\begin{figure}[h]
\begin{center}
\begin{tikzpicture}
\draw (0,2) [fill] circle (.4mm);
\draw[->](0,0)--(1,0);
\draw[->](1,0)--(2,0);
\draw[->](2,0)--(3,0);
\draw[->](0,2)--(0,1);
\draw[->](0,1)--(0,0);
\draw[->](1,1)--(1,0);
\draw[->](0,1)--(1,1);
\draw[->](1,1)--(2,1);
\draw[->](2,1)--(2,0);
\node[below] at (3,0) {$q$};
\node [below] at (2,0) {$\frac{\pytt}{\pyz}$};
\node [below] at (1,0) {$\frac{\pyzz}{\pyo}$};
\node [below] at (0,0) {$\frac{\pyoo}{p_\emptyset}$};
\node [above] at (1,1) {$\frac{\pyz}{p_\emptyset}$};
\node [above] at (2,1) {$\frac{\pyt}{p_\emptyset}$};
\node [left] at (0,1) {$\frac{\pyo}{p_\emptyset}$};
\node [left] at (0,2) {$1$};
\end{tikzpicture}
\end{center}
\caption{The diagram defining the superpotential for $k=3$ and $n=5$.}
\label{fig:superpotential}
\end{figure}
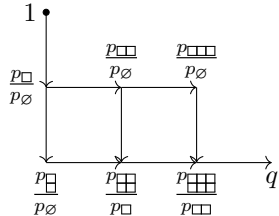
Namely each arrow in the diagram determines a Laurent monomial by dividing the expression at the head by the expression at the tail of the arrow. The Laurent polynomial obtained as the sum of all of these Laurent monomials gives the expression for $W$ in Proposition~\ref{p:XLaurent}.  
So in this example, we have 
\begin{equation}\label{super}
W = 
\pyo+
\frac{\pyoo}{\pyo}+
\frac{\pyzz }{\pyo \ \pyz}+
\frac{\pytt }{\pyz \ \pyt}+
\frac{\pyz}{\pyo} + 
\frac{\pyt}{\pyz}+
\frac{\pyzz }{\pyo \ \pyoo}+
\frac{\pytt \ \pyo}{\pyz \ \pyzz}+
q \frac{\pyz}{\pytt},
\end{equation}
where we have made use of the normalization of Pl\"ucker coordinates on $\opencheckX$ given by  
$p_\emptyset = 1$. 

\begin{rem}
The quiver underlying the diagram above was introduced by \cite{BC-FKvS} where it was encoding the EHX Laurent polynomial superpotential   \cite{EHX} associated to a Grassmannian (in the vein of Givental's quiver for the full flag variety \cite{Givental:fullflag}). It was related to the Peterson variety in \cite{RietschNagoya} before appearing in connection with the rectangles cluster in \cite{MarshRietsch}. 
\end{rem}

\subsection{Polytopes via tropicalisation}\label{sec:superpolytope}

In this section we define 
a polytope $\Q_G^r$ in terms of inequalities, which are obtained by 
restricting the superpotential to the cluster torus $\mathbb T^\vee_G$ and 
applying a tropicalisation procedure, see \cite{MacLaganSturmfels} and references therein.  %
We also define a polytope $\Q_G(r_1,\dots, r_n)$, which generalizes
$\Q_G^r$, and which will be discussed in \cref{s:generalD}.

\begin{definition}[naive Tropicalisation]\label{def:Tropicalisation}
To any Laurent polynomial $\mathbf h$ in variables $X_1,\dotsc, X_m$ with coefficients in $\R_{>0}$ we associate a piecewise linear map $\Trop(\mathbf{h}):\R^m \to \R$ 
 called the {\it tropicalisation} of $\mathbf h$ as follows. We set $\Trop(X_i)(y_1,\dotsc, y_m)=y_i$. If $\mathbf{h_1}$ and $\mathbf{h_2}$ 	are two positive Laurent polynomials, and $a_1,a_2\in\R_{>0}$, then we impose the condition that
\begin{equation}\label{eq:min-etc}
\Trop(a_1\mathbf{h_1}+a_2\mathbf{h_2}) = 
\min(\Trop(\mathbf{h_1}),\Trop(\mathbf{h_2})),\text{ and }
\Trop(\mathbf{h_1}\mathbf{h_2}) = 
\Trop(\mathbf{h_1}) + \Trop(\mathbf{h_2}).
\end{equation}
This defines $\Trop(\mathbf h)$ for all positive Laurent polynomials $\mathbf h$, by induction.
\end{definition}

 
\begin{rem} Informally, $\Trop(\mathbf h)$ is obtained by replacing multiplication by addition, and addition by $\min$. For example if $\mathbf h=X_1\inv X_3^2+5 X_2+X_1 X_2^{-3} X_3$ then $\Trop(\mathbf h)(y_1,y_2,y_3)=\min(2y_3-y_1,y_2, y_1-3y_2+y_3)$.
\end{rem}

Now let $G$ be a reduced plabic graph of type $\pi_{k,n}$ with associated set
of cluster variables $\PCG$, see 
\eqref{e:PCG}. 
Suppose $\mathbf{h}:\mathbb T^\vee_G\x \C^*\to\C$  is a positive Laurent polynomial in the variables $\PCG\cup\{q\}$ with coefficients in $\R_{>0}$. In this case the tropicalisation is a (piecewise linear) map
\[
\Trop(\mathbf{h}):\R^{\mathcal P_G}\x \R \to \R,
\] 
in variables that we denote $((v_\mu)_{\mu\in\mathcal P_G}, r)$.  Similarly, if $\mathbf h:\mathbb T^\vee_G\to\C$, then $\Trop(\mathbf{h}):\R^{\mathcal P_G} \to \R$. 
\begin{definition}\label{d:TropG}
Suppose $f\in \C[\opencheckX]$ is universally positive with  $\mathcal A$-cluster expansion  $\mathbf f^G$.
Then we define 
$\Trop_G(f)$ to be the tropicalisation $\Trop(\mathbf f^G):\R^{\mathcal P_G}\to \R$.  Similarly, if $f\in \C[\opencheckX\x\C^*_{q}]$ is universally positive, so that $\mathbf f^G$ is a positive Laurent polynomial in the variables $\PCG\cup\{q\}$, then we use the same notation, $\Trop_G(f)$, to mean the map $\Trop(\mathbf f^G):\R^{\mathcal P_G}\x \R \to \R$.
\end{definition}

By \cref{rem:cluster},  the superpotential $W$ is universally positive, so that 
$\Trop_G(W):
\R^{\mathcal P_G} \times \R \to\R$
	 is well-defined for any seed $\check\Sigma^{\mathcal A}_G$. We now use $\Trop_G(W)$ to define a polytope.

\begin{definition}\label{l:B-modelPolytope}
For $r\in\R$ we define the \emph{superpotential polytope}
\[
\Q_G^r=\{v\in \R^{\mathcal P_G}\mid \Trop_G(W)(v,r)\ge 0\}.
\]
When $r=1$, we will also write $\Q_G := \Q^1_G$.
\end{definition}

\begin{remark}\label{rem:dilation}
Note that the right hand side is a convex subset of
$\R^{\mathcal P_G}$ given by inequalities determined by the Laurent polynomial $W^G=W|_{\mathbb T^\vee_G\x\C^*}$. It will follow from
\cref{lem:integralGT}
and
\cref{c:PolytopeMutation}
that $\Q_G^r$ is in fact bounded and hence a convex polytope for $r\ge 0$.
In this case it also follows directly from the definitions that 
$\Q_G^r=r\Q_G $. Hence we will primarily restrict our attention to 
$\Q_G$. If $r<0$ we will have $\Q_G^r=\emptyset$ as follows from Proposition~\ref{p:empty}. 
\end{remark}

\begin{remark}\label{rem:rectanglesinequalities}
	In the case that $G = G_{k,n}^{\rect}$ from \cref{sec:Grectangles}, we can 
use the formula \eqref{eq:superpotential} for the superpotential
to obtain the following inequalities 
defining $\Q^r_{G_{k,n}^{\rect}} $:
\begin{align}
0 &\leq v_{1 \times 1} \label{eq1} \\
v_{(n-k) \times k} - v_{(n-k-1)\times (k-1)} &\leq r \label{eq2}\\
v_{(i-1)\times j} - v_{(i-2)\times (j-1)} &\leq 
v_{i \times j}  - v_{(i-1)\times (j-1)}
\ \text{ for }2 \leq i \leq n-k \text{ and }
1\leq j \leq k \label{eq3}\\
v_{i\times (j-1)}-v_{(i-1)\times (j-2)} &\leq 
v_{i\times j}  - v_{(i-1)\times (j-1)}
\ \text{ for } 1\leq i \leq n-k \text{ and }
2 \leq j \leq k \label{eq4}
\end{align}
\end{remark}

\begin{example}\label{ex:Q}
	Let $G=G_{3,5}^{\rect}$  be the plabic graph
from \cref{G25-partitions}.
The superpotential $W$ is written out in terms
	of  $\PCG \cup \{q\}$ in
\eqref{super}.  
We obtain the following inequalities which 
define the polytpe $\Q^\RG$.
\begin{align*}
0 &\leq v_{\ydiagram{1}} & &
0  \leq v_{\ydiagram{1,1}} - v_{\ydiagram{1}}\\
0 & \leq v_{\ydiagram{2,2}} - v_{\ydiagram{1}} - v_{\ydiagram{2}} & &
0  \leq v_{\ydiagram{3,3}} - v_{\ydiagram{2}} - v_{\ydiagram{3}}\\
0 & \leq v_{\ydiagram{2}} - v_{\ydiagram{1}}&&
0  \leq v_{\ydiagram{3}} - v_{\ydiagram{2}}\\
0 & \leq v_{\ydiagram{2,2}} - v_{\ydiagram{1}} - v_{\ydiagram{1,1}}&&
0  \leq v_{\ydiagram{3,3}} + v_{\ydiagram{1}} - v_{\ydiagram{2}} - v_{\ydiagram{2,2}}\\
0 & \leq r+ v_{\ydiagram{2}} - v_{\ydiagram{3,3}}& 
\end{align*}
One can check that in this case,
$\Q_G$ is precisely the polytope
$\conv_G$ from 
\cref{ex:NO}.
This is true  for any rectangles cluster, 
see \cref{prop:convQ},
	 but is false in general, see 
\cref{sec:Milena}.
It would be interesting to characterize when $\Q_G = \conv_G$.
\end{example}

We also have a natural generalisation of the superpotential polytope defined as follows.
Recall the summands $W_i\in\C[\opencheckX]$ 
of the superpotential from \eqref{e:W_i}.
Each $W_i$ is itself universally positive and gives rise
to a piecewise linear function $\Trop_G(W_i): \R^{\mathcal{P}_G} \to \R$ for any $\mathcal A$-cluster seed $\check\Sigma^{\mathcal A}_G$.

\begin{definition}\label{def:generalQ}
Choose $r_1,\dots,r_n \in \R$.  
	We define the \emph{generalized superpotential polytope} by 
	\begin{equation}
	\Q_G(r_1,\dots,r_n) = \bigcap_i \{v\in \R^{\mathcal{P}_G} \ \vert \ \Trop_G(W_i)(v) + r_i \geq 0\}.
	\end{equation}
In particular if $r_{n-k}=r$ and $r_i=0$ for $i \neq n-k$, then 
$\Q_G(r_1,\dots,r_n) = \Q_G^r$.
\end{definition}

\section {Tropicalisation, total positivity, and mutation}

\subsection{Total positivity and generalized Puiseux series}

The $\mathcal A$-cluster structure on the Grassmannian $\checkX$, which is a {\it positive atlas} in the terminology of \cite{FG1}, gives rise to a `tropicalized version' of $\checkX$. This, inspired by \cite{Lusztig3}, is defined in \cite{FG1}  as the analogue of the totally positive part with $\R_{>0}$ replaced by the tropical semifield $(\R,\min, +)$. We construct the tropicalisation of $\checkX$ and our polytopes in terms of total positivity  over generalized Puiseux series, extending the original construction of \cite{Lusztig3}. Our initial goal will be to describe how the polytopes $\Q_G(r_1,\dotsc, r_n)$ behave under mutation of $G$. 

\begin{definition}[Generalized Puiseux series]\label{d:genPuiseux}
Following~\cite{Markwig:Puiseux}, let $\mathbf K$ be the field of generalized 
Puiseux series in one variable with set of exponents taken from 
\[
\MonSeq=\{ A\subset \R \mid  \operatorname{Cardinality}(A\cap \R_{\le x})<\infty \text{ for arbitrarily large $x\in \R$} \}.
\]
Note that a set $A\in \operatorname{MonSeq}$ can be thought of as a strictly monotone increasing sequence of numbers which is either finite or countable tending to infinity. We write $(\alpha_m) \in \MonSeq$ if $(\alpha_m)_{m\in\Z_{\ge 0}}$ is such a strictly monotone increasing sequence, and we have
\begin{equation}\label{e:K}
\mathbf K=\left\{c(t)=\sum_{(\alpha_m) \in \MonSeq} c_{\alpha_m} t^{\alpha_m}\mid c_{\alpha_m}\in\C \right\}.
\end{equation}
Note that $\mathbf K$ is complete and algebraically closed, see \cite{Markwig:Puiseux}. 
 We denote by $\mathbf K_{>0}$ the subsemifield of $\mathbf K$ defined by
\begin{equation}\label{e:Kpos}
\mathbf K_{>0}=\left\{c(t)\in \mathbf K\mid c(t)=\sum_{(\alpha_m) \in \MonSeq} c_{\alpha_m} t^{\alpha_m} , \ c_{\alpha_0}\in \R_{>0} \right\}.
\end{equation}
We have an $\R$-valued valuation, 
$%
\ValK:\mathbf K\setminus\{0\} \to \R$,
given by
$\ValK\left(c(t)\right)=\alpha_0$
if $c(t)=\sum  
c_{\alpha_m} t^{\alpha_m}$ where the lowest order term is assumed to have non-zero coefficient, $c_{\alpha_0}\ne 0$.

We also use the notation $\mathbf L:=\R((t))$ for the field of real Laurent series in one variable. Note that $\mathbf L\subset \mathbf K$. We let $\mathbf L_{>0}=\LL\cap \mathbf K_{>0}$, and denote by  $\ValL$ the lowest-order-term valuation of $\mathbf L$.
\end{definition}

Lusztig \cite{Lusztig3} applied his theory of total positivity for an algebraic group $\mathcal G$ not just to defining a notion of $\R_{>0}$-valued points, `the totally positive part', inside $\mathcal G(\R)$, but also to introducing $\mathbf L_{>0}$-valued points $\mathcal G(\mathbf L)$. Moreover, he used this theory to describe his parameterization of the canonical basis, see \cite[Section~10]{Lusztig3}. In our setting, there is a notion of totally positive part  $\checkX(\mathbf L_{>0})$ in $\checkX(\mathbf L)$ which plays a similar role, and which we employ in this section to give an interpretation to the lattice points of the generalized superpotential polytopes. Moreover we give an analogous interpretation of all of the points of our polytopes by applying the same construction with $\LL_{>0}$ replaced by $\mathbf K_{>0}$.

\vskip .2cm

Recall that we have fixed $\p_{\emptyset}=1$ on $\opencheckX$. We make the following definition.

\begin{defn}[Positive parts of $\checkX$]\label{def:pos} Recall from 
\cref{def:posGrass}
that the 
totally positive part of the Grassmannian  $\checkX$ can be defined 
as the subset of the real Grassmannian where 
the Pl\"ucker coordinates $\p_\lambda$ are positive \cite{Postnikov}.
Now let $\mathbf F$ be an infinite field and $\mathbf F_{>0}$ a  subset in $\mathbf F\setminus \{0\}$ which is closed under addition, multiplication and inverse. For example $\mathbf F=\R$ with with the positive real numbers, or $\mathbf F= \LL,\mathbf K$ with $\mathbf F_{>0}$ as in \cref{d:genPuiseux}. We define
\[
\checkX(\mathbf F_{>0})=\opencheckX(\mathbf F_{>0}):=\{x \in\opencheckX(\mathbf F)\mid \p_\lambda(x)\in\mathbf F_{>0},\ \ \lambda\in \mathcal P_{k,n}\}.
\]   
Note that for any $x \in \checkX(\mathbf K_{>0})$, all of the Pl\"ucker coordinates 
$\p_{\lambda}(x)$ are automatically nonzero, and that we have inclusions $\checkX(\R_{>0})\subset \checkX(\mathbf L_{>0})\subset \checkX(\mathbf K_{>0})$.
\end{defn}

We record that we have the standard parameterizations of the totally positive part also in this situation.

\begin{lemma}\label{lem:clusterparam}
Suppose $\Phi^\vee_G$ is an $\mathcal{A}$-cluster chart (see
\eqref{eq:clusterchart}). Suppose $\mathbf F$  and  $\mathbf F_{>0}$ are as in \cref{def:pos}.
We can consider $\Phi^\vee_G$ over the field~$\mathbf F$. In this case we have that 
\begin{equation}\label{eq:param}
\checkX(\mathbf F_{>0})=\Phi^\vee_{G}((\mathbf F_{>0})^{\mathcal P_G}),
\end{equation}
and the map $\Phi^\vee_{G}:(\mathbf F_{>0})^{\mathcal P_G}\to \checkX(\mathbf F_{>0})$ is a bijection. 
\end{lemma}

\begin{proof}
This follows in the usual way from the cluster algebra structure on the Grassmannian \cite{Scott}, 
by virtue of which 
each cluster variable can be written as a subtraction-free rational function 
in any cluster.  So in particular, if the elements of one cluster have values in $\mathbf F_{>0}$, then
so do all cluster variables. 
\end{proof}

\begin{remark}\label{r:genparam}
The right hand side of the equation \eqref{eq:param} is independent of $G$, by positivity of mutation. Note that the notion of the $\mathbf F_{>0}$-valued points extends to a general $\mathcal A$-cluster variety if we take \eqref{eq:param} as the definition in place of \cref{def:pos}.
\end{remark}

\subsection{Tropicalisation of a positive Laurent polynomial}

We record the following straightforward lemma which interprets the tropicalisation $\Trop(\mathbf h)$ of a  positive Laurent polynomial  $\mathbf h$, see \cref{d:universallypositive} and \cref{def:Tropicalisation}, in terms of the semifield $\mathbf K_{>0}$ and the valuation $\ValK$. 
See \cite[Proof of Proposition 9.4]{Lusztig3} and \cite[Proposition 2.5]{SpeyerWilliams} for
related statements.

\begin{lem}\label{l:valandTrop}
Let $\mathbf h\in\C[X_1^{\pm 1},\dotsc, X_m^{\pm 1}]$ be a {\it positive} Laurent polynomial. We may evaluate $\mathbf h$ on $(k_i)_{i=1}^m\in(\mathbf K_{>0})^{m}$. On the other hand, associated to each $k_i$ we have $y_i:=\ValK(k_i)$, so that $(y_i)_{i=1}^m\in \R^m$. Then 
\[
\Trop(\mathbf h)(y_1,\dotsc, y_m)=\ValK (\mathbf h(k_1,\dotsc, k_m)).
\]
In particular, $\ValK (\mathbf h(k_1,\dotsc, k_m))$ depends only on the valuations $y_i$ of the $k_i$. 
\end{lem}

\begin{proof}
If $\mathbf h=X_i$ then both sides agree and equal to $x_i$. Clearly any product $\mathbf h=\mathbf h_1\mathbf h_2$ gives a $\mathbf K$-valuation equal to $\ValK (\mathbf h_1(k_1,\dotsc, k_m))+\ValK (\mathbf h_2(k_1,\dotsc, k_m))$. Now let $\mathbf h=\mathbf h_1+\mathbf h_2$. Because all of the coefficients of $\mathbf h_1,\mathbf h_2$ are positive and the leading terms of the $k_i$ also have positive coefficients, there can be no cancellations when working out the valuation of the sum $(\mathbf h_1+\mathbf h_2)(k_1,\dotsc, k_m)$. This implies that the latter valuation is given by $\min(\ValK(\mathbf h_1(k_1,\dotsc, k_m)),\ValK(\mathbf h_2(k_1,\dotsc, k_m)))$. Thus the right hand side has the same properties as define the left hand side, see \cref{def:Tropicalisation}.
\end{proof}

\subsection{Tropicalisation of $\checkX$ and zones} 

We introduce a (positive) tropical version of our cluster variety $\checkX$ via an equivalence relation on 
elements 
of $\checkX(\mathbf K_{>0})$, 
analogous to Lusztig's construction of  `zones' in $U^+(\LL_{>0})$  \cite{Lusztig3}.  
This is also very close to the notion 
of positive tropical variety from \cite[Section 2]{SpeyerWilliams}.

\begin{defn}[{Zones and tropical points}] {\label{d:Rzones}} 
Let us define an equivalence relation on 
$\checkX(\mathbf K_{>0})$ by
\[
x\sim x' \quad :\iff \quad \ValK(\p_\lambda(x))=\ValK(\p_\lambda(x'))\ \text{ for all $\lambda\in \mathcal P_{k,n}$.}
\]
In other words, they are equivalent if the exponent vectors of the leading terms of all the Pl\"ucker coordinates agree.
	We write $[x]$ for the equivalence class of $x \in\checkX(\mathbf K_{>0})$ and
let 
$\RZones:=\checkX(\mathbf K_{>0})/\sim$
denote the set of equivalence classes, also called {\it tropical points} %
of $\checkX$. If a tropical point has a representative $x\in\checkX(\LL_{>0})$ then we call it a {\it zone} inspired by the terminology of Lusztig. The zones are precisely those tropical points $[x]$ for which all $\ValK(p_\lambda(x))$ lie in $\Z$.
\end{defn}

\begin{lem}\label{l:LTrop}
For any seed $\check\Sigma^\mathcal A_G$ the following map is well-defined and gives a bijection,
\begin{equation}\label{eq:varphi}
\pi_G:\RZones \to \R^{\mathcal{P}_G}, \qquad [x]\mapsto
(\ValK(\varphi_\mu(x)))_{\mu},
\end{equation}
where the $\varphi_\mu$ run over the set of cluster variables $\PCG$, and the indexing set of cluster variables is denoted $\mathcal P_G$.
\end{lem}

\begin{defn}[Tropicalized $\mathcal A$-cluster mutation]\label{d:PsiGG'}
Suppose $\check\Sigma^{\mathcal A}_G$ and $\check\Sigma^{\mathcal A}_{G'}$ are general $\mathcal{A}$-cluster seeds of type $\pi_{k,n}$ which are related by a single mutation at a vertex $\nu_i$. Let the cluster variables for $\check\Sigma^{\mathcal A}_G$ be  indexed by  
$\mathcal P_G=\{\nu_1,\dots, \nu_N\}$. Recall that $\PC(G')=\PCG\cup\{\varphi_{\nu_i'}\}\setminus \{\varphi_{\nu_i}\}$, and the $\mathcal A$-cluster mutation $\Mut^{\mathcal A}_{\nu_i}$ gives a positive Laurent polynomial expansion of the new variable $\varphi_{\nu_i'}$ in terms of $\PCG$, see~\eqref{e:AclusterMut}. We tropicalise this change of coordinates between $\PCG$ and $\PC(G')$ and denote the resulting piecewise linear map 
 by  $\Psi_{G,G'}$. 
  Explicitly, 
 $\Psi_{G,G'}:\R^{\mathcal P_G}\to \R^{\mathcal P_{G'}}$ takes $(v_{\nu_1},v_{\nu_2}, \dots,v_{\nu_N})$ to
$(v_{\nu_1}, \dots,v_{\nu_{i-1}}, v_{\nu_i'},v_{\nu_{i+1}},\dotsc, v_{\nu_N})$, where 
\begin{equation}\label{e:tropmut}
v_{\nu'_i} = \min(\sum_{\nu_j \to \nu_i} v_{\nu_j}, \sum_{\nu_i \to \nu_j} v_{\nu_j})\, - \, v_{\nu_i}, 
\end{equation}
and the sums are over arrows in the quiver $Q(G)$ pointing towards $\nu_i$ or away from $\nu_i$, respectively. We call $\Psi_{G,G'}$ a \emph {tropicalized $\mathcal A$-cluster mutation}.
\end{defn}

\begin{rem}
Note that if $G$  and $G'$ are plabic graphs related by the square move (M1) -- we can suppose we are doing the square move at $\nu_1$ in Figure~\ref{labeled-square} -- then $\Psi_{G,G'}$ is simply given by
\begin{equation*}  \label{e:tropmutG}
v_{\nu'_1} = \min(v_{\nu_2}+v_{\nu_4}, v_{\nu_3}+v_{\nu_5}) - v_{\nu_1}.
\end{equation*}
\end{rem}

\begin{lem}\label{l:piGmutation}
Suppose $G$ and $G'$ index arbitrary $\mathcal{A}$-seeds of type $\pi_{k,n}$ which are related by a single mutation at vertex $\nu_1$, where 
the cluster variables are indexed by 
$(\nu_1,\dots, \nu_N)$. Then we  have a commutative diagram 
\begin{equation}\label{e:triangle}
	\begin{tikzcd}
		&\RZones	\arrow{dl}[swap]{\pi_G} \arrow{dr}{\pi_{G'}} & \\
		\R^{\mathcal{P}_G} \arrow{rr}{\Psi_{G,G'}} &&
 \R^{\mathcal{P}_{G'}}
	\end{tikzcd}
\end{equation}
where the map along the bottom is the \emph {tropicalized $\mathcal A$-cluster mutation} $\Psi_{G,G'}$ from \cref{d:PsiGG'}.
\end{lem}

\begin{proof}[Proof of Lemmas~\ref{l:LTrop} and \ref{l:piGmutation}]
Recall that the cluster chart $\Phi_G^{\vee}$ 
from \cref{lem:clusterparam}
 gives a bijective parameterization
of $\checkX(\mathbf K_{>0})$ where the inverse 
$(\Phi_G^{\vee})^{-1}:\checkX(\mathbf K_{>0})\longrightarrow (\mathbf K_{>0})^{\mathcal P_G}$ is precisely the map $x\mapsto(\varphi_\mu(x))_\mu$. We have the following composition of surjective maps
\[ 
\operatorname{Comp}_G:\checkX(\mathbf K_{>0}) \overset{(\Phi_G^{\vee})^{-1}}\longrightarrow (\mathbf K_{>0})^{\mathcal P_G}
\overset{{\ValK}}\longrightarrow {\R}^{\mathcal P_G}.
\]

We define an equivalence relation $\sim_G$  by letting $
x\sim_Gx'$ if and only if $ 
\operatorname{Comp}_G(x)=
\operatorname{Comp}_G(x')$.
Clearly with this definition, $\operatorname{Comp}_G$ descends to a bijection $[\operatorname{Comp}_G]:\checkX(\mathbf K_{>0})/\sim_G \ \to \, \R^{\mathcal P_G}$. To prove \cref{l:LTrop} it suffices to show that the equivalence relation $\sim_G$ is independent of $G$ and recovers the original equivalence relation $\sim$ from \cref{d:Rzones}.  Then $[\operatorname{Comp}_G]=\pi_G$ and we are done. 

If $G$ and $G'$ are related by a single mutation, see \cref{Aseed}, 
then we have a commutative diagram 
\begin{equation}
\label{e:triangle0}
	\begin{tikzcd}
		&\checkX(\mathbf K_{>0})	\arrow{dl}[swap]{\operatorname{Comp}_G} \arrow{dr}{\operatorname{Comp}_{G'}} & \\
		\R^{\mathcal{P}_G} \arrow{rr}{\Psi_{G,G'}} &&
 \R^{\mathcal{P}_{G'}}
	\end{tikzcd}
\end{equation}
	where $\Psi_{G,G'}$ is the tropicalized $\mathcal{A}$-cluster mutation. This follows by an application of \cref{l:valandTrop}.
Since $\Psi_{G,G'}$ is a bijection (with inverse $\Psi_{G',G}$) it follows   that $\operatorname{Comp}_G(x)=\operatorname{Comp}_G(x')$ if and only if $\operatorname{Comp}_{G'}(x)=\operatorname{Comp}_{G'}(x')$. Thus $\sim_G$ and $\sim_{G'}$ are the same equivalence relation. Therefore the equivalence relation $\sim_G$ is independent of $G$. If $G$ is a plabic graph indexing a Pl\"ucker cluster then $x\sim x'$ implies that $x\sim_G x'$. On the other hand if $x\sim_G x'$ then also $x\sim_{G'}x'$ for any other $G'$. Therefore it follows that $\ValK(p_\lambda(x))=\ValK(p_\lambda(x'))$ for all Pl\"ucker coordinates $p_\lambda$, since every Pl\"ucker coordinate appears in {\it some} seed $\check\Sigma^{\mathcal A}_{G'}$. As a consequence $x\sim_G x'$ implies $x\sim x'$ and Lemma~\ref{l:LTrop} is proved. 

Since all of the equivalence relations $\sim_G$ are equal to $\sim$, we can factor all of the vertical maps $\operatorname{Comp}_G$ through $\sim$ and then \eqref{e:triangle0} turns into the commutative diagram of bijections which is precisely the one given in Lemma~\ref{l:piGmutation}.
\end{proof}
\begin{rem} \label{r:TropGen}
The main observation of the above proof was that
 if we consider an arbitrary $\mathcal{A}$-seed $\check\Sigma^{\mathcal A}_G$ of type $\pi_{k,n}$, then
for $x, x'\in\checkX(\mathbf K_{>0})$ we have
\begin{equation}\label{e:equivG}
x\sim x' \quad \iff\quad  \ValK(\varphi_\mu(x))=\ValK(\varphi_\mu(x'))\text{ for all cluster variables
$\varphi_\mu$ of $\check\Sigma_G^{\mathcal A}$.
}
\end{equation}
This says that equivalence of points in $\checkX(\mathbf K_{>0})$ can be checked using a single, arbitrarily chosen seed, and gives an alternative definition for the equivalence relation $\sim$.  With this alternative definition \eqref{e:equivG} of $\sim$, the definition of `tropical points' and `zones' as equivalence classes generalises to an arbitrary $\mathcal A$-cluster algebra, compare~\cref{r:genparam}. 
\end{rem}

\begin{rem} \label{r:integral}
	We note that the inverse of the tropicalized $\mathcal{A}$-cluster mutation $\Psi_{G,G'}$ is always just given by $\Psi_{G',G}$. Since both maps $\Psi_{G,G'}$ and $\Psi_{G',G}$ map integral points to integral points we have that $\Psi_{G,G'}$ restricts to a bijection $\Z^{\mathcal P_G}\to \Z^{\mathcal P_{G'}}$ and the entire diagram \eqref{e:triangle} restricts to give  
the commutative diagram of bijections,
\begin{equation}\label{e:triangleZ} 
	\begin{tikzcd}
		&\Zones	\arrow{dl}[swap]{\pi_G} \arrow{dr}{\pi_{G'}} & \\
		\Z^{\mathcal{P}_G} \arrow{rr}{\Psi_{G,G'}} &&
 \Z^{\mathcal{P}_{G'}}
	\end{tikzcd}
\end{equation}

\end{rem}

\subsection{Mutation of polytopes}
\label{sec:mutateQ}

In this section we give an interpretation of the superpotential polytopes $\Q_G^r$ from 
\cref{l:B-modelPolytope} and their generalisations~$\Q_G(r_1,\dotsc,r_n)$
from \cref{def:generalQ} in terms of $\Trop(\checkX)$.
We use this interpretation to relate the polytopes coming from different 
$\mathcal{A}$-clusters $G$.

\begin{defn}\label{d:posset}
Suppose $h\in\C[\opencheckX]$ has the property that it is {\it universally positive} for the $\mathcal A$-cluster algebra structure of $\C[\opencheckX]$, as in \cref{d:universallypositive}. 
Let $m\in\R$. In this case we define inside $\Trop(\checkX)$ the set
\[
\PosSet_{(m)}(h):=\{[x]\in\Trop(\checkX)\mid \ValK(h(x))+m\ge 0\}.
\]
For a given choice of seed $\check\Sigma_G^{\mathcal A}$ we also associate to $h$ the subset of $\R^{\mathcal P_G}$,
\[
\PosSet^G_{(m)}(h):=\{v\in\R^{\mathcal P_G}\mid \Trop(\mathbf h^G)(v)+m\ge 0\}.
\] 
\end{defn}

\begin{rem} Note that for $m=0$ the set 
$\PosSet^G_{(0)}(h)$ is a (possibly trivial) polyhedral cone described as intersection of half-spaces. Introducing the $m\in \R$ amounts to shifting the half-spaces. 
\end{rem}

\begin{lem}\label{l:PosSetTrop} Given any seed $\check\Sigma_G^{\mathcal A}$, a universally positive $h\in\C[\opencheckX]$, and any $m\in\R$,
the bijection $\pi_G:\Trop(\checkX)\to\R^{\mathcal P_G}$ from \cref{l:LTrop} restricts to give a bijection,
\[\pi_G:\PosSet_{(m)}(h)\longrightarrow 
\PosSet^G_{(m)}(h),\]
 between the sets from Definition~\ref{d:posset}, which we again denote $\pi_G$ by abuse of notation. We have the following commutative diagram of bijections
\begin{equation}\label{e:triangle2}
	\begin{tikzcd}
		&\PosSet_{(m)}(h)	\arrow{dl}[swap]{\pi_G} \arrow{dr}{\pi_{G'}} & \\
		\PosSet^G_{(m)}(h) \arrow{rr}{\Psi_{G,G'}} &&
		\PosSet^{G'}_{(m)}(h)
	\end{tikzcd}
\end{equation}
	where the map $\Psi_{G,G'}$ is the restriction of the \emph {tropicalized $\mathcal{A}$-cluster mutation} from \cref{l:piGmutation}. 
\end{lem}

\begin{proof} The set $\PosSet^G_{(m)}(h)$ in $\R^{\mathcal P_G}$ is indeed the image of $\PosSet_{(m)}(h)$ under the bijection $\pi_G$ from  \cref{l:LTrop}. This follows, since  $h$ is universally positive, from \cref{l:valandTrop}.
The rest of the lemma is immediate from \cref{l:piGmutation}. 
\end{proof}

Recall that the summands $W_i$ of the
superpotential  
are universally positive by \cref{rem:cluster}.

\begin{cor}\label{c:PolytopeMutation}
Let $r_1,\dotsc, r_n\in\R$ and choose $\check\Sigma^{\mathcal A}_G$ a general seed. The subset of $\RZones$ defined by
\[\Q(r_1,\dotsc,r_n):=\bigcap_{i=1}^n\PosSet_{(r_i)}(W_i)\] 
 is in bijection with the generalized superpotential polytope 
$\Q_G(r_1,\dotsc,r_n)=\bigcap_{i}\PosSet^G_{(r_i)}(W_i)$, by the restriction of the map $\pi_G$ from \cref{l:LTrop}. 
	Moreover if  $\check\Sigma^{\mathcal A}_{G'}$ is related to $\check\Sigma^{\mathcal A}_G$ by a cluster mutation $\Mut^{\mathcal A}_{\nu}$, then we have that the \emph {tropicalized $\mathcal{A}$-cluster mutation} $\Psi_{G,G'}$ restricts to a bijection
\[
\Psi_{G,G'}: \Q_G(r_1,\dotsc,r_n)\to \Q_{G'}(r_1,\dotsc,r_n).
\]
\end{cor}
\begin{proof}
The equality $\Q_G(r_1,\dotsc,r_n)=\bigcap_{i}\PosSet^G_{(r_i)}(W_i)$ is just an equivalent restatement of Definition~\ref{def:generalQ}. The corollary is immediate from Lemma~\ref{l:PosSetTrop}. 
\end{proof}

\begin{corollary}\label{c:intmoves}
The number of lattice points 
of  
$\Q_{G}(r_1,\dotsc, r_n)$ is independent of $G$. 
\end{corollary}
\begin{proof}
	By \cref{r:integral} and \cref{c:PolytopeMutation}, if $G'$ indexes a seed which is obtained by mutation from $G$, then the corresponding tropicalized $\mathcal{A}$-cluster mutation $\Psi_{G,G'}$ restricts to a bijection from the lattice points of $\Gamma_G(r_1,\dotsc, r_n)$ to the lattice points of $\Gamma_{G'}(r_1,\dotsc, r_n)$.
Since all seeds %
are connected by mutation, 
it follows that the number of lattice points of $\Q_G(r_1,\dotsc, r_n)$ is independent of the choice of seed. 
\end{proof}

\begin{remark}
	We note that there are also mutations of Laurent polynomials studied in \cite{ACGK}.  However,
	these are of a different type from the cluster algebra mutations we study here, although, as in our setting, there is an associated notion of mutation of polytopes (dual to Newton polytopes) by piecewise linear maps. 
\end{remark}

\section{Combinatorics of perfect matchings}\label{s:matchings}

We now return to $\X$ and study the flow polynomials which express the 
Pl\"ucker coordinates $P_\lambda$ in terms of the network chart associated to a plabic graph.
Namely, in this section we use perfect matchings to show that for any Pl\"ucker coordinate, 
the corresponding flow polynomial coming from a plabic graph always has a strongly minimal and a strongly maximal term, see \cref{def:minimal}.  

Let $G$ be a bipartite plabic graph with boundary vertices
labeled $1,2,\dots, n$.
We assume that each boundary vertex is adjacent to one white vertex and no
other vertices.  
A \emph{(perfect) matching} of $G$ is a collection of edges of $G$ which cover
each internal vertex exactly once.  For a matching $M$,
we let $\partial M \subset [n]$ denote the subset of the boundary vertices
covered by $M$.  
Given $G$, we say that $J$ is \emph{matchable} if there is at least
one matching $M$ of $G$ with boundary $\partial M=J$.

There is a partial order on matchings, which makes
the set of matchings with a fixed boundary into a \emph{distributive
lattice}, i.e. a partially ordered set in which every two elements have a unique supremum
(the join)
and a unique infimum (the meet), in which the operations of join and meet distribute
over each other.
Let $M$ be a matching of $G$ and let $\mu$ label an internal face of $G$
such that $M$ contains exactly half the edges in the boundary of $\mu$
(the most possible).  The \emph{flip} (or \emph{swivel}%
) of $M$ at $\mu$
is the matching $M'$, which contains the other half of the edges
in the boundary of $\mu$ and is otherwise the same as $M$.
Note that $M'$ uses the same boundary edges as $M$ does.
We say that the flip of $M$ at $\mu$ is a \emph{flip up} 
from $M$ to $M'$, and we write $M \lessdot M'$, if, 
when we orient the edges in the boundary of $\mu$ clockwise, the 
matched edges in $M$ go from white to black.  Otherwise we say 
that the flip is a \emph{flip down} from $M$ to $M'$, and we write
$M' \lessdot M$.  See the leftmost column of 
\cref{flip}.   We let $\leq$ denote the partial order on matchings
generated by the cover relation $\lessdot$.

\begin{figure}[h]
\includegraphics[height=1.5in]{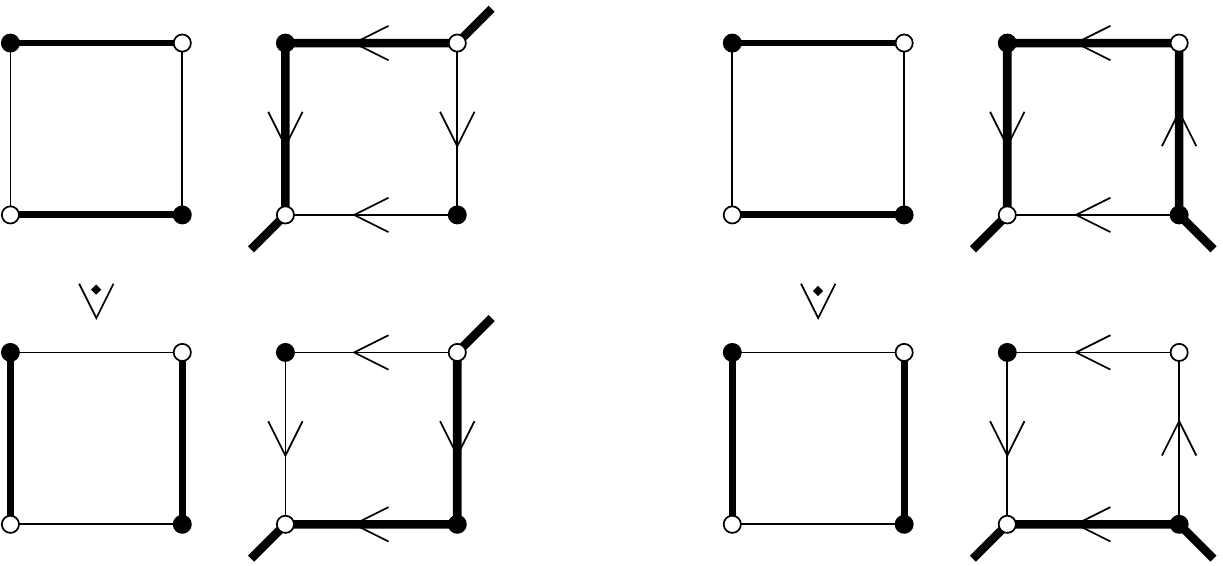}
\caption{Flipping up a face, and some examples of the effect on 
corresponding flows. }
\label{flip}
\end{figure}

The following result appears as \cite[Theorem B.1]{MullerSpeyer}
and \cite[Corollaries B.3 and B.4]{MullerSpeyer},
and is deduced from \cite[Theorem 2]{Propp}.
\begin{theorem}[{\cite[Theorem B.1]{MullerSpeyer} and \cite[Theorem 2]{Propp}}]
\label{cor:minmax}
Let $G$ be a reduced bipartite plabic graph, and let $J$
be a matchable subset of $[n]$.  Then the partial order $\leq$
makes the set of matchings on $G$ with boundary $J$ into a finite
distributive lattice, which we call $\Match^G_J$ (or $\Match^G_{\lambda}$, if $\lambda \subseteq (n-k) \times k$
is the partition corresponding to $J$).

In particular, the set $\Match^G_J$ of matchings of $G$ with boundary $J$ has a unique minimal
element $M_J^{min}$ and a unique maximal element $M_J^{max}$, 
assuming the set is nonempty.
Moreover, any two elements of $\Match^G_J$ are connected
by a sequence of flips.
\end{theorem}

\begin{definition}\label{GJ}
Given $G$ and $J$ as in \cref{cor:minmax},
we let $G(J)$ denote the subgraph of $G$ consisting of the (closure of the) faces involved
in a flip connecting elements of $\Match^G_J$.  
And if $\lambda\in \mathcal{P}_{k,n}$ is the partition with vertical steps 
$J(\lambda)$, then we also use $G(\lambda)$ to denote $G(J(\lambda))$.
\end{definition}

Note that the elements of $\Match^G_J$ can be identified with the perfect matchings of 
$G(J)$.
Our next goal is to relate matchings of $G$ to flows in a perfect orientation
of $G$.  The following lemma is easy to check;
see \cref{flow-matching}.

\begin{lemma}\label{lem:flowmatch}
Let $\mathcal{O}$ be a perfect orientation of a plabic graph $G$,
with source set $I_{\mathcal{O}}$.  Let $J$ be a set of boundary 
vertices with $|J|=|I_{\mathcal{O}}|$.  There is a bijection 
between flows $F$ from $I_{\mathcal{O}}$ to $J$, and matchings of 
$G$ with boundary $J$.  In particular, if $G$ has type $\pi_{k,n}$, then 
$|J| = n-k$.
The matching $M(F)$ associated to flow $F$
is defined by 
\begin{align*}
M(F)  = 
&\{e \ \vert \ e\notin F \text{ and }e \text{ is directed towards 
its incident white vertex in }\mathcal{O} \} \cup \\
&\{e \ \vert \ e\in F \text{ and }e \text{ is directed  away from
its incident  white vertex in }\mathcal{O} \}.
\end{align*}
We write $F(M)$ for the flow corresponding to the matching $M$.
\end{lemma}

\begin{figure}[h]
\includegraphics[height=1.2in]{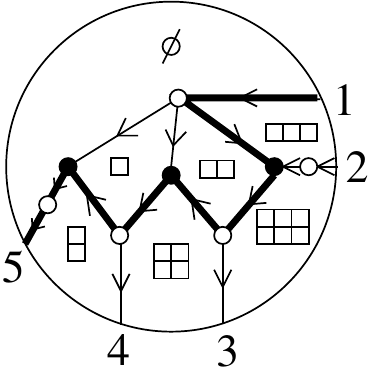} \hspace{.5cm}
\includegraphics[height=1.2in]{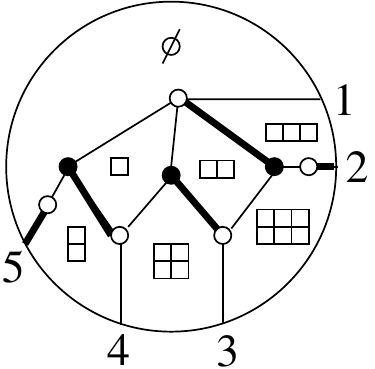}
\caption{A flow $F$ used in the flow polynomial
$P_{25}^G$ and the corresponding matching $M(F)$.  Here $F$
is the minimal flow for $P_{25}^G$ and $M(F)$ is the minimal
matching with boundary $\{2,5\}$.}
\label{flow-matching}
\end{figure}

We now use \cref{cor:minmax} to show that flow polynomials have 
strongly minimal and maximal terms. Recall the notations from Section~\ref{sec:poschart}.

\begin{corollary}
\label{prop:strongminimal}
Let $G$, $\mathcal{O}$, and $J$ be as in \cref{lem:flowmatch}.
The flow polynomial $P_J^G  = \sum \wt(F)$ has a strongly minimal
term $\Pmin_J^{G}$ such that $\Pmin_J^{G}$ divides $\wt(F)$ for 
all flows $F$ 
from $I_{\mathcal{O}}$ to $J$.  
And it has a strongly 
maximal term 
which is divisible by $\wt(F)$ %
for each flow $F$ 
from $I_{\mathcal{O}}$ to $J$.  
If $\lambda$ is the partition corresponding 
to $J$, we  also write $\Pmin_{\lambda}^{G}$ %
instead of $\Pmin_J^{G}$.
\end{corollary}

\begin{proof}
A simple case by case analysis shows that if matching $M'$ is obtained
from $M(G)$ by flipping face $\mu$ up, i.e. $M(F) \lessdot M'$,
then the flow $F':=F(M')$ satisfies
$\wt(F') = \wt(F) x_{\mu}$.  See Figure \ref{flip}.
The result now follows from 
\cref{cor:minmax},  where the strongly minimal and maximal 
terms of $P_J^G$ are the weights of the flows  
$F(M_J^{min})$ and 
$F(M_J^{max})$, respectively.
\end{proof}

\section{Mutation of Pl\"ucker coordinate valuations for $\X$}
\label{s:ConvMutation}

In this section we will again restrict our attention to plabic graphs
(as opposed to $\mathcal{X}$-clusters), and will use the combinatorics of flow
polynomials to describe explicitly how valuations
of Pl\"ucker coordinates of $\X$ behave under mutation.
This will be an important tool in 
proving \cref{t:ValuationsFormula},
 which describes all lattice points of $\Delta_G$, when $G$ is a reduced plabic graph 
of type $\pi_{k,n}$.

\begin{theorem}\label{thm1:tropcluster}
Suppose that $G$ and $G'$ are reduced plabic graphs of type $\pi_{k,n}$, 
which are related by a single move.
If $G$ and $G'$ are related by one of the moves (M2) or (M3), then $\mathcal P_G=\mathcal P_{G'}$ and the
polytopes $\conv_{G}(D)\subset \R^{\mathcal P_{G}}$ and $\conv_{G'}(D)\subset \R^{\mathcal P_{G'}}$ are identical.  
If $G$ and $G'$ are related by the square move (M1), then for any Pl\" ucker coordinate $P_K$ of $\mathbbX$, 
\[
\val_{G'}\left({P_K}\right)=\Psi_{G,G'}\left(\val_G\left({P_K}\right)\right),
\]
	for $\Psi_{G,G'}:\R^{\mathcal P_G}\to \R^{\mathcal P_{G'}}$ the \emph{tropicalized $\mathcal{A}$-cluster mutation} from 
\eqref{e:tropmutG}, and where we have written $\val_G(P_K)$ for $\val_G(P_K/P_\Max)$.
\end{theorem}

Explicitly, suppose we obtain $G$ from $G'$ by a square move at the 
face labeled by $\nu_1$ in Figure~\ref{labeled-square}.  Then any vertex $(V_{\nu_1},V_{\nu_2}, \dots,V_{\nu_N})$ of $\conv_G$, where the $\nu_i$ are the ordered elements of $\mathcal P_{G}$, without
loss of generality  starting from $\nu_1$,
 transforms to a  vertex of $\conv_{G'}$ by the following piecewise-linear  transformation 
$\Psi_{G,G'}$,
\begin{equation}\label{e:PsiV}
\Psi_{G,G'}:(V_{\nu_1},V_{\nu_2}, \dots,V_{\nu_N})\mapsto
(V_{\nu'_1},V_{\nu_2}, \dots,V_{\nu_N}), \text{ where }
\end{equation}
\begin{equation*}
V_{\nu'_1} = \min(V_{\nu_2}+V_{\nu_4}, V_{\nu_3}+V_{\nu_5}) - V_{\nu_1}.
\end{equation*}

\begin{rem}{\label{r:Pluckerprods}} 
We note that 
a statement analogous to 
\cref{thm1:tropcluster}
fails already for the products $P_K P_J$, because while 
\[
\val_G\left(P_K P_J\right)=\val_G\left(P_K\right)+\val_G\left(P_J\right),
\] 
and $\psi_{G,G'}\left(\val_G\left(P_I\right)\right)=\val_{G'}\left(P_I\right)$ for $I=J, K$, by Theorem~\ref{thm1:tropcluster}, we potentially have 
\[
\Psi_{G,G'}\left(\val_G\left(P_K P_J\right)\right)=\Psi_{G,G'}\left(\val_G\left(P_K\right)+\val_G\left(P_J\right)\right )\not=
\val_{G'}\left(P_K\right)+\val_{G'}\left(P_J\right)=\val_{G'}\left(P_K P_J\right),
\] 
	since the tropicalized $\mathcal{A}$-cluster mutations are not linear. 
\end{rem}

\begin{rem}\label{r:polmut} 
Note that while 
 $\Psi_{G,G'}$ sends the lattice points of 
 $\conv_G$
to the lattice points of $\conv_{G'}$, it does not in general 
 send the whole polytope $\conv_G$
	to the polytope $\conv_{G'}$.  Namely, since $\Psi_{G,G'}$ is only piecewise linear, 
	it could map a (non-integral) point of $\conv_G$
	 to a (non-integral) point which does not lie in $\conv_{G'}$.
However,  
the Newton-Okounkov body $\Delta_G$, which can be larger than $\conv_G$
	(recall \cref{sec:Milena}), is in fact sent to $\Delta_{G'}$ by $\Psi_{G,G'}$,  by  
\cref{thm:main} and 
\cref{c:PolytopeMutation}.
	This property
	 is highly nontrivial in light of \cref{r:Pluckerprods}.
\end{rem}

\begin{proof}[Proof of Theorem~\ref{thm1:tropcluster}]
By \cref{PSW-lemma} and \cref{acyclic}, we have an acyclic perfect orientation $\O$ of $G$ whose set of boundary
sources is $\{1,2,\dots,n-k\}$.   Therefore if we apply Theorem \ref{thm:Talaska},
our expression for the Pl\"ucker coordinate 
$P_{\Max}$ is $1$.  Moreover, we have expressions for the 
other Pl\"ucker coordinates $P_K = P_K^G$ as flow polynomials, which 
are sums over
pairwise-disjoint collections of self-avoiding walks in $\O$.  The weight of each walk is the product 
of parameters $x_{\mu}$, where $\mu$ ranges over all face labels to the left of a walk.

It is easy to see that 
the flow polynomials $P_K^G$ and $P_K^{G'}$  are equal if $G$ and $G'$ differ by one of the moves (M2) or (M3):
in either case, there is an obvious bijection between perfect orientations of both graphs involved
in the move, and this bijection is weight-preserving.

Now suppose that $G$ and $G'$ differ by a square move.  By Lemma \ref{PSW-lemma}, it suffices to compare
perfect orientations $\O$ and $\O'$ of $G$ and $G'$ which differ as in Figure \ref{oriented-square-move}. 
Without loss of generality, $G$ and $G'$ are at the left and right, respectively, of Figure \ref{oriented-square-move}.
(We should also consider the case that $G$ is at the right and $G'$ is at the left, but the proof in this 
case is analogous.)
Recall that by \cref{prop:strongminimal}, each flow polynomial $P_K$ has a strongly minimal flow
(see \cref{def:minimal}) $F_{min}$, and hence $\val_G(P_K) = \wt(F_{min})$.
The main step of the proof is to 
prove the following claim about 
how strongly minimal flows  change under an oriented  square move.
\vspace{.2cm}

{\bf Claim.} 
{\it Let $G$ and $\mathcal O$ be as above, 
let $K$ be an $(n-k)$-element subset of $\{1,\dotsc,n\}$, and  
let $F_{\min}$ be the strongly minimal flow from $\{1,\dotsc, n-k\}$ to $K$. 
\begin{enumerate}
\item 
Assuming the orientations in $\mathcal O$ locally around the face $\nu_1$ are as shown in the left-hand side of Figure~\ref{oriented-square-move}, then the restriction of $F_{\min}$ 
to the neighborhood of face $\nu_1$ is as in the left-hand
side of one of the six pictures in Figure~\ref{check-oriented-square}, say picture $I$, where 
$I\in \{\mathbf{A}, \mathbf{B}, \mathbf{C}, \mathbf{D}, \mathbf{E}, \mathbf{F}\}$.
\item If we let $F'_{\min}$ denote the flow obtained from $F_{\min}$ by the 
local transformation indicated in picture $I$, then $F'_{\min}$ is strongly minimal.
\end{enumerate}}

\begin{figure}[h]
\centering
\includegraphics[height=1.8in]{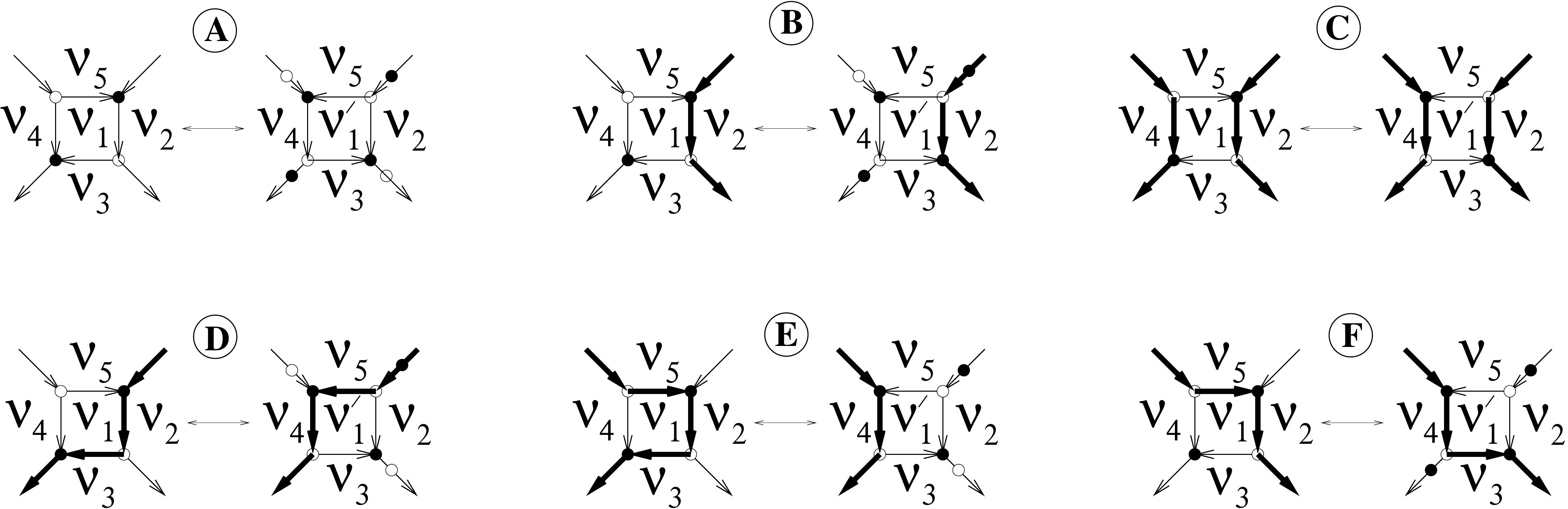}
\caption{How minimal flows change in the neighborhood of face $\nu_1$ 
as we do an oriented square move.  The perfect orientations $\O$
and $\O'$ 
for $G$ and $G'$ are shown at the left and right of each pair, respectively.
Note that in the top row, the flows do not change, but in the bottom row they do.
Also note that the picture at the top left indicates the case that the flow is not incident to 
face $\nu_1$.}
\label{check-oriented-square}
\end{figure}

Let us check (1). In theory, 
the restriction of $F_{\min}$ to the neighborhood of face $\nu_1$ could be 
as in the left-hand
side of any of the six pictures from Figure~\ref{check-oriented-square}, \emph{or} it could be 
as in Figure \ref{fig:excluded}.
However, if a flow locally looks like Figure \ref{fig:excluded}, then it cannot be minimal -- 
the single path shown in Figure \ref{fig:excluded}  could be deformed  to go around the other side of 
the face labeled $\nu_1$, and that would result in a smaller weight.  
More specifically, the weight of a flow which locally looks like Figure~\ref{fig:excluded}, when restricted
to coordinates $(\nu_1,\nu_2, \nu_3,\nu_4, \nu_5)$, has valuation 
$(i+1,i+1,i+1,i,i+1)$, whereas the weight of its deformed version has valuation
$(i,i+1,i+1,i,i+1)$, for some nonnegative integer $i$. 
This proves the first statement of the claim.

\begin{figure}[h]
\centering
\includegraphics[height=.7in]{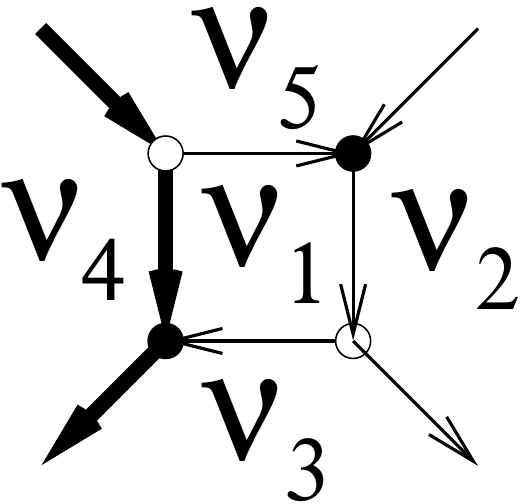}
\caption{A path whose weight is not minimal.}
\label{fig:excluded}
\end{figure}

Now let us write $\wt(F_{\min})  = \prod_{\mu\in \mathcal P_G} x_\mu^{a_\mu}$, 
so that $(a_{\mu})_{\mu\in\mathcal P_G}=\val_G(P_I)$.   Suppose that the restriction of $F_{\min}$ to the 
neighborhood of face $\nu_1$ looks as in picture $I$ of Figure~\ref{check-oriented-square}. 
Let $F'_{\min}$ be the flow in $G'$ obtained from $F_{\min}$ by the local transformation indicated in picture $I$, and 
write $\wt(F'_{\min})  = \prod_{\mu\in \mathcal P_{G'}} x^{a'_\mu}$.  (Clearly $F'_{\min}$ is indeed a flow
in $G'$.)
We need to show that $F'_{\min}$ is strongly minimal.

Let $F'$ be some arbitrary flow in $G'$, and write 
$(b'_{\mu})_{\mu\in \mathcal{P}_G'} = \val_{G'}(\wt(F'))$.
We need to show that $a'_{\mu} \leq b'_{\mu}$ for all $\mu \in \mathcal{P}_{G'}$.
We can assume that the restriction of $F'$ to the neighborhood of 
face $\nu'_1$ looks as in the right hand side of one of the six pictures in Figure 
\ref{check-oriented-square}, say picture $J$.
A priori there is one more case (obtained from the right hand side of picture $\BBB$ by 
deforming the single path to go around $\nu'_1$), but since this increases $b'_{\nu'_1}$,
we don't need to consider it.
Now let $F$ be the flow in $G$ obtained from $F'$ by the local transformation indicated in picture $J$, 
and write $(b_{\mu})_{\mu\in \mathcal{P}_G} = \val_{G}(\wt(F))$.

We already know, by our assumption on $G$, that $b_\mu\ge a_\mu$ for all $\mu\in\mathcal P_G$. 
 Moreover it is clear from Figure~\ref{check-oriented-square} that  
 \begin{equation}\label{e:transformations}
a'_\mu=\begin{cases}a_\mu & \text{ if }\mu\ne \nu'_1\\ a_{\nu_1}\ \text{ or }\ a_{\nu_1}+ 1 & \text{ if } \mu=\nu'_1,\end{cases} \quad\text{and}\quad
b'_\mu=\begin{cases}b_\mu & \text{ if }\mu\ne \nu'_1\\ b_{\nu_1}\ \text{ or }\ b_{\nu_1}+ 1 & \text{ if } \mu=\nu'_1.\end{cases}
\end{equation}
More specifically, $a'_{\nu_1} = a_{\nu_1}+1$ (respectively, $b'_{\nu_1} = b_{\nu_1}+1$) precisely when picture $I$
(respectively, picture $J$)
is one of the cases $\DDD, \EEE, \FFF$
from Figure~\ref{check-oriented-square}.

From the cases above, it follows that $b'_\mu\ge a'_\mu$ for all $\mu\ne \nu'_1$ and $\mu\in \mathcal P_{G'}$. 
We need to check only that $b'_{\nu'_1}\ge a'_{\nu'_1}$.
Since $b_{\nu_1} \geq a_{\nu_1}$, the only way to get $b'_{\nu'_1} < a'_{\nu'_1}$ is if 
$a'_{\nu_1} = a_{\nu_1}+1$ and $b'_{\nu_1} = b_{\nu_1} = a_{\nu_1}$. In particular then 
$I \in \{\DDD, \EEE, \FFF\}$ and $J \in \{\AAA, \BBB, \CCC\}$.
So we need to show that each of these nine cases is impossible when $b_{\mu}\ge a_\mu$ and $b_{\nu_1}=a_{\nu_1}$.

Let us set $i = a_{\nu_1} = b_{\nu_1}.$
If $I = \DDD$, then
the vector $(a_{\nu_1}, a_{\nu_2}, a_{\nu_3}, a_{\nu_4}, a_{\nu_5})$ has the form 
$(i, i+1, i+1, i, i)$.
If $I = \EEE$, 
the vector $(a_{\nu_1}, a_{\nu_2}, a_{\nu_3}, a_{\nu_4}, a_{\nu_5})$  has the form 
$(i, i+1, i+1, i, i+1)$. 
And if 
 $I = \FFF$,
the vector $(a_{\nu_1}, a_{\nu_2}, a_{\nu_3}, a_{\nu_4}, a_{\nu_5})$  has the form 
$(i, i+1, i, i, i+1)$. 

Meanwhile, 
if $J = \AAA$, 
then 
$(b_{\nu_1}, b_{\nu_2}, b_{\nu_3}, b_{\nu_4}, b_{\nu_5}) = (i, i, i, i, i)$.  
If $J = \BBB$, then
$(b_{\nu_1}, b_{\nu_2}, b_{\nu_3}, b_{\nu_4}, b_{\nu_5}) = (i, i+1, i, i, i)$.  
And if $J = \CCC$, 
then  
$(b_{\nu_1}, b_{\nu_2}, b_{\nu_3}, b_{\nu_4}, b_{\nu_5}) = (i, i+1, i, i-1, i)$.  

In all nine cases, we see that we get a contradiction to the fact that 
$a_{\mu} \leq b_{\mu}$ for all $\mu$. To be precise, by looking at cases $\AAA$, $\BBB$ and $\CCC$ we see that always $b_{\nu_3}=b_{\nu_5}=i$, while for $a_\mu$ we always have either $a_{\nu_3}=i+1$ or $a_{\nu_5}=i+1$, looking at $\DDD$, $\EEE$ and $\FFF$.
This completes the proof of the claim.

Now it remains to check that the tropicalized $\mathcal{A}$-cluster relation
\eqref{e:PsiV} is satisfied for each of the six cases shown in Figure \ref{check-oriented-square}.
For example, in the top-middle pair shown in Figure \ref{check-oriented-square}, we have
$a_{\nu_1} = a_{\nu_3} = a_{\nu_4} = a_{\nu_5} = a'_{\nu'_1} = i$, and $a_{\nu_2} = i+1$.
Clearly we have $a_{\nu_1}+a'_{\nu'_1} = \min(a_{\nu_2}+a_{\nu_4}, a_{\nu_3}+a_{\nu_5}).$  
In the top-right pair, we have 
$a_{\nu_2}=i+2$, $a_{\nu_1} = a_{\nu_3} = a_{\nu_5} = i+1$, $a_{\nu_4} = i$, and $a'_{\nu'_1} = i+1$,
which again satisfy \eqref{e:PsiV}.  The other three cases can be similarly checked.
This completes the proof of Theorem~\ref{thm1:tropcluster}.
\end{proof}

\section{Pl\"ucker coordinate valuations in terms of the rectangles network chart}

In this section we work with a very special choice of plabic graph, namely the plabic graph
$G=G^{\rect}_{k,n}$ from Section~\ref{sec:Grectangles}, and we 
 provide an explicit formula
 in \cref{p:rect-val}
for the valuations 
 $\val_G(P_\lambda)$ of the Pl\"ucker coordinates.
In order to describe the  $\val_G(P_\lambda)$ we introduce the notion of 
 \emph{GT tableau}.
The name GT tableau comes from  the connection with Gelfand-Tsetlin polytopes, see 
\cref{lem:integralGT} and 
\cref{lem:GT}.

\begin{figure}[h]
\centering
\includegraphics[height=1in]{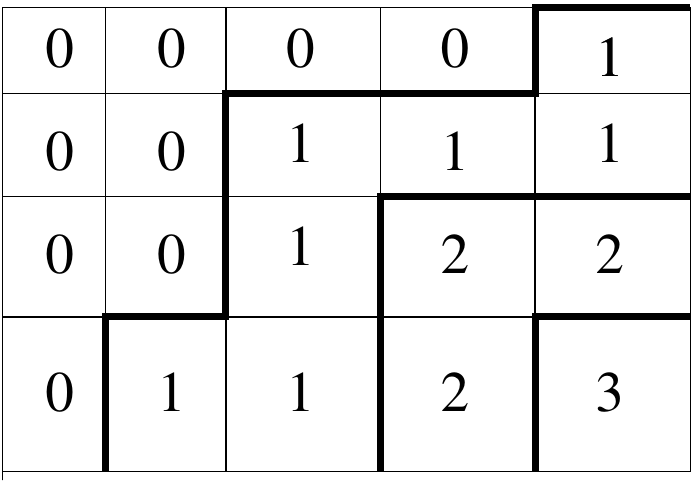} \hspace{2cm}
\includegraphics[height=1in]{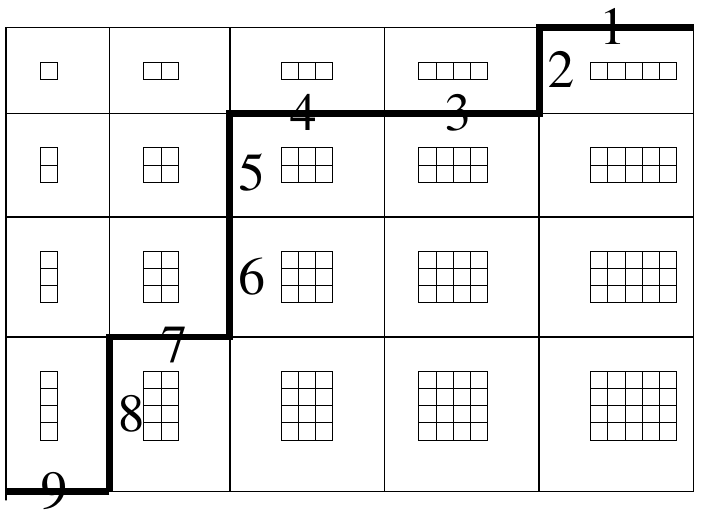}
\caption{At the left: an example of a GT tableau $\{V_{i\x j}\}$.  At the right: the labeling of the entries of the grid by rectangles,
  along with an associated lattice path determining the $J$ such that   
$\val_G(P_J)=(V_{i\x j})$  (in this case $J=\{2,5,6,8\}$).} 
\label{fig:tableauminflow}
\end{figure}
\begin{definition}\label{GTtableau}
We define a \emph{GT tableau} to be a rectangular array of integers $\{V_{i \times j}\}$ in a grid
(where $i \times j$ ranges over the nonempty rectangles
contained in $\Shkn$), which satisfy the following properties:
\begin{enumerate}
\item Entries in the top row and leftmost column are at most $1$.\label{enum:1}
\item $V_{i \times j} \leq V_{(i-1) \times (j-1)} + 1$.\label{enum:2}
\item $V_{1 \times 1} \geq 0$. \label{enum:3}
\item Entries weakly increase from left to right 
in the rows, and from top to bottom in the columns. \label{enum:4} 
\item If $V_{i \times j} > 0$, then $V_{(i+1) \times (j+1)} = V_{i \times j}+1$. \label{enum:5}
\end{enumerate}
See the left hand side of Figure \ref{fig:tableauminflow} for an example.
\end{definition}

Note that the plabic graph $G_{k,n}^{\rect}$ has 
a simple perfect orientation $\O^{\rect}$, 
which is shown in Figure~\ref{rectangles-oriented}.  The source set is 
$\{1,2,\dots,n-k\}$.

\begin{lemma}\label{lem:vertices}
Let $G = G_{k,n}^{\rect}$, and 
 choose the perfect orientation $\O^{\rect}$ of $G_{k,n}^{\rect}$ 
from Figure \ref{rectangles-oriented}.  
Each valuation
$\val_G(P_J)=(V_{i\x j})$ for $J\in{{[n]}\choose n-k}$ determines a GT tableau $\{V_{i\x j}\}$. Conversely, each GT tableau $\{V_{i\x j}\}$ arises from
the valuation of a uniquely determined P\"ucker coordinate $P_J$.
For the GT tableau $\{0\}$ the corresponding Pl\"ucker coordinate is
 $P_{\max}$.  
For any other GT tableau the 
Pl\"ucker coordinate $P_J$ is found by considering the lattice path starting at the upper right hand corner of the diagram which separates the zero and nonzero entries of the GT tableau,
and then reading off $J$ from the vertical labels of this 
lattice path, as illustrated in Figure~\ref{fig:tableauminflow}.
In particular, for $J \neq J'$,
$\val_G(P_J)$ 
and $\val_G(P_{J'})$  are distinct.
\end{lemma}


\begin{figure}[h]
\centering
\includegraphics[height=2in]{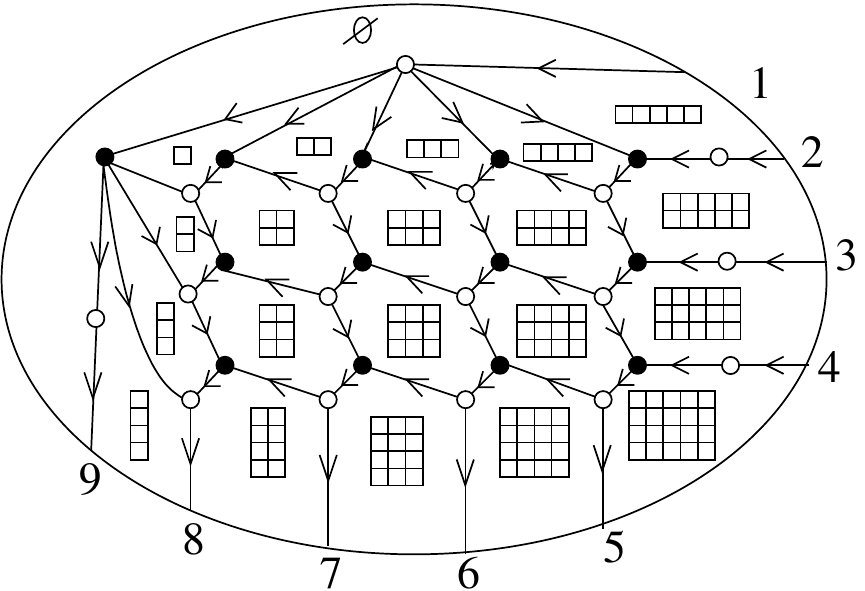}
\caption{A perfect orientation $\O^{\rect}$ of the reduced plabic graph 
$G_{5,9}^{\rect}$.  Note that the source set 
$I_{\O^{\rect}} = \{1,2,3,4\}$.  There is an obvious generalization
of $\O^{\rect}$ to any $G_{k,n}^{\rect}$, which has source set
$\{1,2,\dots,n-k\}$.}
\label{rectangles-oriented}
\end{figure}

\begin{proof}
Since the source set is $I_{\O^{\rect}} = \{1,2,\dots,n-k\}$,
and the perfect orientation is acyclic, it follows that 
the flow polynomial 
$P_{\Max}^G$ equals $1$.  Choose an arbitrary 
total order on the parameters $x_{\mu} \in \TBG$.

Recall that each flow polynomial $P_J^G$ (which can be identified with a Pl\"ucker coordinate) 
is a sum over flows
from $I_{\O^{\rect}} = \{1,2,\dots, n-k\}$ to $J$.  Since 
$\O^{\rect}$ is acyclic, each flow is just a collection 
of pairwise vertex-disjoint walks from $\{1,2,\dots,n-k\} \setminus J$
to $J \setminus \{1,2,\dots, n-k\}$
in $\O^{\rect}$.  Note that if we write 
$\{1,2,\dots, n-k\} \setminus J = 
\{i_1 > i_2 > \dots > i_{\ell}\}$ and write 
$J \setminus \{1,2,\dots,n-k\} =
\{j_1 < j_2 < \dots < j_{\ell}\}$, then any such flow 
must consist of $\ell$ paths which connect $i_1$ to $j_1$, $i_2$ to $j_2$,
\dots, and $i_{\ell}$ to $j_{\ell}$.
For example, 
in Figure \ref{rectangles-oriented},
any flow used to compute
$P_{\{2,5,6,8\}}$ must consist of three paths which connect
$4$ to $5$,
$3$ to $6$, and 
$1$ to $8$.

Recall that the weight $\wt(q)$ of a path $q$ is the product of the parameters
$x_{\mu}$ where $\mu$ ranges over all face labels to the left 
of the path.  Because of how the faces of $G_{k,n}^{\rect}$ are
arranged in a grid, we can define a partial order on the set of 
all paths from 
a given boundary source $i$ to a given boundary sink $j$, 
with $q_1 \leq q_2$ if and only if $\wt(q_1) \leq \wt(q_2)$.  
In particular, among such paths, there is a unique \emph{minimal} path,
which ``hugs" the southeast border of $G_{k,n}^{\rect}$.

It is now clear that 
	the strongly minimal flow $F_J$
	(whose existence is asserted by \cref{prop:strongminimal})
from $\{1,2,\dots,n-k\} \setminus J$
to $J \setminus \{1,2,\dots, n-k\}$ is obtained
by:
\begin{itemize}
\item choosing the minimal path $q_1$ in $\O^{\rect}$ from 
$i_1$ to $j_1$;
\item choosing the minimal path $q_2$ in $\O^{\rect}$ from 
$i_2$ to $j_2$ which is vertex-disjoint from $q_1$;
\item \dots 
\item choosing the minimal path $q_{\ell}$ in $\O^{\rect}$ from
$i_{\ell}$ to $j_{\ell}$ which is vertex-disjoint from 
$q_{\ell-1}$.
\end{itemize} 
For example, when $J=\{2,5,6,8\}$, the 
strongly minimal flow $F_J$ associated to $J$ is shown at the left of Figure \ref{fig:minflow}.
At the right of Figure \ref{fig:minflow} we have re-drawn the plabic graph to emphasize
the grid structure; this makes the structure of a strongly minimal flow even more transparent.
\begin{figure}[h]
\centering
\includegraphics[height=2in]{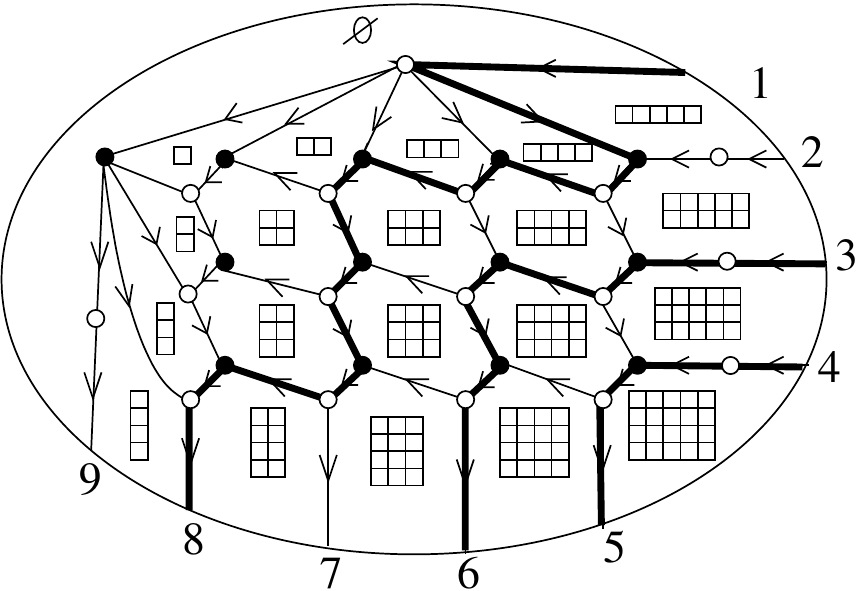} \hspace{.2cm} 
\includegraphics[height=1.8in]{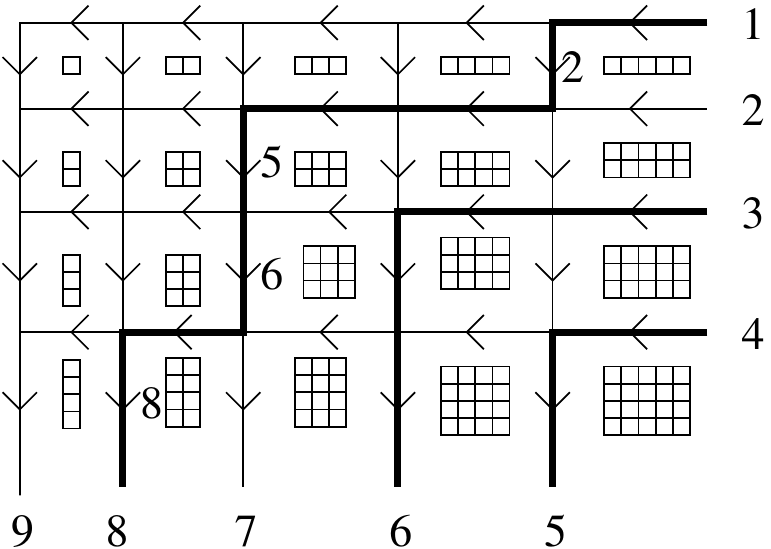}
\caption{The strongly minimal flow associated to 
$J=\{2,5,6,8\}$.  The associated GT tableau encoding the valuation is shown in Figure
\ref{fig:tableauminflow}.}
\label{fig:minflow}
\end{figure}

	Suppose that $\wt(F_J)= \prod x_{i \times j}^{V_{i \times j}}$  and consider the rectangular array $\{V_{i \times j}\}$ 
	thereby associated to $F_J$.
This rectangular array $\{V_{i \times j}\}$ encodes $\val_G(P_J)=(V_{i \times j})$ since $F_J$ is the minimal flow. 
Moreover, since $F_J$ is a flow, $\{V_{i \times j}\}$  satisfies \eqref{enum:1}, \eqref{enum:2}, \eqref{enum:3}, and \eqref{enum:4}
in \cref{GTtableau},
and since it is  minimal, it satisfies \eqref{enum:5}.  In other words, $\{V_{i \times j}\}$ is a GT tableau.
In the other direction, if one starts with a GT tableau, one may partition the set of boxes into regions based on the value of their 
entries $V_{i\times j}$.  The lattice paths separating these regions give rise to a collection of non-intersecting paths which comprise
a minimal flow, 
see the left hand side of \cref{fig:tableauminflow}.

Moreover, if one labels the steps of the northwest-most lattice path in $F_J$ by natural numbers starting from the label of its source  (equal to $1$ in the shown example),
then there is a correspondence between the labels of the vertical steps and the destination set of the flow (namely, $J$),
see the right hand side of \cref{fig:minflow}.  
In particular, the vertical step labeled $j$ can be connected to the edge of the grid incident
to $j\in J$ by a line of slope $-1$. Thus also $J$ is  determined by the valuation $\val_G(P_J)$.
\end{proof}

\begin{definition}\label{def:maxdiag}
Given two partitions $\lambda$ and $\mu$ in $\mathcal{P}_{k,n}$,
we let $\mu \setminus \lambda$ denote the corresponding 
	\emph{skew diagram}, i.e. the set of boxes
remaining if we justify both $\mu$ and $\lambda$ at the top-left
of a $(n-k) \times k$ rectangle, then remove from $\mu$ any boxes
that are in $\lambda$.
We let $\maxdiag (\mu \setminus \lambda)$ denote the maximum
number of boxes in $\mu \setminus \lambda$ that lie along
any diagonal (with slope $-1$) of the rectangle.
\end{definition}

\begin{proposition}\label{p:rect-val}
Let $G = G_{k,n}^{\rect}$ be the plabic graph 
defined in 
\cref{sec:Grectangles}
(see \cref{G-rectangles}).  Then 
\begin{equation*}
\val_G(P_{\lambda})_{i \times j} = \maxdiag (i \times j \setminus \lambda).
\end{equation*}
\end{proposition}

Before proving \cref{p:rect-val}, we make several simple 
observations about the relationships between 
the faces of $G_{k,n}^{\rect}$, partitions, and strongly minimal flows.

\begin{remark}\label{rem:ij}
Consider an $(n-k)$ by $k$ rectangle $R$, with boxes labeled by rectangular Young diagrams
as in the right of \cref{fig:minflow} (for $k=5$ and $n=9$).  Then if we place an $i$ by $j$ rectangle
justified to the northwest of $R$, the region in its southeast corner will be 
labeled by the Young diagram $i \times j$.
\end{remark}

\begin{remark}\label{rem:toppath}
Let $I \mapsto \lambda(I)$
denote the bijection
from Section \ref{Young}
between $(n-k)$-element subsets of $[n]$
and elements of $\mathcal{P}_{k,n}$.  
Then the topmost path in the strongly minimal flow for $P_I$ cuts out 
the southeast border of $\lambda(I)$. For example, the right hand side of \cref{fig:minflow} shows the strongly minimal flow
for $P_{\{2,5,6,8\}}$.  Note that the topmost path in the flow
cuts out the partition $(4,2,2,1)$, which is the partition associated
to $\{2,5,6,8\}$.
This observation is already implicit in the proof of \cref{lem:vertices}.  Namely 
if one starts by labeling the vertical steps of the partition
cut out by the topmost path in the strongly minimal flow (as is done
in \cref{fig:minflow}) and then propagates each label southeast
as far as possible, each label will end up on an edge incident
to some destination $i\in I$ for the flow $P_I$.
\end{remark}

\begin{proof}[Proof of \cref{p:rect-val}]
To compute $\val_G(P_{\lambda})$ we use the strongly minimal flow
for $P_{\lambda}$ in $G = G_{k,n}^{\rect}$, which by 
\cref{rem:toppath} cuts out the partition $\lambda$,
see \cref{fig:ValLemma}.
To compute the $i \times j$ component in $\val_G(P_{\lambda})$,
we need to compute the number of paths of the flow that are
above the box $b$ which is labeled by the  partition $i \times j$.
By \cref{rem:ij}, this box is the southeast-most box 
in the $i$ by $ j$ rectangle indicated in \cref{fig:ValLemma}.
The boxes of a maximal diagonal are in bijection with these paths of the flow above $b$,
	so the number of paths we are trying to 
compute is precisely $\maxdiag(i \times j \setminus \lambda)$.
\begin{figure}[h]
\centering
\includegraphics[height=1.5in]{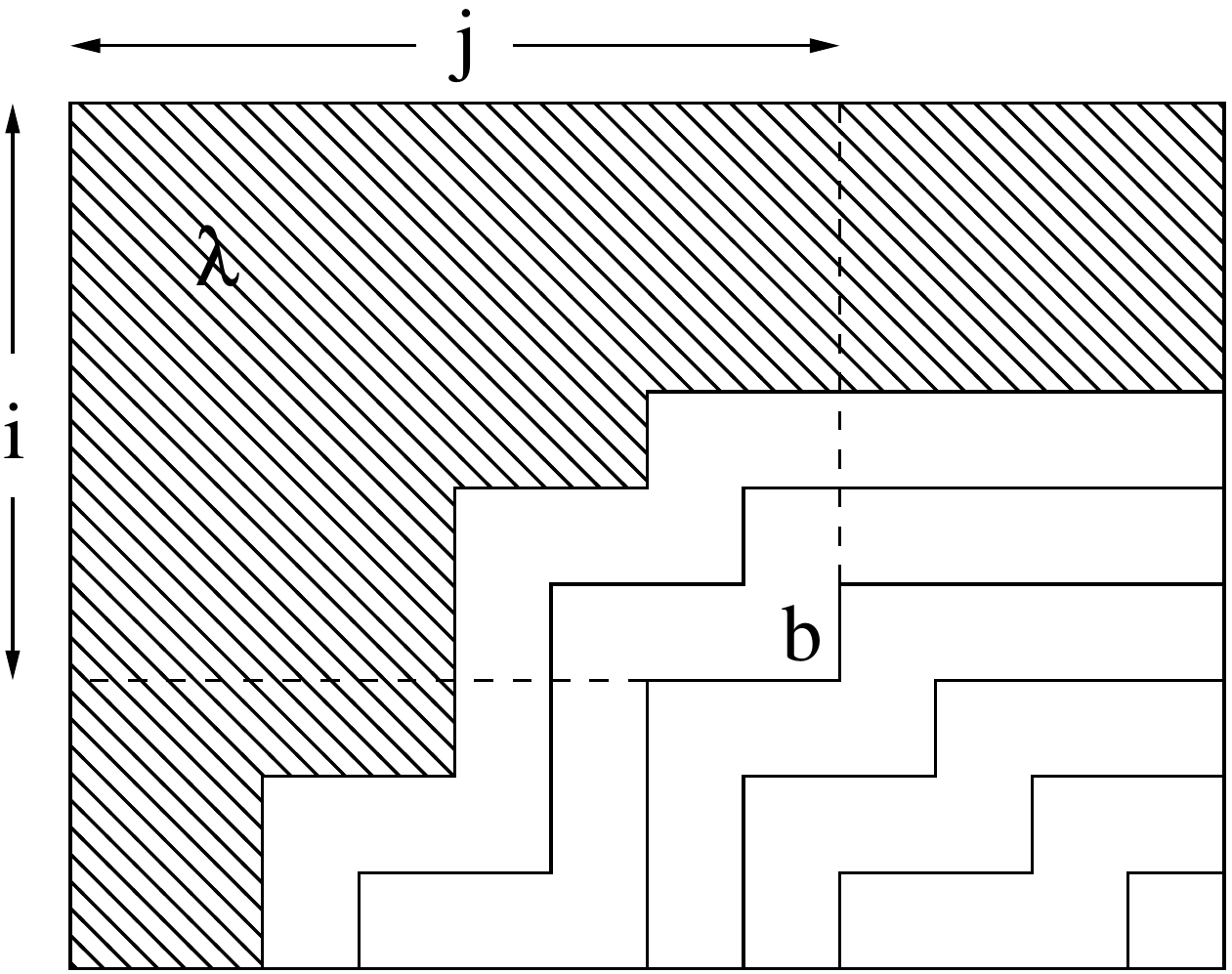}
\caption{}
\label{fig:ValLemma}
\end{figure}
\end{proof}

\section{A Young diagram formula for Pl\"ucker coordinate valuations}
\label{s:Pluckervals}

In this section we prove the general \cref{t:ValuationsFormula},
which gives an explicit formula 
for all leading terms of
 flow polynomials $P_{\lambda}^G$, that is, the valuations 
$\val_G(P_{\lambda})$, when $G$ is a reduced plabic graph of type $\pi_{k,n}$. We then use \cref{t:ValuationsFormula} to give explicit formulas for 
Pl\"ucker coordinates corresponding to frozen variables, see \cref{s:frozen}.
Comparing with a result of 
Fulton and Woodward \cite{FW} (which was refined in the Grassmannian
setting by \cite{PostnikovDuke}) we find that the right-hand side of 
our formula has an interpretation in terms of the 
quantum multiplication in the quantum cohomology of the Grassmannian.

\subsection{Valuations of Pl\"ucker coordinates}

\begin{theorem}\label{t:ValuationsFormula}
Let $G$ be any reduced plabic graph of type $\pi_{k,n}$ and $\lambda\in\mathcal P_{k,n}$.
For any partition $\mu\in \mathcal{P}_G$, 
\begin{equation*}
\val_G(P_{\lambda})_{\mu} = \maxdiag (\mu \setminus \lambda),
\end{equation*}
where 
$\maxdiag (\mu \setminus \lambda)$ is as in \cref{def:maxdiag}. 
\end{theorem}

\begin{remark}
By \cite{FW}, $\maxdiag (\mu \setminus \lambda)$ is equal to the 
smallest degree $d$ such that $q^d$ appears in the quantum product of two Schubert
classes $\sigma_{\mu}\star \sigma_{\lambda^c}
$ in the quantum cohomology ring  $QH^*(Gr_k(\C^n))$, when this 
product is expanded in the Schubert basis.
See also \cite{Yong} and \cite{PostnikovDuke}. Here $\sigma_{\lambda^c}$ is the Poincar\'e dual Schubert class to $\sigma_\lambda$, compare \cref{rem:toppathdual}. 
\end{remark}

Note that \cref{p:rect-val} is precisely 
 \cref{t:ValuationsFormula} in the special case of the rectangles cluster. We prove the theorem in general by explicitly constructing an element in $\Trop(\checkX)$, %
 which we think of as associated to $P_\lambda$ by mirror symmetry. 

\begin{theorem}\label{t:matrix}
Fix $\lambda\in\mathcal P_{k,n}$. There exists an element
$x_{\lambda}(t) \in \opencheckX(\mathbf K_{>0})$ such that 
for any partition $\mu$,
\begin{equation*}
\ValK(p_{\mu}(x_{\lambda}(t))  = \maxdiag (\mu \setminus \lambda).
\end{equation*}
\end{theorem}

\begin{definition} Consider $\lambda\in\mathcal P_{k,n}$, that is a Young diagram fitting into a $(n-k)\x k$ rectangle. We {\it transpose} $\lambda$, which means we reflect $\lambda$ along the $-1$ diagonal, to obtain a Young diagram
  which fits into a $k\x (n-k)$ rectangle. 
	
	In order to define an element 
$x_{\lambda}(t)
 \in \opencheckX(\mathbf K_{>0})$, 
we use the network parameterization - on the $\checkX$ side - associated to the 
grid like the one shown at the left of Figure \ref{fig:matrixval}.  All edges
are directed left and down, but there are now $k$ rows and $n-k$ columns in the grid.  We make specific
choices for network parameters labeling the regions, as follows. We transpose $\lambda$ to fit into the $k\x (n-k)$ grid and rotate it,  placing it in the southeast corner of the grid. Note that the interior of $\lambda$ is now southeast of its boundary path.  
Then the boxes immediately northwest of inner and outer corners
of the boundary of $\lambda$ are filled with $t$ and $t^{-1}$, respectively. 
All other boxes receive the parameter $1$.  
This gives rise
to an element 
${x}_{\lambda}(t)
 \in \opencheckX(\mathbf K_{>0})$, whose Pl\"ucker coordinates 
are computed as sums over flows, as in \cref{thm:Talaska}.
\end{definition}

\begin{rem} The element $x_{\lambda}(t)$ determines a zone or integral point $[x_{\lambda}(t)]$ in $\Trop(\checkX)$, see \cref{d:Rzones}. %
 Indeed $x_{\lambda}(t)$ is constructed to lie in $\checkX(\LL_{>0})$.
\end{rem}

\begin{remark}\label{rem:toppathdual}
Given a partition  $\mu \in \mathcal{P}_{k,n}$,
let $\mu^c$ denote the Young diagram which is the complement of $\mu$ in the 
$(n-k)$ by $ k$ rectangle rotated by $180^\circ$.
For $\checkX$ with its analogous associated network, and  any $J\subset {[n]\choose k}$ interpreted as set of west steps of a partition $\mu=\mu(J)$, we have the following version of \cref{rem:toppath}. The
topmost path in the strongly minimal flow for $p_J$ cuts out 
the southeast border of the transpose of $\mu(J)^c$. See the right hand side of Figure~\ref{fig:matrixval} for an example.
\end{remark}

\begin{figure}[h]
\centering
\includegraphics[height=1.8in]{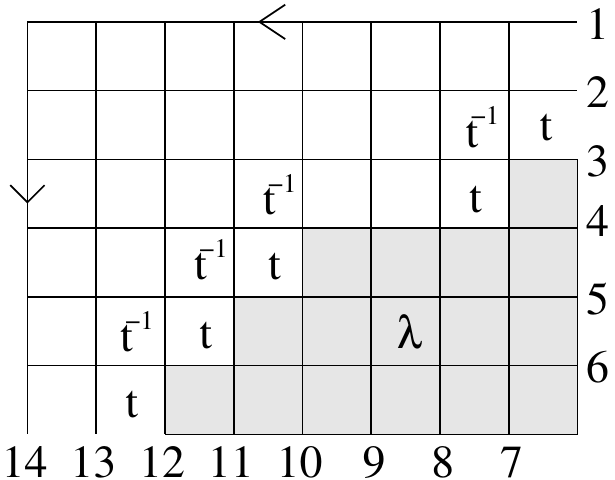} \hspace{1cm}
\includegraphics[height=1.8in]{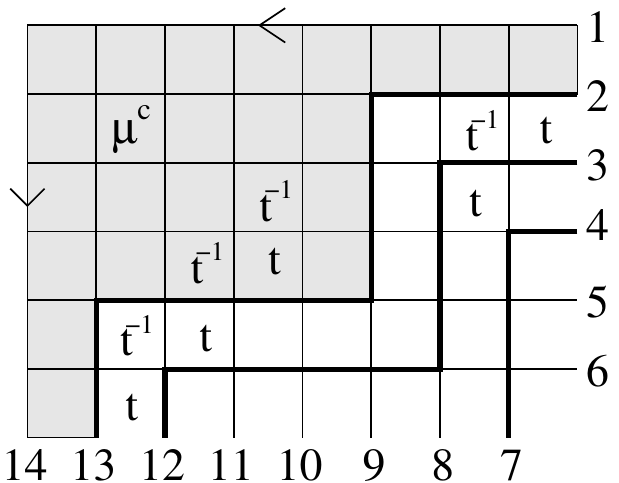}
\caption{The picture on the left 
shows the rotation of the transpose of the partition $\lambda$ and the network coordinates used to define $x_{\lambda}(t)$, in the case that 
$k=6$, $n=14$, and $\lambda = (4,3,3,3,2,1)$.  The picture on the  
right  shows the strongly minimal flow for 
$p_{\{1,5,6,7,12,13\}}({x}_{\lambda}(t))=
p_{\mu}({x}_{\lambda}(t))$, where 
$\mu = (5,5,5,2,2,2,2)$ can be read off as the transpose of the rotation of the partition made up of the white boxes. Note that $\{1,5,6,7,12,13\}$ are the west steps of $\mu$, which appear as south steps of the boundary lattice path in the picture.   The shaded region labeled $\mu^c$ is actually the transpose of the partition we call $\mu^c=(6,4,4,4,4,1,1,1)$. }
\label{fig:matrixval}
\end{figure}

\begin{example}
The right hand side of Figure \ref{fig:matrixval}
shows the strongly minimal flow for 
$p_{\{1,5,6,7,12,13\}}({x}_{\lambda}(t))=
p_{\mu}({x}_{\lambda}(t))$, where 
$\lambda = (4,3,3,3,2,1)$ and $\mu = (5,5,5,2,2,2,2)$. See \cref{rem:toppathdual}.
The flow has weight $t^3$, because the 
path from $2$ to $13$ has weight $t^2$, and the path from $3$ to 
	$12$ has weight $t$, and all other paths have weight $1$. (Recall that the weight of 
	a path is the product of the contents of all boxes to the left of it, and the weight of 
	a flow is the product of the weights of its paths.)
Meanwhile, the right hand side of Figure \ref{fig:matrixval2}
shows another flow for 
$p_{\mu}({x}_{\lambda}(t))$, which has weight $t^2$.  
This corresponds to the fact that $\maxdiag(\mu \setminus \lambda) = 2$. Note that the 
lowest order term of $t$ in 
$p_{\mu}({x}_{\lambda}(t))$  
	is not realized by the strongly minimal flow.
\end{example}

\begin{proof}[Proof of \cref{t:matrix}]
The strongly minimal flow $F^0$ contributing to 
$p_{\mu}({x}_{\lambda}(t))$ is the flow shown in 
\cref{fig:matrixval}, whose topmost path coincides with the 
southeast border of (the transpose of) $\mu^c$.  So the (reflected and rotated) partition 
$\mu$ consists of the boxes which are southeast of the 
topmost path of the flow.  
All other flows 
contributing to $p_{\mu}({x}_{\lambda}(t))$  
have the same 
starting and ending
points as $F^0$ but now the paths are arbitrary pairwise non-intersecting
paths consisting
of west and south steps.

Let us call a path in the network \emph{rectangular} if 
it consists of a series of west steps followed by south steps.
Note that by construction, the weight of any path in the network 
associated to ${x}_{\lambda}(t)$ will be $t^{\ell}$ for some 
$\ell \geq 0$.  Note that if a given path $p$ from $i$ to $j$
encloses a box with $t$ or $t^{-1}$, then any path $p'$ from $i$ to 
$j$ which is weakly above $p$ 
will have weight $t^{\ell}$ for $\ell \geq 1$.  Moreover,
the rectangular path from $i$ to $j$
will have
weight precisely $t$, because if one looks at the boxes northwest of the 
	inner and outer corners of the boundary of $\lambda$, 
	the rectangular path will always contain one more ``inner'' than ``outer'' box. 

\begin{figure}[h]
\centering
\includegraphics[height=1.8in]{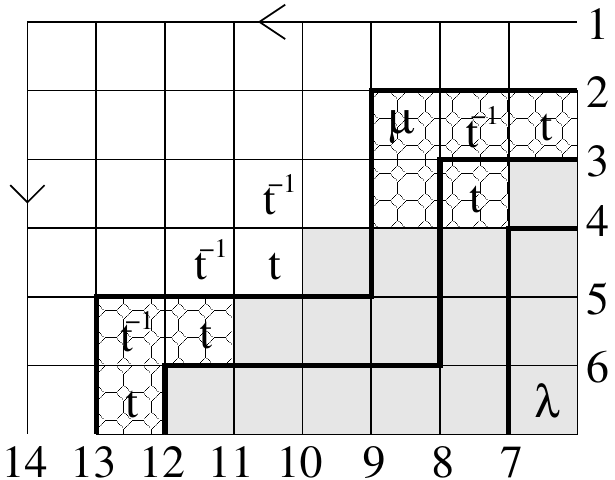} \hspace{1cm}
\includegraphics[height=1.8in]{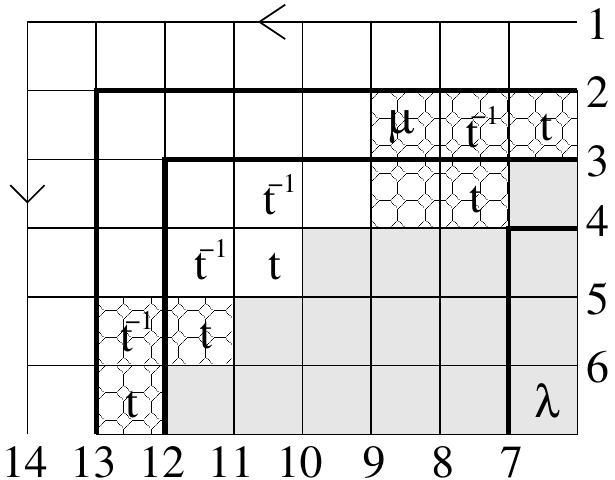}
\caption{At the left we have the network parameterization used to 
define ${x}_{\lambda}(t)$, together with the strongly minimal
flow $F^0$ associated to $p_{\mu}$.
The flow $F$ at the right is obtained from $F^0$ by replacing 
the paths from 
$2$ to $13$ 
and from 
$3$ to $12$ 
by the corresponding
rectangular paths.
}
\label{fig:matrixval2}
\end{figure}

Note that $\ValK(p_{\mu}({x}_{\lambda}(t)))=\ValK(\wt(F))$, 
where $F$ is the flow associated to $p_{\mu}$ whose weight is $t^{\ell}$ 
for $\ell$ as small as possible.
By the observations of the previous paragraph, we can construct
the desired flow $F$ from the strongly minimal flow $F^0$ by
replacing each path from $i$ to $j$ whose weight is \emph{not} $1$
by the rectangular path from $i$ to $j$, see \cref{fig:matrixval2}.
Then $\wt(F)=t^{\ell}$, where 
$\ell$ is the number of paths $p$ in  
$F^0$ such that $\wt(p) \neq 1$.  But the paths in $F^0$
with weight not equal to 
$1$ are precisely the paths which 
enclose at least one box
with $t$ or $t^{-1}$.  So 
$\ValK(p_{\mu}({x}_{\lambda}(t))=\ell$, where
$\ell$ is the number of paths in $F^0$ which 
enclose at least one box
with $t$ or $t^{-1}$.   It is not hard to see that this number
is equal to 
 $\maxdiag (\mu \setminus \lambda).$
\end{proof}

\begin{proof}[Proof of \cref{t:ValuationsFormula}
]
We want to show that 
for any reduced plabic graph $G$ and any $\lambda$ and $\mu$, 
\begin{equation}\label{claim}
\val_G(P_{\lambda})_{\mu} 
= \maxdiag (\mu \setminus \lambda).
\end{equation}
By \cref{p:rect-val}, we know that \eqref{claim} is true
when 
$G = G_{k,n}^{\rect}$ and $\mu$ is a rectangle.
Combining this with \cref{t:matrix}, we obtain that 
if $G = G_{k,n}^{\rect}$, then
\begin{equation}\label{claim2}
(\val_G(P_{\lambda})_{\mu})_{\mu\in \mathcal{P}_G} = 
(\ValK(p_{\mu}(x_{\lambda}(t)))_{\mu \in \mathcal{P}_G}.
\end{equation} 

But now if we apply a move to $G$, obtaining another plabic graph 
$G'$, then 
	\cref{l:piGmutation}
	implies that 
the right-hand side of \eqref{claim2}
transforms via the map $\Psi_{G,G'}$, while
\cref{thm1:tropcluster} implies that 
the left-hand side of \eqref{claim2}
transforms via the map $\Psi_{G,G'}$.
Therefore \eqref{claim2} holds for all plabic graphs $G$ and
all partitions $\mu\in\mathcal P_G$. \cref{t:ValuationsFormula} now follows from \eqref{claim2} and \cref{t:matrix}.
\end{proof}

\subsection{Flow polynomials
for frozen Pl\"ucker coordinates}\label{s:frozen}

In this section we describe the Pl\"ucker coordinates 
$P_{\mu_{i}}$ %
corresponding to the frozen vertices of our quivers.  %

\begin{definition}\label{def:balanced}
Let $Q = (Q_0,Q_1)$ be an arbitrary quiver with no loops or $2$-cycles,
where $Q_0$ denotes the set of vertices of $Q$ and $Q_1$ the set of arrows.  
Given  $v\in\Z^{Q_0}$ and mutable vertex $\nu$, we define the quantity
\begin{equation}\label{e:excess}
	\overset{\circ}v_\nu := \sum_{\mu\to \nu} v_{\mu} - 
	\sum_{v\leftarrow \mu'} v_{\mu'},
\end{equation}
where the summands correspond to arrows in $Q$ 
to and from the vertex $\nu$, respectively. 
We say that $\nu$ is \emph{balanced} with respect to the pair $(v,Q)$
if $\overset{\circ}v_\nu = 0$.
\end{definition}

\cref{lem:monomial} below follows 
directly from the $\mathcal X$-cluster mutation formula 
\eqref{e:XclusterMut}.
 
\begin{lemma}\label{lem:monomial}
Let $Q=(Q_0,Q_1)$ be a quiver as above, with $v\in \Z^{Q_0}$
	and corresponding monomial $x^v$ in $\mathcal{X}$-cluster variables. 
	Then the $\mathcal X$-mutation
$\Mut_{\nu}^{\mathcal{X}}(x^v)$ 
	(recall 
	\cref{Xseed})
	at the vertex $\nu$
 is a monomial if and only if $\overset{\circ}v_\nu=0$. Moreover if it is a monomial then 
	its new exponent vector $v'$ is given by 
the (linear) formula 
\begin{equation}\label{e:balancedMutation}
v'_\eta=
\begin{cases} 
(\sum_{\mu \to \nu} v_\mu) - v_\nu, & \eta=\nu,\\
v_\eta, & \eta \ne \nu.
\end{cases}
\end{equation}
	Note that since $\overset{\circ}v_\nu=0$, this is an instance of 
	 tropicalized $\mathcal{A}$-cluster mutation, \emph{cf} \eqref{e:tropmut}.
\end{lemma}\qed

\begin{prop}\label{p:PmuiMonomial}
Let $G$ be an arbitrary $\mathcal X$-cluster seed of type $\pi_{k,n}$, 
with quiver $Q=Q(G)$ and set of vertices $\mathcal P_G$.  Choose
	$j\in\{0,1,\dotsc,n-1\}$. Recall the definition of $\mu_j$ from \cref{Young}. 
We have the following:
\begin{enumerate}
\item
$P_{\mu_j}$ is a Laurent monomial when written in terms of the $\mathcal X$-seed $G$,  i.e. 
 $P_{\mu_j}=x^{v}$ for some $v \in \Z^{\mathcal{P}_G}$.
\item
$\overset{\circ}v_\nu=0$ for all mutable vertices $\nu$ in $\mathcal P_G$.
\item 
If $G'$ is obtained from $G$ by mutation at a mutable vertex $\nu$, then 
when $P_{\mu_j}$ is written (as a Laurent monomial) in terms of the $\mathcal X$-seed $G'$, its
new exponent vector $v'$ is obtained from $v$ 
by 
\eqref{e:balancedMutation}.
\end{enumerate}
\end{prop}

\begin{proof}
By Proposition~\ref{p:XLaurent}, any $\mathcal X$-torus embeds into $\openX$. Since $P_{\mui}/P_\Max$ is regular on $\openX$ it expands as a Laurent polynomial in terms of $\mathcal X$-cluster coordinates $\Network(G)$. Since $P_{\mui}/P_\Max$ is nonvanishing by definition of $\openX$ it follows that it must be given by a single Laurent monomial. 
Properties (2) and (3) follow from 
\cref{lem:monomial}.
\end{proof}

\begin{rem}
While we used the embedding of $\mathcal X$ into $\Xcirc$ to give a quick proof that the frozen variables are Laurent monomials in any $\mathcal X$-torus, the same follows from a general result which we learned from Akhtar, which holds in any $\mathcal X$-cluster algebra constructed out of a quiver with no loops or $2$-cycles. Namely \cref{p:Akhtarcor} is a reformulation of \cite[Proposition~4.8]{Akhtar:PolygonalQuivers} \label{p:Akhtar}.

\begin{prop}\label{p:Akhtarcor} If $x^v$ is a monomial on an $\mathcal X$-cluster torus such that $v$ is balanced, i.e. $\overset{\circ}v_\mu=0$ for all mutable vertices $\mu$, then $x^v$ stays monomial with balanced exponent vector under any sequence of $\mathcal X$-mutations. 
\end{prop}
\end{rem}

\section{The proof that $\Delta_G = \Q_G$}\label{sec:mutationDelta}

In this section we mostly 
work in the setting of arbitrary $\mathcal{X}$- and $\mathcal{A}$-seeds of type $\pi_{k,n}$.
Recall that associated to any reduced plabic graph $G$ of type $\pi_{k,n}$, we have both an 
$\mathcal{X}$-seed $(Q(G), \widetilde{\Network}(G))$ which determines a torus in $\Xcirc$, and an $\mathcal{A}$-seed $(Q(G),\wPC(G))$ which determines a torus in $\opencheckX$.
And more generally, for any quiver mutation equivalent to $Q(G)$, we have an associated 
$\mathcal{X}$-seed and $\mathcal{A}$-seed and associated tori, which we continue to index by a letter $G$.
Our main result is that for any choice of $G$,
the Newton-Okounkov body $\Delta_G$ (which is defined in terms of the $\mathcal{X}$-seed 
associated to $G$)
is equal to the superpotential polytope 
$\Q_G$ (which is defined in terms of the $\mathcal{A}$-seed associated to $G$). 
Our proof starts by verifying this fact for $G = G_{k,n}^{\rect}$, proving along the way 
 that in this case, $\Q_G$ is isomorphic to a Gelfand-Tsetlin polytope
(via a unimodular transformation). From this we deduce various properties of $\Q_G$ including that $\Q_G=\Delta_G$ in the case where $\Q_G$ is a lattice polytope. We then use the 
\emph{theta function basis} of 
Gross, Hacking, Keel, and Kontsevich \cite{GHKK}, as well as \cref{c:PolytopeMutation}, 
which describes how the polytopes $\Q_G^r$ mutate, to deduce that $\Q_G=\Delta_G$ in general and complete the proof.

\subsection{
The rectangles cluster, Gelfand-Tsetlin polytopes, and the integral case}{\label{s:QGT}}

In 1950 Gelfand and Tsetlin \cite{GT} introduced integral polytopes $GT_\omega$ associated to arbitrary dominant weights $\omega$ of $GL_n$, such that the lattice points of $GT_\omega$ parameterize a basis of the representation $V_\omega$ (the Gelfand-Tsetlin basis) and such that $GT_{r\omega}=rGT_{\omega}$. If $\omega=\omega_{n-k}$ this construction gives a polytope with $n\choose k$ lattice points, such that the number of lattice points in its $r$-th dilation agrees with the dimension of the irreducible representation $V_{r\omega_{n-k}}$.
The representation $V_{r\omega_{n-k}}$ is isomorphic to the degree $r$ component of the homogeneous
coordinate ring of $\X$ by a special case of the Borel-Weil Theorem, and is furthermore isomorphic to the subspace $L_r\subset \C(\X)$ 
from \eqref{e:projnormal}. Thus this number of lattice points also computes the dimension of the latter two vector spaces.
We start by explaining how the polytope
$\Q_{G_{k,n}^{\rect}}^r$ is isomorphic to a Gelfand-Tsetlin polytope $GT_{r\omega_{n-k}}$ via a unimodular 
transformation.

\begin{definition}[Gelfand-Tsetlin polytope]\label{d:GT}
Let $GT_{r\omega_{n-k}} \subset \R^{\mathcal P_{G^{\rect}_{k,n}}}$ denote the polytope defined by 
\begin{align}
	 &0 \leq  f_{1 \times 1}  \label{Eq3}\\
	&f_{(n-k)\times k}   \leq  r  \label{Eq4}\\
	&0  \leq  f_{i \times j} - f_{(i-1)\times j}  \label{Eq1}\\
	&0  \leq  f_{i \times j} - f_{i \times (j-1)}, \label{Eq2}
\end{align}
where 
the defining variables $f_{i \times j}$ range over all nonempty
rectangles $i \times j$
contained in a 
$(n-k) \times k$ rectangle.  This polytope 
is called the \emph{Gelfand-Tsetlin polytope} for the highest weight $r\omega_{n-k}$.
\end{definition}

One often expresses Gelfand-Tsetlin polytopes in terms of 
\emph{Gelfand-Tsetlin patterns}, triangular arrays of real numbers
whose top row is fixed and whose rows interlace.  Clearly  
$GT_{r\omega_{n-k}}$ is the set of all such Gelfand-Tsetlin patterns with 
top row $(0^k, r^{n-k})$.  See Figure \ref{fig:GT} for the example
with $k=3$ and $n=5$.
\begin{figure}[h]
\centering
\includegraphics[height=1.3in]{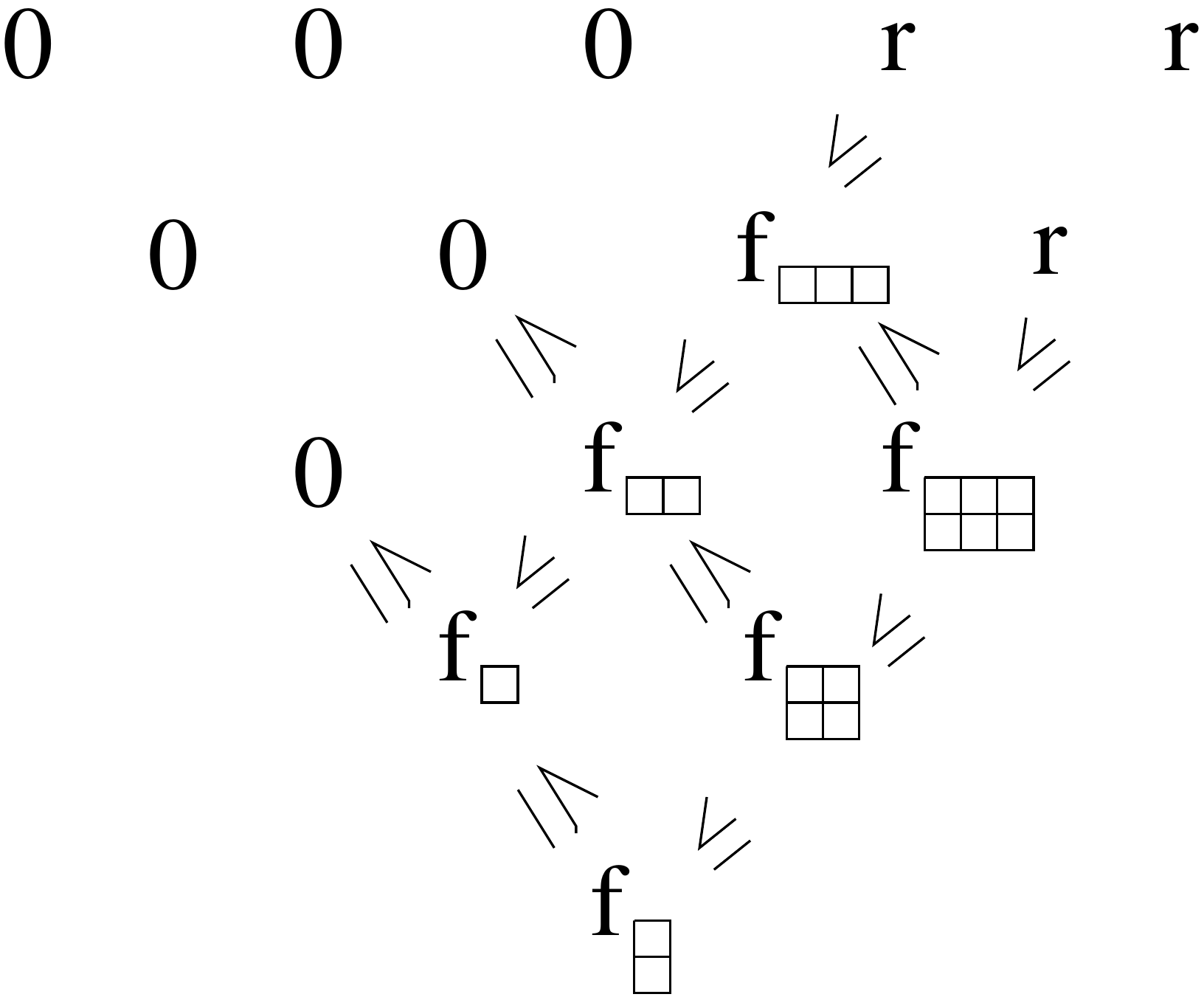}
\caption{Gelfand-Tsetlin patterns for $GT_{r\omega_{n-k}}$ with $k=3$ and $n=5$.  The convex hull
of all such patterns is the polytope $GT_{r\omega_{n-k}}$.}
\label{fig:GT}
\end{figure}
When $r=1$ the polytope $GT_{\omega_{n-k}}$ has integer vertices, one for each Young diagram in $\mathcal P_{k,n}$.

The following lemma explicitly describes 
an isomorphism between the polytope $\Q^r_{G_{k,n}^\rect} $ 
and the Gelfand-Tsetlin polytope $GT_{r\omega_{n-k}}$. 
If one compares Figures \ref{fig:superpotential} and \ref{fig:GT}, the isomorphism becomes quite
transparent. An analogous transformation comes up in \cite[Section 5.1]{AlexeevBrion}. 

\begin{lemma}\label{lem:integralGT}
The map $F:\R^{\mathcal P_{G^{\rect}_{k,n}}}\to \R^{\mathcal P_{G^{\rect}_{k,n}}}$ defined by
\[
(v_{i\times j})\mapsto (f_{i \times j}) = ( v_{i \times j} - v_{(i-1)\times (j-1)})
\]
is a unimodular linear transformation, with inverse given by 
$v_{i \times j} = f_{i \times j} + f_{(i-1) \times (j-1)} + 
f_{(i-2) \times (j-2)} + \dots.$  Moreover, 
$F(\Q_{G_{k,n}^{\rect}}^r)=GT_{r\omega_{n-k}}$.  %
Therefore the polytope $\Q_{G_{k,n}^{\rect}}^r$ is isomorphic to the Gelfand-Tsetlin polytope
$GT_{r\omega_{n-k}}$ by a unimodular linear transformation, and in particular has integer vertices. 
\end{lemma}
\begin{proof}
If we rewrite the inequalities \eqref{eq1} through \eqref{eq4}
defining $\Q^r_{G_{k,n}^{\rect}}$
in terms of $f$-variables, we obtain the system
of inequalities given by \eqref{Eq1}, \eqref{Eq2}, \eqref{Eq3}, and \eqref{Eq4}
which define 
the Gelfand-Tsetlin polytope $GT_{r\omega_{n-k}}$.
\end{proof}

\begin{defn}[Integer decomposition property] \label{d:IntClosed} A polytope $P$ is said to 
have the \emph{integer decomposition property} (IDP), or be \emph{integrally closed}, if every lattice point in 
the $r$th dilation $rP$ of $P$ is a sum of $r$ lattice points in $P$, that is, $\Lattice(rP)=r\Lattice(P)$.
\end{defn}

\begin{lemma}\label{lem:intclosed}
The polytopes  $GT_{\omega_{n-k}}$ 
and $\Q_{G_{k,n}^{\rect}} $ have the 
integer decomposition property.
\end{lemma}
\begin{proof}
This is well-known for  $GT_{\omega_{n-k}}$ 
and can also 
be proved explicitly by an inductive argument on 
integral Gelfand-Tsetlin patterns.
The result for 
$\Q_{G_{k,n}^{\rect}} $  now follows from
\cref{lem:integralGT}.
\end{proof}

\begin{lemma}\label{lem:GT}
Any GT tableau $T$
	(see \cref{GTtableau}),
	 viewed as an element of $\R^{\mathcal{P}_{G}}$ for $G = G_{k,n}^{\rect}$, is a lattice point of 
$\Q_G$.
\end{lemma}

\begin{proof} 
If we apply the map $F$ from 
 Lemma \ref{lem:integralGT} to $T$,
it gets
transformed into an
$(n-k) \times k$ array of $0$'s and $1$'s with rows and columns weakly
increasing.  For example, Figure \ref{fig:GT-tableau} shows
both the tableau from Figure \ref{fig:tableauminflow} (in ``$v$-variables")
and also its image under the map $F$
(in ``$f$-variables").  
Therefore $F(T)$ 
is  an integral Gelfand-Tsetlin pattern in $GT_{\omega_{n-k}}$, see 
Figure \ref{fig:GT}, and hence  
by \cref{lem:integralGT},
$T \in \Q _G$.
This shows that the integer point $T$ lies in 
$\Q_G$.
\end{proof} 

\begin{figure}[h]
\centering
\includegraphics[height=1in]{TableauMinFlow} \hspace{1cm}
\includegraphics[height=1in]{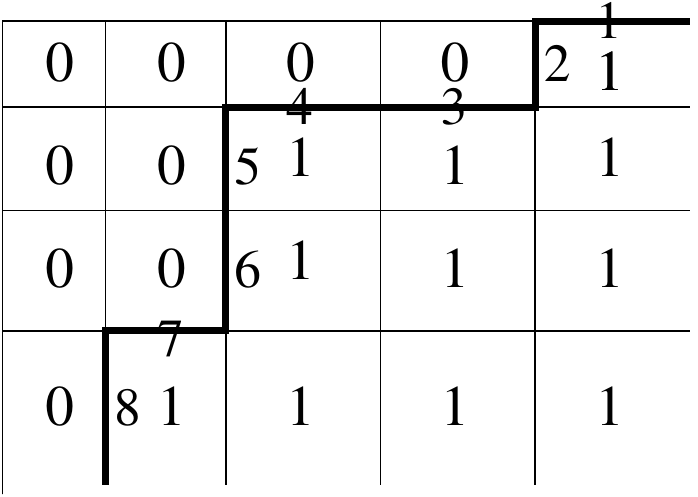}
\caption{The ``tableau", or exponent vector associated to the strongly minimal flow from Figure
	\ref{fig:minflow} and its associated Gelfand-Tsetlin pattern.}
\label{fig:GT-tableau}
\end{figure}

\begin{proposition}\label{prop:convQ}
When $G = G^{\rect}_{{k,n}}$, the polytopes
$\conv_G$ and $\Q_G$ coincide, and the lattice points of 
$\conv_G = \Q_G$ are precisely the ${n \choose k}$ points
$\val_G(P_{\lambda})$ for $\lambda \in \mathcal{P}_{k,n}$.
\end{proposition}

\begin{proof}
We write $G$ for $G^{\rect}_{k,n}$. 
By definition, $\conv_G$ is the convex hull 
of  $\{\val_G(P_\lambda)\mid \lambda\in\mathcal P_{k,n}\}$.
By Lemma \ref{lem:vertices}, we have that for 
$\lambda \neq \lambda'$,
$\val_G(P_\lambda)$ 
and $\val_G(P_{\lambda'})$ are distinct, so 
$\conv_G$ contains at least ${n \choose k}$ lattice points.

We next show that $\conv_G \subseteq \Q_G$.
By \cref{lem:vertices}, each point
$\val_G(P_\lambda)$ can be encoded by a GT tableau $T$. By \cref{lem:GT}, 
$T$ is a lattice point of 
$\Q _G$, and hence 
$\conv_G \subseteq \Q_G$.

By  \cref{lem:integralGT}, 
the polytope $\Q^\RG$ is an integral polytope with precisely
$\dim V_{r \omega_{n-k}}$ lattice points. 
In particular, 
 $\Q _G$ is  integral  with 
 precisely  ${n \choose k}$ lattice points.  It follows that 
$\conv_G = \Q_G$, and the lattice points of 
$\conv_G = \Q_G$ are precisely the ${n \choose k}$ points
$\val_G(P_{\lambda})$ for $\lambda \in \mathcal{P}_{k,n}$.
\end{proof}

\begin{prop}\label{p:latticeQr} 
For an arbitrary seed, indexed by $G$, the number of lattice points of the superpotential polytope $\Gamma_G^r$ coincides with the dimension $\dim V_{r\omega_{n-k}}$. 
\end{prop}

\begin{proof}
If $G=G^{\rect}_{k,n}$ then the statement follows from the analogous property of the Gelfand-Tsetlin polytope, because of \cref{lem:integralGT}.
By \cref{c:intmoves}, this cardinality is independent of $G$.  
\end{proof}

\begin{cor}\label{c:GammaVol}
For an arbitrary seed $\check\Sigma^{\mathcal A}_G$, the volume of the superpotential polytope $\Gamma_G$ is given by
\begin{equation}\label{e:volume}
\prod_{1 \leq i \leq k} \frac{(k-i)!}{(n-i)!}
\end{equation}
\end{cor}

\begin{proof} 
The Hilbert polynomial $h_\X(r)$ of the Grassmannian $\X$
 in its Pl\"ucker embedding satisfies
$h_\X(r) = \dim V_{r \omega_{n-k}}$ for $r >> 0$.
And moreover the leading coefficient of $h_\X(r)$ is 
$\prod_{1 \leq i \leq k} \frac{(k-i)!}{(n-i)!}$ \cite{GrossWallach}.
But \cref{p:latticeQr} implies that 
$\dim V_{r \omega_{n-k}}$ equals the number of lattice points in the $r$-th dilation
of $\Gamma_G$, which implies that the Ehrhart polynomial of $\Gamma_G$
	(the polynomial whose value at $r$ is the number of lattice points
	in the dilated polytope $r \Gamma_G$)
equals $h_{\X}(r)$.  Since the leading coefficient of the Ehrhart polynomial equals the volume
of the corresponding polytope,  the corollary follows.
\end{proof}

\begin{cor}\label{c:volumecompare} For arbitrary $G$, the superpotential polytope $\Gamma_G$ and the Newton-Okounkov  body $\Delta_G$ have the same volume.  
\end{cor}

\begin{proof}
Using \cref{p:latticeQr}, it follows that the volume of $\Q_G$ equals
\[\Vol(\Q_G)=\lim_{r\to\infty}\frac{\dim V_{r{\omega_{n-k}}}}{r^{\dim(\X)}}.
\]
Meanwhile, it is a fundamental property of  Newton-Okounkov  bodies 
	(associated to valuations with one-dimensional leaves, 
see \cref{d:1dimleaves} and 
	\cref{l:okounkovlemma})
	that their volume encodes the asymptotic dimension of the space of sections $H^0(\X,\mathcal O(rD))$ as $r\to \infty$. Explicitly we have by \cite[Proposition~2.1]{LazarsfeldMustata} that 
\[
\Vol(\Delta_G)=\limsup_{r\to\infty}\frac{\dim H^0(\X,\mathcal O(rD))}{r^{\dim(\X)}}.
\]
Since $ H^0(\X,\mathcal O(rD))$ is isomorphic to the representation $V_{r\omega_{n-k}}$, the result 
follows.
\end{proof}

\begin{cor}\label{c:nchoosek}
Suppose $G$ is a reduced plabic graph of type $\pi_{k,n}$. The $n\choose k$ lattice points in $\Gamma_G$ are precisely the valuations $\val_G(P_{\lambda})$ of Pl\"ucker coordinates.
\end{cor}

\begin{proof}
	If $G$ is the rectangles plabic graph this is the contents of \cref{prop:convQ}. If we mutate the plabic graph $G$ to another plabic graph $G'$ by a square move, then the  tropicalized $\mathcal{A}$-cluster mutation transforms $\val_G(P_{\lambda})$ to $\val_{G'}(P_{\lambda})$ by \cref{thm1:tropcluster}. On the other hand the tropicalized $\mathcal{A}$-cluster mutation gives a bijection between the lattice points of $\Q_G$ and $\Q_{G'}$ by %
 \cref{c:intmoves}.
\end{proof}

\begin{rem}
Again when $G$ is a reduced plabic graph of type $\pi_{k,n}$, one may use results of \cite{PSW} to 
prove
that the polytope $\conv_G$ has $n \choose k$ lattice points
$\{\val_G(P_J) \ \vert \ J\in {[n] \choose n-k} \}$; 
moreover, each of those 
lattice points is a vertex.
To see this, recall that 
in \cite{PSW}, the authors studied the \emph{matching polytope} 
associated to a reduced
plabic graph $G$, which 
is defined by taking the convex hull of \emph{all} exponent vectors in the 
flow polynomials $P_J^G$ from \eqref{eq:Plucker},
where $J$ runs over elements in ${[n] \choose n-k}$.  
It was shown there that 
every such exponent vector gives rise to a distinct vertex of the matching polytope.
Since $\conv_G$ is defined as the convex hull of a subset of the exponent vectors
used to define the matching polytope, it follows that the elements of 
$\{\val_G(P_J) \ \vert \ J\in {[n] \choose n-k} \}$ are vertices of 
$\conv_G$, and are all distinct.  
\end{rem}

\begin{theorem}\label{t:intcase} 
Suppose $G$ is a reduced plabic graph of type $\pi_{k,n}$ for which $\Gamma_G$ is a lattice polytope. Then the Newton-Okounkov  body 
$\NO_{G}$ is equal to  
 $\Q_{G}$, %
 and these polytopes furthermore coincide with $\conv_G$.
\end{theorem}

\begin{proof} 
If $\Gamma_G$ is a lattice polytope, then it is the convex hull of its lattice points. By \cref{c:nchoosek} this implies $\Gamma_G=\Conv_G$. On the other hand we have $\Delta_G \supseteq \Conv_G$, by definition. So we get $\Delta_G\supseteq\Gamma_G$. But by \cref{c:volumecompare} we know that $\Gamma_G$ and $\Delta_G$ both have the same volume, and given any inclusion $A\supseteq B$ of convex bodies where $A$ and $B$ have the same volume it follows that $A=B$.
\end{proof}

\subsection{The theta function basis}\label{s:theta}

Recall that cluster $\mathcal{A}$- and $\mathcal{X}$-varieties are constructed by gluing
together ``seed tori" via birational maps known as cluster transformations; 
cluster varieties
were introduced by Fock and Goncharov in \cite{FG} and are a more geometric point of view
on the cluster algebras of Fomin and Zelevinsky \cite{ca1}.  
The cluster $\mathcal{A}$-variety
is the geometric counterpart of a cluster algebra, 
while the cluster $\mathcal{X}$-variety
corresponds to the $y$-seeds of Fomin and Zelevinsky \cite[Definition 2.9]{ca4}.  
In this section 
we will assume that the reader has some familiarity with \cite{GHK} and \cite{GHKK};
 in particular we will use the notation for cluster varieties from  \cite[Section 2]{GHK}.

Note that the network charts for $\mathbbX^{\circ}$ in \cref{sec:poschart}
and their further $\mathcal X$-mutations 
give $\mathbbX^{\circ}$ roughly the structure of a 
cluster $\mathcal{X}$-variety,\footnote{Technically a cluster $\mathcal{X}$-variety is defined to be the 
union of the cluster charts, which $\mathbbX^{\circ}$ agrees with up to codimension $2$; since we 
are concerned only with coordinate rings, this difference is inconsequential.} see Section~\ref{s:twist}.
Similarly, the cluster charts for 
$\checkX^{\circ}$ in 
\cref{sec:cluster}  give
$\checkX^{\circ}$ the structure of a  cluster $\mathcal{A}$-variety.
See \cite{Postnikov}, \cite{Scott}, and \cite[Section 1.1]{MullerSpeyer} for more details.

\begin{theorem}
There is a theta function basis $\mathcal{B}(\mathbbX^\circ)$ for the coordinate ring
$\C[\widehat{\mathbbX}^\circ]$ of the affine cone over the cluster $\mathcal{X}$-variety
 $\mathbbX^{\circ}$,  
 which restricts to a theta function basis 
$\mathcal{B}(\mathbbX)$ 
for the homogeneous coordinate ring
$\C[\widehat{\mathbbX}]$ of the Grassmannian.
	And  $\mathcal{B} (\mathbbX)$
	restricts to a basis 
 $\mathcal{B}_r$ 
of 
the degree $r$ component of the homogeneous coordinate ring, for every $r\in\Z_{\ge 0}$.
\end{theorem}

\begin{rem}
We note that the degree $r$ component of the homogeneous coordinate ring above is naturally isomorphic to $L_r$ by the map which sends a degree $r$ polynomial $P$ in Pl\"ucker coordinates to $P/{P_\Max^r}\in L_r$. We will use this isomorphism to identify  $\C[\widehat\X]_r$ with $L_r$ when convenient.
 \end{rem}
\begin{proof}
Gross-Hacking-Keel-Kontsevich \cite[Theorem 0.3]{GHKK} showed that canonical bases of global regular 
``theta" functions exist for a formal version of cluster varieties, and in many cases
(when ``the full Fock-Goncharov conjecture holds"), these extend to bases for regular functions
on the actual cluster varieties.  They pointed out that the full Fock-Goncharov conjecture holds if 
there is a maximal green sequence for the cluster variety; in the case of $\openX$, a maximal green sequence
was found by Marsh and Scott, see \cite[Section 11]{MarshScott}.
Therefore we indeed have a theta function basis $\mathcal{B}(\mathbbX^\circ)$ for 
the coordinate ring $\C[\widehat{\mathbbX}^\circ]$ of the affine cone over the cluster $\mathcal{X}$-variety
$\openX$.

Note that there is also a theta function basis
$\mathcal{B}(\mathbbX)$ 
for the  coordinate ring 
$\C[\widehat{\mathbbX}]$ of the affine cone over the Grassmannian; see \cite[Section 9]{GHKK}
for a discussion of how \cite[Theorem 0.3]{GHKK}  extends to partial compactifications of cluster
varieties coming from frozen variables.  Moreover we claim that 
$\mathcal{B}(\mathbbX) \subset 
 \mathcal{B}(\mathbbX^\circ)$.  
This follows from \cite[Proposition 9.4 and Corollary 9.17]{GHKK}.

	Finally $\mathcal{B} (\mathbbX)$
restricts to a basis of $L_r$ because it is compatible with the 
one-dimensional 
torus action (which is overall scaling in the 
Pl\"ucker embedding). 
\end{proof}

We now prove \cref{th:pointed}, which says 
that for a cluster $\mathcal{X}$-variety, and an arbitrary choice of 
$\mathcal{X}$-chart, each
 theta basis element $\theta$ is  pointed with respect to the
 $\mathcal{X}$-chart.  In other words, $\theta$ can be written
 as a Laurent monomial multiplied by a polynomial with constant term 
$1$ (cf.\ \cref{def:minimal}) in the variables of the $\mathcal{X}$-chart.
\cref{th:pointed} 
follows from the machinery of \cite{GHKK}, and we are grateful to  
Man-Wai (Mandy) Cheung, Sean Keel, 
and Mark Gross for their useful explanations on this topic.

Note that \cref{th:pointed} confirms a conjecture of Fock and Goncharov, 
see \cite[Conjecture 4.1, part 1]{FG}, and also \cite[page 41]{GoncharovShen2}.

\begin{theorem}\label{th:pointed}
Fix a cluster $\mathcal{X}$-variety and an arbitrary $\mathcal{X}$-chart.  Then
every element of the theta function basis can be written as a 
	pointed Laurent polynomial in the variables of the 
	$\mathcal{X}$-chart.
  Moreover the exponents of the leading terms are all distinct.
\end{theorem}

\begin{proof}
Elements of the theta function basis for $\mathcal{X}$
are constructed using a consistent \emph{scattering diagram}
$\mathfrak{D^{\mathcal{X}}}$ associated to the seed.  In keeping with \cite{GHKK} we denote by $N$ the character group of the $\mathcal X$-cluster torus of our chosen $\mathcal X$-cluster seed, which we embed as a lattice in $N_\R=N\otimes \R$. The $n\in N$ are interpreted as exponents of monomial functions on the $\mathcal X$-cluster torus. The $\mathcal X$-cluster variables define a basis of $N$ and therefore $N_\R$. Its $\Z_{\ge 0}$ span, denoted $N^+$, is the set of lattice points in the associated positive orthant.  
 
The theta functions $\theta_n$ 
are indexed by lattice points
$n \in N$, see \cite[Definition 7.12]{GHKK},
and 
$$\theta_n = \sum_{\gamma} \Mono(\gamma),$$ where the sum is over all broken lines with initial 
exponent $n$.  
There is a monomial attached to each domain of linearity of a
broken line, which is inductively computed based on which walls of the scattering
diagram have been crossed; 
 $\Mono(\gamma)$ is the monomial attached to the last domain of linearity.
From the construction it is clear that if every function attached to each 
wall of $\mathfrak{D^{\mathcal{X}}}$ is positive, i.e.\ if it is a power series in 
$\mathbf{x}^n$ for $n$ in the positive orthant $N^+$, then the element
$\theta_n$ will be pointed with leading term $\mathbf{x}^n$, 
and the exponent vectors of leading terms of the $\theta_n$'s will in particular all be distinct.

In \cite{GHKK}, the authors explain how to construct the scattering diagram for 
$\mathcal{X}$ from that for $\mathcal{A}_{\prin}$, which maps to  
 $\mathcal{X}$.  By 
\cite[Construction 2.11]{GHKK}, the walls of 
$\mathfrak{D}^{\mathcal{A}_{\prin}}$ 
have the form 
$(n,0)^{\perp}$ for $n\in N^+$.
And by \cite[Construction 7.11]{GHKK}, 
the functions on walls of 
$\mathfrak{D}^{\mathcal{A}_{\prin}}$ 
are series in $\mathbf{z}^{(p^*(n),n)} = \mathbf{a}^{p^*(n)} \mathbf{x}^n$ for $n\in N^+$.
As noted in \cite[footnote 2, page 72]{GHKK}, one can then obtain the scattering diagram
$\mathfrak{D}^{\mathcal{X}}$ from $\mathfrak{D}^{\mathcal{A}_{\prin}}$
by intersecting each wall with $w^{-1}(0)$, where 
$w$ is the weight map from tropical points of $\mathcal{A}_{\prin}$ to $\Hom(N,\Z)$ \cite[page 71]{GHKK},
and replacing the series 
in $\mathbf{z}^{(p^*(n),n)} = \mathbf{a}^{p^*(n)} \mathbf{x}^n$ by
the corresponding series in $\mathbf{x}^n$. 
Therefore each function attached to a wall
 of $\mathfrak{D^{\mathcal{X}}}$ is a power series in 
$\mathbf{x}^n$ for $n\in N^+$.
\end{proof}

\begin{lemma}\label{lem:base}
When $G=G^{\rect}_{k,n}$, 
	for each lattice point 
$d\in \Q_G^r$, there is an element $\theta_d\in \mathcal{B}_r$ such that
 $\val_G(\theta_d) = d$.
\end{lemma}

\begin{proof}
By \cref{lem:intclosed}, the polytope $\Q_G$ has the integer decomposition property in the rectangles cluster case. Furthermore by \cref{t:intcase}, %
we have $\Delta_G=\Conv_G=\Q_G$, and hence  $\val_G(L_r)\subset r\Q_G$.   Therefore the lattice points in $\Q_G^r=r\Q_G$ are precisely the elements in $\val_G(L_r)$, since 
by \cref{p:latticeQr} and \cref{l:okounkovlemma}, both sets have the same cardinality.
Since the elements of $\mathcal{B}_r$ are a basis of $L_r$, and have distinct valuations by \ref{th:pointed},
it follows that for each lattice point $d\in \Q_G^r$, there is an 
element $\theta_d$ of $\mathcal{B}_r$, which when expressed in terms of 
the variables $\Network(G)$ of the 
$\mathcal{X}$-seed $G$, is pointed with 
leading term $x^d$.
\end{proof}

\begin{lemma}\label{lem:commute}
If $G$ and $G'$ index two $\mathcal{X}$-seeds which are connected by a single mutation,
then we have a commutative diagram 
	
	\begin{equation} %
	\begin{tikzcd}
		&\mathcal{B}_r	\arrow{dl}[swap]{\val_G} \arrow{dr}{\val_{G'}} & \\
		\val_G(L_r) \arrow{rr}{\Psi_{G,G'}} &&
		\val_{G'}(L_r)
	\end{tikzcd}
\end{equation} 
where $\Psi_{G,G'}$ is a bijection, 
	the \emph{tropicalized $\mathcal{A}$-cluster mutation} from \cref{l:LTrop}. 
\end{lemma}
\begin{proof}
Since $\mathcal{B}_r$ is a basis of $L_r$ and the elements have distinct leading terms,
the maps $\val_G$ and $\val_{G'}$ are bijections.  The fact that 
the diagram is commutative follows from the fact that 
the elements of $\mathcal{B}(\openX)$ are parameterized by 
the tropical points of the $\mathcal{A}$-variety (see 
\cite[(0.2)]{GHKK} and \cite[Conjecture 1.11]{GHK}
for this parameterization, as well as
 \cite[(12.4) and (12.5)]{FG1} for the mutation rule for tropical points
of the $\mathcal{A}$-variety).
Since the diagonal maps are bijections, $\Psi_{G,G'}$ is a bijection; see also Remark~\ref{r:integral}.
\end{proof}
Note that \cref{lem:commute} would not hold if we replaced $\mathcal{B}_r$ by 
e.g. the standard monomials basis of $L_r$.
Working with 
$\Psi_{G,G'}$ is a bit delicate, since the map is only  piecewise linear 
(see \cref{r:polmut}).

We now prove \cref{thm:main}, the second  main result of this paper.

\begin{theorem}\label{thm:main}
Let $G$ be any reduced plabic graph of type $\pi_{k,n}$, or more generally,
any $\mathcal{X}$-seed $G$ of type $\pi_{k,n}$.
Then the Newton-Okounkov  body $\Delta_G$ coincides with the superpotential polytope $\Q_G$. Moreover, the Newton-Okounkov  body is a rational polytope.
\end{theorem}

\begin{proof}
When $G = G^{\rect}_{{k,n}}$, we have
from 
\cref{lem:base}
that 
$\val_G(L_r) = \Lattice(\Q_G^r)$.
By \cref{lem:commute}, 
$$\Psi_{G,G'}: 
\val_G(L_r ) \to 
\val_{G'}(L_r )$$ is a bijection, 
and by the proof of
\cref{c:intmoves},
\[\Psi_{G,G'}: \Lattice(\Q_G^r) \to \Lattice(\Q_{G'}^r)
\] is a bijection.
Therefore, 
 using the fact that all $\mathcal{X}$-seeds are connected by mutation,
it follows that 
$\val_G(L_r) = \Lattice(\Q_G^r)$ for any $\mathcal{X}$-seed $G$ of type %
$\pi_{k,n}$.  
Now since $\Q_G^r = r \Q_G$
(see 
\cref{rem:dilation}),
 we have 
\begin{equation}
\Q_G = 
\overline{\operatorname{ConvexHull}
\left(\bigcup_r\frac{1}{r} 
\Lattice(\Q_G^r) \right)} = 
\overline{\operatorname{ConvexHull}
\left(\bigcup_r \frac{1}{r} 
\val_G(L_r)\right)} = 
\Delta_G
\end{equation}
for any $G$ of type $\pi_{k,n}$, where the first equality is as in \cref{e:polytope}.
\end{proof}

In the  plabic graph case
we summarise our results as follows.

\begin{cor}\label{l:injection}
Let $G$ be a reduced plabic graph of type $\pi_{k,n}$.
Then $\Delta_G$ equals to $\Q_G$, and has precisely 
$n \choose k$ lattice points. These are the valuations of Pl\"ucker coordinates
$P_{\lambda}$ for $\lambda \in \mathcal{P}_{n,k}$, and  they
can be computed explicitly using the formula
\begin{equation*}
\val_G(P_{\lambda})_{\mu} = \maxdiag (\mu \setminus \lambda), 
\end{equation*}
where 
$\maxdiag (\mu \setminus \lambda)$ is given in \cref{def:maxdiag}. Here 
the $\mu$'s run through $\mathcal P_G$. \qed
\end{cor}

This corollary is a combination of 
\cref{t:ValuationsFormula},
\cref{c:nchoosek}, 
and 
\cref{thm:main}.
In \cref{s:clusterexpansion} we will give an explicit description of $\Q_G$ in terms of the plabic graph.

\section{Khovanskii bases and toric degenerations}\label{degeneration}

Under certain conditions the Newton-Okounkov  body construction can be used to obtain Khovanskii or SAGBI bases \cite{KavehManon16} and toric degenerations, see 
for example 
\cite{Kaveh0}, \cite{Kaveh} and 
\cite{Anderson}.  We will briefly review this connection as it applies in our setting. 

\subsection{Khovanskii bases}
\begin{definition}[{following \cite[Definition 1]{KavehManon16}}] \label{def:SAGBI} Suppose $R$ is a finitely generated $\C$-algebra with Krull dimension $d$ and discrete valuation $\val:R\setminus\{0\}\to \Z^d$ where we view $\Z^d$ as a group with a total ordering such that $v<v'$ implies $v+u<v'+u$. The {\it value semigroup} $S=S(R,\val)$ of $\val$ is by definition the subsemigroup of $\Z^d$ which is the image of $\val$. For each $v\in S$ define the subspaces
\[
R_{\ge v}:=\{f\in R\mid \val(f)\ge v\}\cup\{0\},\qquad 
R_{> v}:=\{f\in R\mid \val(f)>v\}\cup\{0\},
\]
and define the associated graded algebra  $\gr_{\val}(R)=\bigoplus_{v\in S}  R_{\ge v}/R_{>v}$,  graded over the semigroup $S$.
For each nonzero $f$ in $R$ there is an element $\bar f$ in  $\gr_{\val}(R)$, which lies in $R_{\ge v}/R_{>v}$ for $v=\val(f)$, and which is represented by $f$. 
A (finite) set $\mathcal B\subset R$ is called a (finite) {\it Khovanskii basis} for $(R,\val)$ if the image of $\mathcal B$ in the associated graded $\gr_{\val}(R)$ forms a set of algebra generators.
\end{definition}

The example we have in mind for $R$ is the homogeneous coordinate ring of $\X$ in {\it some} projective embedding. The valuation will be an extension of  $\val_G$ which also incorporates the grading. 

\begin{definition}[$1$-dimensional leaves]\label{d:1dimleaves} A valuation $\val$ as in Definition~\ref{def:SAGBI} is said to have \emph{$1$-dimensional leaves} if the graded components of $\gr_{\val}(R)$ are at most $1$-dimensional. 
  \end{definition}

\begin{rem}\label{r:leaves} We will always assume that  the valuation $\val$ has $1$-dimensional leaves. In this case we have that $\mathcal B\subset R$ is a Khovanskii basis if the set $\val(\mathcal B)$  of valuations generates the semigroup $S$. This definition generalises the concept of a SAGBI basis, see also  {\cite[Definition 5.24]{KavehKhovanskii08}},
as well as \cite[Remark~4.9]{Hering}. The terminology SAGBI stands for Subalgebra Analogue of 
Gr\"obner Basis for Ideals and originates from
 the case where $R$ is a subalgebra of a polynomial ring. Note that a finite Khovanskii basis for $R$ exists if and only if $S$ is a finitely generated semigroup. The well-known 
	\emph{subduction algorithm} (see \cite[Algorithm~2.11]{KavehManon16})
	allows one to represent every element of $R$ as a polynomial
	in elements of a  Khovanskii basis.
\end{rem}

\subsection{Toric degenerations}\label{s:Y} Let  $Y$ be an $m$-dimensional, irreducible projective variety, 
with a valuation $\val:\C(Y)\setminus\{0\}\to\Z^m$ with one-dimensional leaves. Fix  an ample divisor $D$ on $Y$. We associate to $(Y,D)$ the graded algebra 
\begin{equation}\label{e:R}
R=\bigoplus_{j=0}^\infty R^{(j)}
=\bigoplus_{j=0}^\infty
t^jH^0(Y,\mathcal O(jD)) \subset \C(Y)[t].
\end{equation}
We define an extended valuation $\overline{\val}$ on $R$, with value semigroup 
$\overline S \subseteq \Z\x \Z^m$ %
by setting
\begin{eqnarray}\label{e:extendedval}
\overline{\val}: R\setminus\{0\}&\to &\Z\x \Z^m,\\
\sum t^j f^{(j)} &\mapsto & \left(j_0,\val (f^{(j_0)})\right),
\end{eqnarray}
where $j_0=\max\{j\mid f^{(j)}\ne 0\}$. 
Note that the projection to its first component 
gives 
$\overline S$  a $\Z_{\ge 0}$-grading.

Following \cite{KavehManon16}, we choose an order for $\Z\x\Z^m$ (and hence $\overline S$) 
using  a combination of the 
reverse order on $\Z$ and the 
standard lexicographical order on $\Z^m$. 
Namely $(r,v)<(r',v')$ if either $r>r'$ or $r=r'$ and $v<v'$. This order
makes  $\overline S$  a {\it maximum well ordered} poset, meaning that any subset of $\overline S$ 
has a maximal element. This property is needed for the subduction algorithm to terminate. 
See 
\cite[Example~3.10]{KavehManon16}.

We focus on the `large enough' case where $D$ is very ample and $Y$ is projectively normal in the projective embedding $Y\hookrightarrow\mathbb P^d$ associated to $D$; therefore $R$ is generated by $R^{(1)}$. Choose a
$g\in H^0(Y,\mathcal O(1))$ 
such that $D$ is the divisor of zeros of $g$. In this case the homogeneous coordinate ring $\C[\widehat Y]$ of the affine cone $\widehat Y$ over $Y$ is  isomorphic to $R$ via the map which sends $f\in\C[\widehat Y]_j$ to $t^j f/g^j\in R$, compare \cite[II, Exercise 5.14]{Hartshorne}.

More general versions of the following result can be found in  \cite[Theorem 1]{Anderson},
\cite[Section 7]{Kaveh}, and \cite{Tei03}. We follow mostly \cite{Anderson}, though  our conventions regarding the ordering $<$ are reversed. 
	
	\begin{prop}\label{p:degeneration}
	Let $Y, D, R,\overline{\val}$ and $\overline S$ be as above, %
	where $\overline\val$ has one-dimensional leaves, $D$ is very ample and $Y$ is projectively normal in the associated projective embedding $Y\hookrightarrow \mathbb P^d$. 
	Suppose that $\overline S$ is generated by its degree $1$ part $\overline S^{(1)}$. Let $C$ denote the cone spanned by $\overline S^{(1)}$, and  $\Delta$ the polytope in $\R^m$ such that $\{1\}\x\Delta$ is the intersection of $C$ with $\{1\}\x\R^m$. Assume $\Delta$ has the 
integer decomposition property.
	
Then there exists a flat family $\mathcal Y\to\mathbb A^1$ embedded in $\mathbb P^d\x\mathbb A^1\to\mathbb A^1$, such that the fiber over $0$ is a normal, projective toric variety $Y_0$, while the other fibers are isomorphic to $Y$. Moreover,  $\Delta$ is the moment polytope of $Y_0$ for its embedding into $\mathbb P^d$, and this embedding is projectively normal.
	\end{prop}
	
\begin{rem}
In contrast with the more general theorem \cite[Theorem 1]{Anderson}, 
we have added the %
assumptions that the semigroup $\overline S$  is generated in degree 1,
  and the polytope $\Delta$ has the integer decomposition property. (These will be true in our application in \cref{s:X}.) If these assumptions are removed, 
then 
$\mathbb P^d$ may need to be replaced by weighted projective space, and the limit toric variety $Y_0$ may not be normal. 
\end{rem}	
	
	\begin{proof} We sketch 
	the construction of the toric degeneration, mostly following \cite{Anderson}. The assumption on $\overline S$ implies that there exists a finite Khovanskii basis 
$\{\phi_{(1,\ell)}\mid \ell\in \mathcal L\}$ of $R$, where $\mathcal L$ denotes
the lattice points of $\Delta$ and 
$\overline\val(\phi_{(1,\ell)}) = (1,\ell)$.
Note that $|\mathcal L|=d+1$ and this Khovanskii basis is a vector space basis of $R^{(1)}$. The degeneration is obtained by applying (relative) $\Proj$ to a graded $\C[s]$-algebra $\mathcal R$ which is constructed from $R$ as follows. 

Consider the polynomial ring $A=\C[x_\ell; \ell\in \mathcal L]$, with the usual $\Z$-grading, as well as an extension of this grading to an $\overline S$-grading via $\deg(x_\ell):=(1,\ell)$.  The maps $ h: A\to R$ and $\overline h: A\to \gr_{\overline\val}R$ defined by
\begin{equation*}
h(x_\ell):=\phi_{(1,\ell)},\qquad \qquad \overline h(x_\ell):=\overline\phi_{(1,\ell)}
\end{equation*}
are homomorphisms of $\Z$-graded algebras and $\overline S$-graded algebras, respectively. 
Then the maximum well-ordered property of the ordering on $\overline S$ implies that the kernel of $h$ has a Gr\"obner basis $g_1,\dotsc, g_m$ whose $\overline S$-initial terms $\overline g_1,\dotsc, \overline g_m$ generate the kernel of $\overline h$. Moreover we can choose the $g_i$ to be homogeneous (say of degree $r_i$) and $\overline g_i$ homogeneous (say of degree $(r_i,v_i)$). Then one can find a linear projection $\pi:\Z\x\Z^m\to\Z$ (see \cite{Anderson}), such that the elements $\widetilde g_i\in A[s]$ defined by
\[
\widetilde g_i:=s^{-\pi(r_i,v_i)}\, g_i((s^{\pi(1,\ell)}x_\ell)_{\ell\in\mathcal L})
\]
are of the form $\overline g_i + s A_{>(r_i, v_i)}$. Moreover $\mathcal R:=
 A[s]/(\widetilde g_1,\dotsc, \widetilde g_m)$ is a flat $\C[s]$-algebra with $\mathcal R/s\mathcal R\cong A/(\overline g_1,\dotsc, \overline g_m)\cong\gr_{\overline\val}(R)$ and $\mathcal R[s\inv]\cong R\otimes\C[s,s\inv]$.
We therefore obtain a family $\mathcal Y$ of projective varieties over $\mathbb A^1$ such that the fiber over $0$ equals the projective toric variety with homogeneous coordinate ring $\gr_{\overline\val}(R)$, and all other fibers are isomorphic to $Y$. 
If we order the set  $\mathcal L$, so $\mathcal L=\{\ell_1,\dotsc, \ell_{d+1}\}$, then the description of $\mathcal R$ gives rise to the embedding of $\mathcal Y$ into $\mathbb P^d\x \mathbb A^1$.  Note that since $\overline\val$ has $1$-dimensional leaves, $\gr_{\overline\val}(R)\cong\C[\overline S]$.  
 Thus the zero fiber $Y_0$ in $\mathbb P^d$ has homogeneous coordinate ring $\C[\overline S]$.  From its degree $1$ part we see that the moment polytope of $Y_0$ is $\Delta$. And since $\Delta$ 
has the integer decomposition property, it follows directly that $Y_{0}$ is projectively normal, and in particular also normal. %
	\end{proof}

\subsection{Applications to the Grassmannian}\label{s:X} Now we consider $Y=\X$. We choose an $\mathcal{X}$-cluster seed $\Sigma^{\mathcal X}_G$ 
of type $\pi_{k,n}$ and the valuation $\val_G:\C(\X)\setminus\{0\}\to \Z^{\mathcal P_G}$ with one-dimensional leaves  (compare  \cref{l:okounkovlemma}). Recall
 that $L_r=H^0(\X,\mathcal O(rD_{n-k}))$, and 
that by \cref{thm:main},
 $\Delta_G :=\overline{\operatorname{ConvexHull}\left(\bigcup_r \frac{1}{r} \val_G(L_r)\right)}
$ is a rational polytope.

Let $R:=\bigoplus_j t^j L_j$ and consider the extended valuation $\overline \val_G:R\setminus\{0\}\to \Z\x\Z^{\mathcal P_G}$ as in \eqref{e:extendedval}. 
Note that $R$ is isomorphic to the homogeneous coordinate ring of $\X$ by  \eqref{e:projnormal}.  The valuation $\overline\val_G$ is again a valuation with $1$-dimensional leaves and we have the following result about $R$ in our setting.   

\begin{lem}\label{l:cone}
 Given  $(R,\overline{\val}_G$) as above, we define the  value semigroup 
\begin{equation}\label{e:SG}
\overline S_{G}:=\{(r,v)\mid r\in\Z_{\ge 0}, v\in\val_G(L_r)\}
\subseteq 
\Z\x\Z^{\mathcal P_G}.
\end{equation}
	Consider the  cone $\Cone(G)$
	in $\R\x\R^{\mathcal P_G}$ defined as 
the $\R_{\geq 0}$-span of
	vectors in $\{(1,w)\mid \text{ $w$ is a vertex of $\Delta_G$}\}$. 
Then $\overline S_{G}$ equals the semigroup $\Cone(G)\cap(\Z\x\Z^{\mathcal P_G})$
consisting 	 of lattice points of $\Cone(G)$. In particular the semigroup $\overline S_G$ is finitely generated, and hence we have a finite Khovanskii basis of $R$.
\end{lem}

\begin{proof}
Clearly $\overline S_G\subseteq \Cone(G)$, as follows from the construction of $\Delta_G$. The lemma says that conversely every lattice point in $\Cone(G)$ lies in $\overline S_G$, i.e. is of the form $(r,\val_G(f))$ for some $f\in L_r$. Equivalently if we fix $r$ it says that the lattice points of $r\NO_G$ agree with the image $\val_G(L_r)$ of the valuation map. But in the proof of \cref{thm:main} we saw that $\val_G(L_r)=\Lattice(r\Q_G)$ and $\Q_G=\Delta_G$. Thus we have shown that $\overline S_G$ is the semigroup of lattice points of $\Cone(G)$. Finally, $\Cone(G)$ is a rational convex cone by \cref{thm:main}. Therefore 
	by Gordon's lemma \cite[Section 1.2, Proposition 1]{Fulton} the semigroup of its lattice points (and therefore the semigroup $\overline S_G$) is finitely generated.
This completes the proof. 
\end{proof}

It is well-known (see e.g. \cite{IDP}) that 
for \emph{any} rational polytope $\Delta\subset \R^m$, there is an
$r\in \Z_{>0}$ such that $r\Delta$ has the integer decomposition property
(\cref{d:IntClosed}); this is an easy consequence of Gordan's lemma.
Therefore we can make the following definition.

\begin{defn}\label{c:RG}
Let $\RG$ denote the minimal positive integer such that 
the dilated polytope
$\RG\NO_G$ has the integer decomposition property.  
And let $\X_{\RG}\subset \mathbb P(\Sym^{\RG}({\bigwedge^k\C^n}))$ be the image of $\X$ after composing the Pl\"ucker embedding with the Veronese map of degree $\RG$. 
In other words $\X_{\RG}$ is the projective variety  obtained via the embedding of $\X$ associated to the ample divisor $\RG D_{n-k}$. 
\end{defn}
We let 
 $\C[\widehat\X_{\RG}]$ 
denote the homogeneous coordinate ring of $\X_{\RG}$.

\begin{defn}\label{d:RG} Associated to $\X_{\RG}$ we have
\[
R_{\RG}=\bigoplus_{j=0}^\infty
t^jH^0(Y,\mathcal O(j\RG D_{n-k})) \subset \C(\X)[t],
\]
with its extended valuation $\overline{\val}_{G,\RG}$
and the value semigroup 
\[
\overline S_{G,\RG}:=\{(r,v)\mid r\in\Z_{\ge 0}, v\in \val_G(H^0(\X,\mathcal O(r\,\RG D_{n-k})))\}.
\]
The semigroup $\overline S_{G,\RG}$ is also obtained by applying the map $(r,v)\to (\frac{1} {\RG} r,v) $ to $\overline S_G\cap (\RG\Z)\x\Z^{\mathcal P_G}$.  

Note that $\X_{\RG}$ is still projectively normal, and $R_{\RG}$ is isomorphic to 
 $\C[\widehat\X_{\RG}]$. The associated Newton-Okounkov  body is $\Delta_G(\RG D_{n-k})=\RG\Delta_G$.
\end{defn}

\begin{lem}\label{l:RG-version}
The semigroup
$\overline S_{G,\RG}$ %
is generated by the finite set  $\overline S_{G,\RG}^{(1)}=\{(1,v)\mid v\in\Lattice(\RG\Delta_G)\}$. 
In particular, for each lattice point $v\in \RG\Delta_G$ 
 we may choose an element $\phi_v\in L_{\RG}\setminus \{0\}$ such that $\val_G(\phi_v)=v$.  
Then the corresponding set $\{\phi_{(1,v)}:=t\phi_v\mid v\in\Lattice(\RG\Delta_G)\}$ is a finite Khovanskii basis of $R_{\RG}$, which lies in the $j=1$ graded component.
\end{lem}

\begin{proof}
Let $(j,v)\in \overline S_{G,\RG}$. Then because $\RG\Delta_G$  has the 
integer decomposition property and $v$ lies in its $j$-th dilation, we 
can write $v=\sum_{i=1}^j v_i$ where $v_i\in\Lattice(\RG\Delta_G)$.
Then  $(j,v)=\sum_{i=1}^j (1,v_i)$.  
Thus $(j,v)$ is in the semigroup generated by $\overline S^{(1)}_{G,\RG}$.
\end{proof}

\begin{cor}\label{c:SAGBI}
Suppose $G$ is represented by a plabic graph and $\RG=1$ (as in the case of $G=\Grec^{k,n}$, see~\cref{lem:intclosed}). Then
the set  $\{tP_{\lambda}/P_\Max\mid \lambda\in \mathcal P_G\}$ is a Khovanskii basis of
the algebra $R$.%
\end{cor}

\begin{proof}
This corollary is a special case of \cref{l:RG-version}, combined with \cref{c:nchoosek}.
\end{proof}

It now follows 
that associated to every seed $\Sigma_G^{\mathcal X}$ we obtain a flat degeneration of $\X$ to a  
toric variety. 

\begin{cor}\label{c:degeneration}
Suppose $\Sigma_G^{\mathcal X}$ is an arbitrary $\mathcal X$-cluster seed of type $\pi_{k,n}$ and $\RG\in\Z_{>0}$ is 
as in \cref{c:RG}. Then we have a flat degeneration of $\X$ to the normal projective toric variety $\X_0$ associated to the polytope $\RG\Delta_G$ (i.e. to  
the Newton-Okounkov body associated to the rescaled divisor $\RG D_{n-k}$).
\end{cor}

\begin{proof} 
By  \cref{l:RG-version} the ring $R_{\RG}$ has a finite Khovanskii basis which is contained in its $j=1$ graded component. By \cref{l:cone} the image of this Khovanskii basis under $\overline\val_{G,\RG}$ is precisely the set of all of the lattice points of $\{1\}\x\RG\Delta_{G}$ (after adjusting according to \cref{d:RG}).  By \cref{c:RG} the polytope $\Delta=\RG\Delta_{G}$ 
has the integer decomposition property. Therefore the conditions of \cref{p:degeneration} are satisfied and we obtain a toric degeneration of $\X$ to the toric variety $\X_0$ associated to $\RG\Delta_G$.
\end{proof}

\section{The cluster expansion of the superpotential and explicit inequalities for $\Q_G  = \Delta_G$}\label{s:clusterexpansion}

Since Newton-Okounkov bodies are defined as a closed convex hull of infinitely many points, very often it is difficult
to give a simple description of them.
However, now that we have proved that $\Delta_G = \Q_G$, we have an inequality description of $\Delta_G$ coming
from the cluster expansion of the superpotential $W$.
In the case where $G$ is a reduced plabic graph of type $\pi_{k,n}$, 
a combinatorial  formula for the cluster expansion of $W$ was given 
in \cite[Section 12]{MarshRietsch}, which followed from 
the work of Marsh and Scott~\cite[Theorem 1.1]{MarshScott}. 
We use this formula to give the inequality description of $\Delta_G = \Q_G$ when $G$
is a plabic graph.

\subsection{The cluster expansion of $W$}
Recall from 
\eqref{e:Wq} that 
$$W= \sum_{i=1}^{n}q^{\delta_{i,n-k}}\frac{p_{{\mu}_i^\square}}{p_{\mu_i}}.$$
Fix a cluster associated to a plabic graph $G$. In order to give the cluster expansion of $W$ it is enough
to give the cluster expansion of each term 
$ W_i={p_{\mu_i^\square}}/{p_{\mu_i}}$.

\begin{definition}[Edge weights]\label{d:edgeweights}
We assign monomials in the elements of $\PCG$ to edges of $G$ as follows.  
Let $v$ be the unique black vertex incident with an edge $e$.  The {\it weight} $w_e$ of $e$
is defined to be the product of the Pl\"ucker coordinates labelling the faces of $G$ which are 
incident with $v$ but not with the rest of $e$ (i.e.\ excluding the two faces on each 
side of $e$).  
(See \cref{fig:weighting} for an illustration of the rule.)
And the weight $w_M$ of a matching $M$ is the product of the weights of all edges in the matching.
\end{definition}

\begin{figure}[h]
\includegraphics[width=2.5cm]{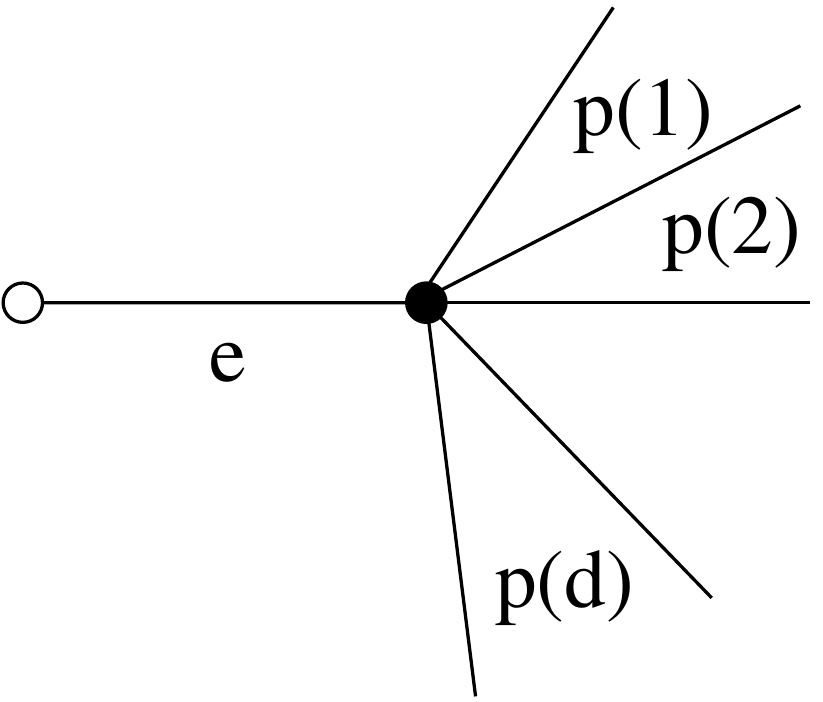}
\caption{Weighting of an edge: $w_e=\p(1)\p(2)\cdots \p(d)$.}
\label{fig:weighting}
\end{figure}

\begin{theorem}[{\cite[(12.2)]{MarshRietsch}}]
\label{eq:superpotentialexpansion}
Fix a reduced plabic graph $G$ of type $\pi_{k,n}$, and let 
$J^i$ be the $(n-k)$-element subset
$\{i+k+1, i+k+2,\dots, i-1\} \cup \{i+1\}$ (with indices considered modulo $n$ as usual).
Then we have that 
	\begin{equation}\label{eq:pM}
\frac{p_{\mu_i^\square}}{p_{\mu_i}} = 
		\sum_{M} p_M, \hspace{.5cm} \text{ where } \hspace{.5cm}
	p_M := \frac{w_M}{\prod_{p\in \PCG} p} p_{\mu_{i-1}} p_{\mu_{i+1}} p_{\mu_{i+2}} \dots p_{\mu_{i+k}}, 
\end{equation}
and the sum is over the set $\Match_G^{J^i}$ of all matchings $M$ of $G$ with boundary $J^i$, compare Section~\ref{s:matchings}.
\end{theorem}

\begin{example}\label{ex:expansion}
Let $k=3$, $n=5$, and $G$ the graph shown in \cref{fig:expanG25}.
We have $\mu_1 = \ydiagram{3}$, $\mu_2 = \ydiagram{3,3}$, $\mu_3 = \ydiagram{2,2}$,
$\mu_4 = \ydiagram{1,1}$, $\mu_5 = \emptyset$.
If $i=2$ then 
there is a unique matching $M$ of $G$ with boundary 
$J^i = J^2 = \{1,3\}$ 
as shown at the left.  This matching has weight 
$w_{M} = p_{\ydiagram{1}} p_{\ydiagram{2}}^2$, so 
$\frac{p_{\mu_i^\square}}{p_{\mu_i}} = 
\frac{p_{\ydiagram{2}}}{p_{\ydiagram{3,3}}}$.
(Recall that $p_{\emptyset} = 1$.)

If $i=3$,  
there are two matchings of $G$ with boundary 
$J^i = J^3 = \{2,4\}$. 
The maximal matching $M_3$ with boundary $J^3$ is shown at the right of \cref{fig:expanG25}
and it has weight $w_M = p_{\ydiagram{1}} p_{\ydiagram{2,2}} p_{\ydiagram{3}}$; so 
one of the two summands in 
$\frac{p_{\mu_i^\square}}{p_{\mu_i}}$ is $\frac{p_{\ydiagram{3}}}{p_{\ydiagram{2}}}$.
\end{example}

\begin{figure}[h]
\includegraphics[height=2.8cm]{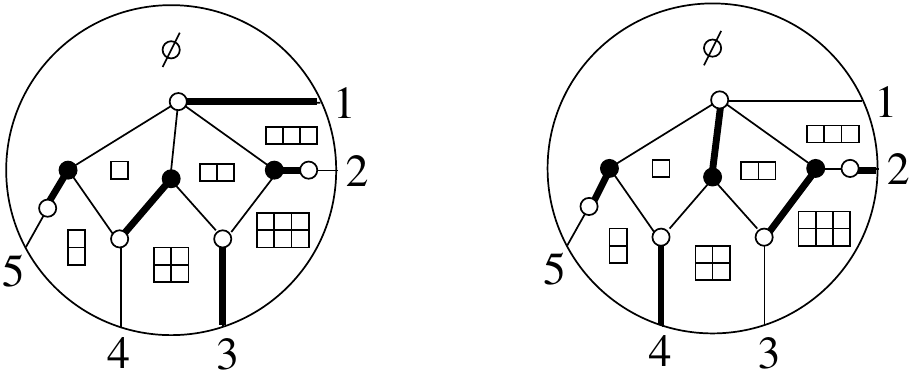}
\caption{A graph $G$ of type $\pi_{3,5}$ together with
 the unique matching with boundary
$J^2 = \{1,3\}$ (left) and the maximal matching with boundary $J^3=\{2,4\}$ (right).}
\label{fig:expanG25}
\end{figure}

Recall the definition of the superpotential polytope, \cref{l:B-modelPolytope}, and 
the generalized superpotential polytope, \cref{def:generalQ}. 
We can now use \cref{thm:main} and 
\cref{eq:superpotentialexpansion}
to write down the inequalities 
cutting out $\Q_G(r_1,\dotsc,r_n)$ and as a special case $\Q_G^r$.

\begin{prop}\label{p:matchingineqs}
Let $r_1,\dotsc, r_n\in\R$. The generalized superpotential polytope $\Q_G(r_1,\dotsc, r_n)$ is cut out by linear inequalities associated to matchings $M\in\Match^{J^i}_G$, where $1 \leq i \leq n$.
Namely for $M\in\Match_G^{J^{i}}$, the associated  inequality is 
\begin{equation}\label{e:speciali}
\Trop_G(p_M) 
	+ r_i \geq 0.
	\end{equation}
\end{prop}

By Theorem~\ref{thm:main}, which identifies the Newton-Okounkov body $\Delta_G$ with the superpotential polytope $\Q_G$, we obtain the following description of $\Delta_G$.

\begin{cor}\label{c:matchingineqs}
The Newton-Okounkov body $\Delta_G$ is a polytope determined by certain linear inequalities associated to matchings $M\in\Match^{J^i}_G$, where $1 \leq i \leq n$.
Namely for $M\in\Match_G^{J^{i}}$, the associated  inequality is 
\eqref{e:speciali}, where $r_{i}=0$ for $i\neq n-k$ and $r_{n-k}=1$.
\end{cor}

\begin{example}
We continue \cref{ex:expansion}.  
When $i=2=n-k$ we have the term $qW_i = q\frac{p_{\mu_i^\square}}{p_{\mu_i}} = 
q\frac{p_{\ydiagram{2}}}{p_{\ydiagram{3,3}}}$ of $W$, which gives rise to the inequality
$r+v_{\ydiagram{2}}-v_{\ydiagram{3,3}} \geq 0$.  
When $i=3$  we have that one of the summands in 
$W_i = \frac{p_{\mu_i^\square}}{p_{\mu_i}}$ is $\frac{p_{\ydiagram{3}}}{p_{\ydiagram{2}}}$,
which gives rise to the inequality $v_{\ydiagram{3}} - v_{\ydiagram{2}} \geq 0$.
This matches up with our description of $\Q_G$ 
from \cref{ex:Q}.
\end{example}

\section{The Newton-Okounkov polytope $\Delta_G(D)$ for more general divisors $D$}\label{s:generalD}
 
In this section we consider the Newton-Okounkov body of a 
general divisor $D$
which is an integer combination of the boundary divisors $D_j$ in $\X$, and 
we prove the analogue of $\Delta_G=\Q_G$. 

Recall that we defined a polytope $\Q_G(r_1,\dotsc,r_n)$  using tropicalisation of the individual summands $W_j$ of the superpotential, see \cref{def:generalQ}.
The result below generalizes 
\cref{thm:main}.

\begin{thm}\label{t:genD}
For the divisor $D=r_1D_1+ r_2 D_2+\dotsc + r_n D_n$ with $r_i\in\Z$, the associated Newton-Okounkov polytope is given by 
\begin{equation}\label{e:genD}
\Delta_G(D)=\Q_G(r_1,\dotsc,r_n).
\end{equation}
Moreover unless $r:=\sum r_j \geq 0$, both 
$\Delta_G(D)$ and $\Q_G(r_1,\dotsc,r_n)$ are the empty set.
\end{thm}

One way to prove \eqref{e:genD} is to try to mimic the proof of \cref{thm:main}:
to first prove it when $G = G^{\rect}_{{k,n}}$, 
and then to show that when one mutates away from 
$G$, the lattice points of both sides satisfy 
the tropical mutation formulas.  
For $\Q_G(r_1,\dotsc,r_n)$
this follows from \cref{c:PolytopeMutation}, but 
for $\Delta_G(D)$ the mutation property requires more work.
While one can complete the proof using this strategy,
we instead deduce the theorem from Theorem~\ref{thm:main}:
 we show that changing the divisor from $r D_{n-k}$ to $D=r_1 D_1 + \dotsc + r_n D_n$ 
with $r= r_1 + \dotsc + r_n$ 
translates both sides of \eqref{e:genD} by the same vector, see Proposition~\ref{p:DeltaTrans} and Proposition~\ref{p:GammaTrans}.

\begin{remark}\label{r:empty}
If $r=\sum r_j<0$, the line bundle $\mathcal O(D)=\mathcal O(r)$ has no non-zero global sections, and hence 
$\Delta_G(D)$ is clearly the empty set.   We demonstrate an analogous result for 
$\Q_G(r_1,\dotsc,r_n)$ in \cref{p:empty}. 
\end{remark}

If $r=0$ then $\mathcal O(D)$ is the structure sheaf $\mathcal O$ and $\Delta_G(D)$ consists of a single point. Namely 
\[
f_D:=\prod_{j=1}^{n} P_{\mu_{j}}^{-r_j} 
\]
is a rational function on $\X$ (since $\sum r_j=0$) and spans $H^0(\X,\mathcal O(D))\cong\C$. Moreover $H^0(\X,\mathcal O(sD))$ is the one-dimensional vector space spanned by $(f_D)^s$. By the definition of $\Delta_G(D)$ we immediately obtain
$
\Delta_G(D)=\{ v_D \}
$, 
where $v_D=-\sum_j r_j\val_G(P_{\mu_j})$ is the valuation $\val_G(f_D)$. 

In order to prove the theorem, we need the following lemma about the 
valuations of the
frozen variables (recall that the frozen variables are the Pl\"ucker coordinates 
$P_{\mu_j}$ for $0 \leq j \leq n-1$).

\begin{lem}\label{l:coefficient}
Fix $j\in\{0,1,\dotsc,n-1\}$ and let $e=e^{(j)}=\val_G(P_{\mu_j})$. Then we have
\begin{equation}\label{e:tropfrozen}
\Trop_G(p_{\muibox}/p_{\mui})(e^{(j)})=\delta_{i,j}-\delta_{i,n-k}=\begin{cases}1 & i=j\ne n-k,\\
-1 & i=n-k,\ j\ne n-k,\\
0 & \text{otherwise.}\end{cases}
\end{equation}
\end{lem}

\begin{proof}
We  check the identity for $\Trop_G(p_{\muibox}/p_{\mui})(e^{(j)})$ first in the case where $G$ is a plabic graph and $\mathcal P_G$ contains $\muibox$. Indeed, in this case the identity follows easily from the max diag formula,~\cref{t:ValuationsFormula}.
Now we can obtain any other seed from this one by a sequence of mutations. Since $e=\val_G(P_{{\mu_j}})$ mutates by the tropical $\mathcal A$-cluster mutation formula, 
see 
	\cref{p:PmuiMonomial}, 
	this implies that the quantity 
$\Trop_G(p_{\muibox}/p_{\mui})(e)$ is independent of the choice of seed $G$. Thus the identity~\eqref{e:tropfrozen} holds in general.
\end{proof}

\begin{proposition}\label{p:DeltaTrans}
Let $D=\sum r_iD_i$ and $r=\sum_j r_j$.  The Newton-Okounkov body  
$\Delta_G(D)$ is obtained from $\Delta_G(rD_{n-k})$ by translation. Explicitly, if $v_D:=-\sum_j r_j\val_G(P_{\mu_j})$, we have
\begin{equation}
\Delta_G(D)=\Delta_G(rD_{n-k})+v_D. 
\end{equation}
Note that $\Delta_G(rD_{n-k})=r \Delta_G$ if $r\ge 0$ and $\Delta_G(rD_{n-k})=\emptyset$ if $r<0$, see Remark~\ref{r:empty}.
\end{proposition}

\begin{proof}
We may suppose that $r=\sum_j r_j\ge 0$. 
To show that 
$\Delta_G(D)=r \Delta_G+v_D$ 
it suffices to check that for every $s\in\Z_{>0}$,
\begin{equation}\label{e:NeedForGeneralDelta}
\frac 1s\val_G(L_{sD})=%
\frac 1 s\val_G(L_{sr})+v_D.
\end{equation}
However for any $D=\sum r_jD_j$ with $r=\sum r_j$ we have an isomorphism of vector spaces
 \[
m: L_r\to L_{D} \quad \text{ given by }\quad
f  \mapsto  f \frac{P_{\max}^{r}}{\prod_j P_{\mu_j}^{r_j}}.
\]
This isomorphism shifts valuations and gives the equality $\val_G(L_{D})=\val_G(L_{r})+v_D$. If we replace $D$ by $sD$, then the resulting equation for $\val_G(L_{sD})$ implies \eqref{e:NeedForGeneralDelta}. This proves the desired formula for $\Delta_G(D)$.  
\end{proof}

\begin{proposition}\label{p:GammaTrans}
Let $r_1,\dotsc, r_n\in\R$ and $r:=\sum_j r_j$.  Then
$\Q_G(r_1,\dots,r_n)$ is related to $\Q_G^r$ by translation, 
\begin{equation}
\Q_G(r_1,\dots,r_n)= \Q_G^r+v_D \text{ where }
v_D:=-\sum_j r_j\val_G(P_{\mu_j}).
\end{equation}
\end{proposition}

\begin{proof}
We want to show that the map $\R^{\mathcal P_{G}}\to\R^{\mathcal P_{G}}$ which sends $v$ to $d=v+v_D$ bijectively takes $\Q_G^r$ to $\Q_G(r_1,\dotsc, r_n)$. Since $\Q_G(r_1,\dotsc, r_n)$ is by definition the intersection of the sets $\PosSet^G_{(r_i)}(W_i):=\{d\mid \Trop_G(p_\muibox/p_\mui)(d)+r_i\ge 0
\}$, it suffices to show the analogous translation property for each such set. 

Note that in general $
\Trop_G(p_\muibox/p_\mui)(d)=\min_M(\Trop_G(p_M)(d))$,
where $p_\muibox/p_\mui=\sum_M p_M$ is the expansion of $p_\muibox/p_\mui$ as sum of Laurent monomials in the cluster variables associated to $G$. In the special case where $\muibox\in\mathcal P_G$ however, $\Trop_G(p_\muibox/p_\mui)(d)=d_{\muibox}-d_{\mui}$ is linear. 

Let us assume first that $\muibox\in\mathcal P_G$. In this case by linearity we have that, for any $v\in\R^{\mathcal P_G}$, 
\begin{equation}
\Trop_G(p_{\muibox}/p_{\mui})(v+v_D)=\Trop_G(p_{\muibox}/p_{\mui})(v)+\Trop_G(p_{\muibox}/p_{\mui})(v_D)=\Trop_G(p_{\muibox}/p_{\mui})(v)-r_i + r\delta_{i,n-k},
\end{equation}
where we have evaluated $\Trop_G(p_{\muibox}/p_{\mui})(v_D)$ using Lemma~\ref{l:coefficient}. As a consequence
\begin{equation}\label{e:Qtranslation}
\Trop_G(p_{\muibox}/p_{\mui})(v+v_D)+r_i= \Trop_G(p_{\muibox}/p_{\mui})(v) + r\delta_{i,n-k}.
\end{equation}
From \eqref{e:Qtranslation} it follows that $v+v_D$ lies in 
$\PosSet^G_{(r_i)}(W_i)$ 
if and only if $v$ lies in 
$\PosSet^G_{(r\delta_{i,n-k})}(W_i)$. Thus we have that, whenever $p_{\muibox}$ is a cluster variable in the $\mathcal A$-cluster associated to $G$,
\begin{equation}\label{e:ConeiTranslate}
\PosSet^G_{(r_i)}(W_i)=\PosSet^G_{(r\delta_{i,n-k})}(W_{i})+v_D.
\end{equation}

We would like to apply a tropicalized $\mathcal A$-cluster mutation $\Psi_{G,G'}$ to both sides of \eqref{e:ConeiTranslate} to obtain the analogous identity for arbitrary seeds.  
Let us now write $v_{D,G}$ instead of $v_D$ to emphasise the dependence on $G$. Note that, since $v_{D,G}$ is a linear combination of elements of the form $\val_G(P_{\mui})$, the results of Section~\ref{s:frozen} imply that $v_{D,G}$  is {\it balanced} and transforms via tropicalized $\mathcal A$-cluster mutation if we mutate $G$. These two properties imply that for any $e\in\R^{\mathcal P_G}$, 
\begin{equation}\label{e:additive}
\Psi_{G,G'}(e+v_{D,G})= \Psi_{G,G'}(e)+v_{D,G'}.
\end{equation}
On the other hand by Lemma~\ref{l:PosSetTrop},
\begin{equation}\label{e:cone1mutation}
\Psi_{G,G'}(\PosSet^G_{(r_i)}(W_i))=\PosSet^{G'}_{(r_i)}(W_i) \quad\text{ and }\quad \Psi_{G,G'}(\PosSet^G_{(r\delta_{i,n-k})}(W_i))=\PosSet^{G'}_{(r\delta_{i,n-k})}(W_i).
\end{equation}
From \eqref{e:additive} and \eqref{e:cone1mutation} put together, we obtain that the translation identity \eqref{e:ConeiTranslate} is preserved under mutation. Thus \eqref{e:ConeiTranslate} holds for all seeds (and all $i$). 

As a consequence the polytope $\Q_G(r_1,\dotsc, r_n)$ is always the shift by $v_{D,G}$ of the polytope $\Q_G^r$.
\end{proof}

\begin{proposition}\label{p:empty}
If $r=\sum r_j<0$, then 
$\Q_G(r_1,\dotsc,r_n)$ is the empty set.
\end{proposition}

\begin{proof}
	By  \cref{p:GammaTrans}, $\Q_G(r_1,..,r_n)$ is related to $\Q_G^r$ by a translation. Therefore it suffices to check that $\Q_G^r$ is the empty set for $r<0$. In the case where $G$ is the rectangles cluster, $\Q_G^r$ is isomorphic via a unimodular transformation to the Gelfand-Tsetlin polytope (see Definition~\ref{d:GT}), which is clearly empty if $r<0$, and a point if $r=0$.  Now we know from \cref{c:PolytopeMutation} that the polytopes $\Q^r_G$ transform via tropicalized $\mathcal{A}$-cluster mutation when we mutate $G$. Therefore $\Q^r_G$ is also the empty set for a general seed.
\end{proof}

\begin{proof}[Proof of \cref{t:genD}]
If $r<0$ the result follows from \cref{r:empty} and  \cref{p:empty}.
Now suppose $r\ge 0$. By \cref{thm:main}, we have that $\Delta_G=\Delta_G(D_{n-k})$ 
and $\Q_G$ coincide, which implies that $r\Delta_G = \Q_G^r$, see Remark~\ref{rem:dilation}.  But now by 
\cref{p:DeltaTrans} and 
\cref{p:GammaTrans}, both $\Delta_G(D)$ and 
$\Q_G(r_1,\dotsc,r_n)$ are obtained from $r\Delta_G$ and $\Q^r_G$ by translation by the same vector.
\end{proof}

\section{The highest degree valuation and Pl\"ucker coordinate valuations}\label{sec:highvaluation}

Recall from \cref{de:val} that given an $\mathcal{X}$-seed $G$ of type 
$\pi_{k,n}$, we defined a valuation $\val_G: \C(\mathbb{X})\setminus \{0\}
\to \Z^{\mathcal{P}_G}$ using the lowest order term.
When $G$ is a plabic graph, the flow polynomials $P_{\lambda}^G$
express the Pl\"ucker coordinates in terms of the $\mathcal{X}$-coordinates,
and have strongly minimal and maximal terms, see \cref{prop:strongminimal}.
In \cref{t:ValuationsFormula}, we gave an explicit formula 
for the Pl\"ucker coordinate valuations $\val_G(P_{\lambda})$, such that 
the $\mu$-th coordinate $\val_G(P_{\lambda})_{\mu}$ 
is related to the smallest degree of $q$ that appears when the quantum product 
of two Schubert classes $\sigma_{\mu} \star \sigma_{\lambda^c}$ is expanded
in the Schubert basis.

In this section we briefly explain what is the analogue of \cref{t:ValuationsFormula}
if we define our valuation in terms of the highest order term instead of the lowest order
term.  We will find that our formula is again connected to quantum cohomology,
but this time to the \emph{highest} degree of $q$ that appears in a corresponding
quantum product.  In order to state our formula we first need to introduce some
notation.

\begin{definition}
Let $\mu$ be a partition in $\mathcal{P}_{k,n}$, so $\mu$ lies in an
$(n-k) \times k$ rectangle.  We let $\diag_0(\mu)$ denote the number
of boxes in $\mu$ along the main diagonal (with slope $-1$).  

Let us identify $\mu$ with the word $\omega(\mu)=(\oomega_1,\dots,\oomega_n)$ 
in $\{0,1\}^n$ obtained by reading the 
border of $\mu$ from southwest to northeast and associating a $0$ to each
horizontal step and a $1$ to each vertical step.  
Then the cyclic shift $S$ acts on partitions in $\mathcal{P}_{k,n}$ by mapping 
the partition corresponding to $(\oomega_1,\dots,\oomega_n)$ to the
partition corresponding to $(\oomega_2,\dots,\oomega_n,\oomega_1)$.
\end{definition}

\begin{example}
Let $\mu = \ydiagram{6,4,4,2}$, viewed inside a $4 \times 6$
rectangle.  Then 
$\omega(\mu) = (0,0,1,0,0,1,1,0,0,1)$, and  $\diag_0(\mu) = 3$.
Applying the cyclic shift to $\omega(\mu)$ gives 
$ (0,1,0,0,1,1,0,0,1,0)$, and hence
$S(\mu) = \ydiagram{5,3,3,1}$.
\end{example}

For partitions in $\mathcal{P}_{k,n}$,
$S^{n-k}(\emptyset) = S^{n-k}(1^{n-k} 0^k) = 0^k 1^{n-k} = \max$,
where $\max$ is the  $(n-k) \times k$ rectangle.

\begin{theorem}\label{t2:ValuationsFormula}
Let $G$ be a reduced plabic graph of type $\pi_{k,n}$.
Let $\val^G(P_{\lambda}) \in \Z^{\mathcal{P}_G}$ denote the exponent 
vector of the strongly {\em maximal} term 
of the flow polynomial $P_{\lambda}^G$. 
Then we have that 
\begin{equation}\label{val^G}
\val^G(P_{\lambda})_{\mu} = \diag_0(\mu) - \maxdiag(\lambda \setminus S^{n-k}(\mu)).
\end{equation}
\end{theorem}

Note that by \cite[Theorem 8.1]{PostnikovDuke}, the right-hand side of \eqref{val^G}
is equal to the
largest degree $d$ such that $q^d$ appears in the quantum product of the Schubert
classes $\sigma_{\mu}\star \sigma_{\lambda^c}$ in the quantum cohomology ring  $QH^*(Gr_k(\C^n))$, when this
product is expanded in the Schubert basis.

We now sketch the proof of 
\cref{t2:ValuationsFormula}, which is analogous to the proof of 
\cref{t:ValuationsFormula}.
\begin{proof}
Recall from 
 \cref{prop:strongminimal} that each flow polynomial 
$P_{\lambda} = P_{\lambda}^G$ has a maximal flow; its exponent vector 
is precisely $\val^G(P_{\lambda})$.  Now, following the proof of 
\cref{thm1:tropcluster}, we show that when we mutate $G$ at a square face,
obtaining $G'$, for any $\lambda$, the Pl\"ucker coordinate valuations
$\val^G(P_{\lambda})$ and $\val^{G'}(P_{\lambda})$ are related by 
	the tropicalized $\mathcal{A}$-cluster mutation $\Psi^{G,G'}$.  Here
$\Psi^{G,G'}$ is defined the same way as $\Psi_{G,G'}$ except that we
replace $\min$ by $\max$.  As in the proof of \cref{thm1:tropcluster},
the main step is to analyze how strongly maximal flows
change under an oriented square move.

Next, we prove an analogue of \cref{p:rect-val}, which gives
the formula for Pl\"ucker coordinate valuations when 
$G = G^{\rect}_{k,n}$.  Concretely, one can give a combinatorial proof that 
\begin{equation*}
\val^G(P_{\lambda})_{i \times j} = \diag_0(i \times j) - 
\maxdiag (\lambda \setminus S^{n-k}(i \times j)).
\end{equation*}
To complete the proof, we follow the proof of \cref{t:ValuationsFormula}
and in particular \cref{t:matrix}, explicitly constructing an 
element $x^{\lambda}(t)$ 
of the Grassmannian over Laurent series, such that 
$$\ValKhigh(p_{\mu}(x^{\lambda}(t)))  = 
\diag_0(\mu) - \maxdiag (\lambda \setminus S^{n-k}(\mu)).$$
But now we have to work with Laurent series (or generalized
Puiseux series) in $t\inv$, that is, series in $t$ whose terms are bounded from above
so that there always exists a maximal exponent.  Then 
$\ValKhigh(h(t))$ records the maximal exponent which occurs among the terms
of $h(t)$.
\end{proof}

\bibliographystyle{alpha}
\bibliography{bibliography}

\newcommand{\etalchar}[1]{$^{#1}$}
\def\cprime{$'$} \def\cprime{$'$} \def\cprime{$'$} \def\cprime{$'$}
  \def\cprime{$'$} \def\cprime{$'$}
\begin{thebibliography}{BCFKvS00}

\bibitem[AB04]{AlexeevBrion}
V.~Alexeev and M.~Brion.
\newblock Toric degenerations of spherical varieties.
\newblock {\em Sel. Math. (N.S.)}, 10(4):453--478, 2004.

\bibitem[ACGK12]{ACGK}
Mohammad Akhtar, Tom Coates, Sergey Galkin, and Alexander~M. Kasprzyk.
\newblock Minkowski polynomials and mutations.
\newblock {\em SIGMA Symmetry Integrability Geom. Methods Appl.}, 8:Paper 094,
  17, 2012.

\bibitem[Akh]{Akhtar:PolygonalQuivers}
M.~E. Akhtar.
\newblock Polygonal quivers.
\newblock in preparation.

\bibitem[And13]{Anderson}
Dave Anderson.
\newblock Okounkov bodies and toric degenerations.
\newblock {\em Math. Ann.}, 356(3):1183--1202, 2013.

\bibitem[BCFKvS00]{BC-FKvS}
Victor~V. Batyrev, Ionu{\c{t}} Ciocan-Fontanine, Bumsig Kim, and Duco van
  Straten.
\newblock Mirror symmetry and toric degenerations of partial flag manifolds.
\newblock {\em Acta Math.}, 184(1):1--39, 2000.

\bibitem[BFF{\etalchar{+}}16]{Hering}
L.~Bossinger, X.~Fang, G.~Fourier, M.~Hering, and M.~Lanini.
\newblock Toric degenerations of {G}r(2,n) and {G}r(3,6) via plabic graphs.
\newblock arXiv:1612.03838 [math.CO], 2016.

\bibitem[BFMMC18]{BFMC}
L.~Bossinger, B.~Frias-Medina, T.~Magee, and A.~Najera Chavez.
\newblock Toric degenerations of cluster varieties and cluster duality.
\newblock arXiv:1809.08369 [math.AG], 2018.

\bibitem[BFZ96]{BFZ}
Arkady Berenstein, Sergey Fomin, and Andrei Zelevinsky.
\newblock Parametrizations of canonical bases and totally positive matrices.
\newblock {\em Adv. Math.}, 122(1):49--149, 1996.

\bibitem[BFZ05]{BFZ2}
Arkady Berenstein, Sergey Fomin, and Andrei Zelevinsky.
\newblock Cluster algebras. {III}. {U}pper bounds and double {B}ruhat cells.
\newblock {\em Duke Math. J.}, 126(1):1--52, 2005.

\bibitem[BK07]{BK:GeometricCrystalsII}
Arkady Berenstein and David Kazhdan.
\newblock Geometric and unipotent crystals. {II}. {F}rom unipotent bicrystals
  to crystal bases.
\newblock In {\em Quantum groups}, volume 433 of {\em Contemp. Math.}, pages
  13--88. Amer. Math. Soc., Providence, RI, 2007.

\bibitem[Bri87]{Brion:moment}
Michel Brion.
\newblock Sur l'image de l'application moment.
\newblock In {\em S\'eminaire d'alg\`ebre {P}aul {D}ubreil et {M}arie-{P}aule
  {M}alliavin ({P}aris, 1986)}, volume 1296 of {\em Lecture Notes in Math.},
  pages 177--192. Springer, Berlin, 1987.

\bibitem[BZ97]{BZ}
Arkady Berenstein and Andrei Zelevinsky.
\newblock Total positivity in {S}chubert varieties.
\newblock {\em Comment. Math. Helv.}, 72(1):128--166, 1997.

\bibitem[CHHH14]{IDP}
David~A. Cox, Christian Haase, Takayuki Hibi, and Akihiro Higashitani.
\newblock Integer decomposition property of dilated polytopes.
\newblock {\em Electron. J. Combin.}, 21(4):Paper 4.28, 17, 2014.

\bibitem[EHX97]{EHX}
Tohru Eguchi, Kentaro Hori, and Chuan-Sheng Xiong.
\newblock Gravitational quantum cohomology.
\newblock {\em Int. J. Mod. Phys.}, A12:1743--1782, 1997.

\bibitem[FG06]{FG1}
Vladimir Fock and Alexander Goncharov.
\newblock Moduli spaces of local systems and higher {T}eichm\"uller theory.
\newblock {\em Publ. Math. Inst. Hautes \'Etudes Sci.}, (103):1--211, 2006.

\bibitem[FG09]{FG}
Vladimir~V. Fock and Alexander~B. Goncharov.
\newblock Cluster ensembles, quantization and the dilogarithm.
\newblock {\em Ann. Sci. \'Ec. Norm. Sup\'er. (4)}, 42(6):865--930, 2009.

\bibitem[Ful93]{Fulton}
William Fulton.
\newblock {\em Introduction to toric varieties}, volume 131 of {\em Annals of
  Mathematics Studies}.
\newblock Princeton University Press, Princeton, NJ, 1993.
\newblock The William H. Roever Lectures in Geometry.

\bibitem[FW04]{FW}
W.~Fulton and C.~Woodward.
\newblock On the quantum product of {S}chubert classes.
\newblock {\em J. Algebraic Geom.}, 13(4):641--661, 2004.

\bibitem[FZ02]{ca1}
Sergey Fomin and Andrei Zelevinsky.
\newblock Cluster algebras. {I}. {F}oundations.
\newblock {\em J. Amer. Math. Soc.}, 15(2):497--529 (electronic), 2002.

\bibitem[FZ07]{ca4}
Sergey Fomin and Andrei Zelevinsky.
\newblock Cluster algebras. {IV}. {C}oefficients.
\newblock {\em Compos. Math.}, 143(1):112--164, 2007.

\bibitem[GHK15]{GHK}
Mark Gross, Paul Hacking, and Sean Keel.
\newblock Birational geometry of cluster algebras.
\newblock {\em Algebr. Geom.}, 2(2):137--175, 2015.

\bibitem[GHKK14]{GHKK}
Mark Gross, Paul Hacking, Sean Keel, and Maxim Kontsevich.
\newblock Canonical bases for cluster algebras, 2014.
\newblock preprint, {\tt arXiv:1411.1394}.

\bibitem[Giv97]{Givental:fullflag}
Alexander~B. Givental.
\newblock Stationary phase integrals, quantum {T}oda lattices, flag manifolds
  and the mirror conjecture.
\newblock In {\em Topics in Singularity Theory: V. I. ArnoldÕs 60th Anniversary
  Collection}, volume 180 of {\em AMS translations, Series 2}, pages 103 --
  116. American Mathematical Society, 1997.

\bibitem[GS15]{GoncharovShen}
Alexander Goncharov and Linhui Shen.
\newblock Geometry of canonical bases and mirror symmetry.
\newblock {\em Invent. Math.}, 202(2):487--633, 2015.

\bibitem[GS16]{GoncharovShen2}
A.~Goncharov and L.~Shen.
\newblock Donaldson-{T}homas transformations of moduli spaces of ${G}$-local
  systems.
\newblock arXiv:1602.0647 [math.AG], 2016.

\bibitem[GSSV12]{GSV:Grass}
M.~Gekhtman, M.~Shapiro, A.~Stolin, and A.~Vainshtein.
\newblock Poisson structures compatible with the cluster algebra structure in
  {G}rassmannians.
\newblock {\em Lett. Math. Phys.}, 100(2):139--150, 2012.

\bibitem[GSV03]{GSV03}
Michael Gekhtman, Michael Shapiro, and Alek Vainshtein.
\newblock Cluster algebras and {P}oisson geometry.
\newblock {\em Mosc. Math. J.}, 3(3):899--934, 1199, 2003.
\newblock \{Dedicated to Vladimir Igorevich Arnold on the occasion of his 65th
  birthday\}.

\bibitem[GT50]{GT}
I.~M. Gel{\cprime}fand and M.~L. Tsetlin.
\newblock Finite-dimensional representations of the group of unimodular
  matrices.
\newblock {\em Doklady Akad. Nauk SSSR (N.S.)}, 71:825--828, 1950.

\bibitem[GW11]{GrossWallach}
Benedict~H. Gross and Nolan~R. Wallach.
\newblock On the {H}ilbert polynomials and {H}ilbert series of homogeneous
  projective varieties.
\newblock In {\em Arithmetic geometry and automorphic forms}, volume~19 of {\em
  Adv. Lect. Math. (ALM)}, pages 253--263. Int. Press, Somerville, MA, 2011.

\bibitem[Har77]{Hartshorne}
Robin Hartshorne.
\newblock {\em Algebraic geometry}.
\newblock Springer-Verlag, New York-Heidelberg, 1977.
\newblock Graduate Texts in Mathematics, No. 52.

\bibitem[HK15]{HaradaKaveh}
Megumi Harada and Kiumars Kaveh.
\newblock Integrable systems, toric degenerations and {O}kounkov bodies.
\newblock {\em Invent. Math.}, 202(3):927--985, 2015.

\bibitem[Jud18]{Judd}
Jamie Judd.
\newblock Tropical critical points of the superpotential of a flag variety.
\newblock {\em J. Algebra}, 497:102--142, 2018.

\bibitem[Kas09]{Kashiwara:PC}
M.~Kashiwara.
\newblock Personal communication, 2009.

\bibitem[Kav05]{Kaveh0}
Kiumars Kaveh.
\newblock S{AGBI} bases and degeneration of spherical varieties to toric
  varieties.
\newblock {\em Michigan Math. J.}, 53(1):109--121, 2005.

\bibitem[Kav15]{Kaveh}
Kiumars Kaveh.
\newblock Crystal bases and {N}ewton-{O}kounkov bodies.
\newblock {\em Duke Math. J.}, 164(13):2461--2506, 2015.

\bibitem[KK08]{KavehKhovanskii08}
Kiumars Kaveh and Askold Khovanskii.
\newblock Convex bodies and algebraic equations on affine varieties, 2008.
\newblock preprint, {\tt arXiv:0804.4095}.

\bibitem[KK12a]{KavehKhovanskii}
K.~Kaveh and A.~G. Khovanskii.
\newblock Newton-{O}kounkov bodies, semigroups of integral points, graded
  algebras and intersection theory.
\newblock {\em Ann. of Math. (2)}, 176(2):925--978, 2012.

\bibitem[KK12b]{KK2}
Kiumars Kaveh and Askold~G. Khovanskii.
\newblock Convex bodies associated to actions of reductive groups.
\newblock {\em Mosc. Math. J.}, 12(2):369--396, 461, 2012.

\bibitem[KLM12]{KLM}
Alex K\"uronya, Victor Lozovanu, and Catriona Maclean.
\newblock Convex bodies appearing as {O}kounkov bodies of divisors.
\newblock {\em Adv. Math.}, 229(5):2622--2639, 2012.

\bibitem[KLS13]{KLS:positroid}
Allen Knutson, Thomas Lam, and David~E. Speyer.
\newblock Positroid varieties: juggling and geometry.
\newblock {\em Compos. Math.}, 149(10):1710--1752, 2013.

\bibitem[KM16]{KavehManon16}
Kiumars Kaveh and Christopher Manon.
\newblock Khovanskii bases, higher rank valuations and tropical geometry, 2016.
\newblock preprint, {\tt arXiv:1610.00298v3}.

\bibitem[KW14]{KW}
Yuji Kodama and Lauren Williams.
\newblock K{P} solitons and total positivity for the {G}rassmannian.
\newblock {\em Invent. Math.}, 198(3):637--699, 2014.

\bibitem[LM09]{LazarsfeldMustata}
Robert Lazarsfeld and Mircea Musta{\c{t}}{\u{a}}.
\newblock Convex bodies associated to linear series.
\newblock {\em Ann. Sci. \'Ec. Norm. Sup\'er. (4)}, 42(5):783--835, 2009.

\bibitem[LS15]{LeeSchiffler}
Kyungyong Lee and Ralf Schiffler.
\newblock Positivity for cluster algebras.
\newblock {\em Ann. of Math. (2)}, 182(1):73--125, 2015.

\bibitem[Lus90]{Lus:CanonBasis}
G.~Lusztig.
\newblock Canonical bases arising from quantized enveloping algebras.
\newblock {\em J. Amer. Math. Soc.}, 3(2):447--498, 1990.

\bibitem[Lus94]{Lusztig3}
G.~Lusztig.
\newblock Total positivity in reductive groups.
\newblock In {\em Lie theory and geometry}, volume 123 of {\em Progr. Math.},
  pages 531--568. Birkh\"auser Boston, Boston, MA, 1994.

\bibitem[Lus10]{Lus:Quantum}
George Lusztig.
\newblock {\em Introduction to quantum groups}.
\newblock Modern Birkh\"auser Classics. Birkh\"auser/Springer, New York, 2010.
\newblock Reprint of the 1994 edition.

\bibitem[Mag15]{Magee}
Timothy Magee.
\newblock Fock-{G}oncharov conjecture and polyhedral cones for $u\subset sl_n$
  and base affine space $sl_n/u$.
\newblock arXiv:1502.03769 [math.AG], 2015.

\bibitem[Mar10]{Markwig:Puiseux}
T.~Markwig.
\newblock A field of generalised {P}uiseux series for tropical geometry.
\newblock {\em Rend. Semin. Mat. Univ. Politec. Torino}, 68(1):79--92, 2010.

\bibitem[MR04]{MR}
R.~J. Marsh and K.~Rietsch.
\newblock Parametrizations of flag varieties.
\newblock {\em Represent. Theory}, 8:212--242 (electronic), 2004.

\bibitem[MR13]{MarshRietsch}
R.~Marsh and K.~Rietsch.
\newblock The {B}-model connection and mirror symmetry for {G}rassmannians,
  2013.
\newblock preprint, {\tt arXiv:1307.1085}.

\bibitem[MS15]{MacLaganSturmfels}
Diane Maclagan and Bernd Sturmfels.
\newblock {\em Introduction to tropical geometry}, volume 161 of {\em Graduate
  Studies in Mathematics}.
\newblock American Mathematical Society, Providence, RI, 2015.

\bibitem[MS16a]{MarshScott}
R.~J. Marsh and J.~S. Scott.
\newblock Twists of {P}l\"ucker coordinates as dimer partition functions.
\newblock {\em Comm. Math. Phys.}, 341(3):821--884, 2016.

\bibitem[MS16b]{MullerSpeyer}
G.~Muller and D.~Speyer.
\newblock The twist for positroids, 2016.
\newblock preprint, {\tt arXiv:1606.08383 [math.CO]}.

\bibitem[Mul16]{Muller:PC}
G.~Muller.
\newblock Personal communication, 2016.

\bibitem[NU14]{nohara_ueda}
Yuichi Nohara and Kazushi Ueda.
\newblock Toric degenerations of integrable systems on {G}rassmannians and
  polygon spaces.
\newblock {\em Nagoya Mathematical Journal}, 214:125?168, 2014.

\bibitem[Oko96]{Ok96}
A.~Okounkov.
\newblock Brunn-{M}inkowski inequality for multiplicities.
\newblock {\em Invent. Math.}, 125(3):405--411, 1996.

\bibitem[Oko98]{Okounkov:symplectic}
A.~Okounkov.
\newblock Multiplicities and {N}ewton polytopes.
\newblock In {\em KirillovÕs seminar on representation theory}, volume 181 of
  {\em Amer. Math. Soc. Transl. Ser. 2}, pages 231--244. Amer. Math. Soc.,
  Providence, RI, 1998.

\bibitem[Oko03]{Ok03}
A.~Okounkov.
\newblock Why would multiplicities be log-concave?
\newblock In {\em The orbit method in geometry and physics ({M}arseille,
  2000)}, volume 213 of {\em Progr. Math.}, pages 329--347. Birkh\"auser
  Boston, Boston, MA, 2003.

\bibitem[Pos]{Postnikov}
A.~Postnikov.
\newblock Total positivity, {G}rassmannians, and networks.
\newblock Preprint. Available at
  \texttt{http://www-math.mit.edu/~apost/papers/tpgrass.pdf}.

\bibitem[Pos05]{PostnikovDuke}
Alexander Postnikov.
\newblock Affine approach to quantum {S}chubert calculus.
\newblock {\em Duke Math. J.}, 128(3):473--509, 2005.

\bibitem[Pro93]{Propp}
Jim Propp.
\newblock Lattice structure for orientations of graphs, 1993.
\newblock preprint, {\tt arXiv:0209005v1}.

\bibitem[PSW07]{PSWv1}
Alexander Postnikov, David Speyer, and Lauren Williams.
\newblock Matching polytopes, toric geometry, and the non-negative part of the
  {G}rassmannian, 2007.
\newblock preprint, {\tt arXiv:0706.2501v1}.

\bibitem[PSW09]{PSW}
Alexander Postnikov, David Speyer, and Lauren Williams.
\newblock Matching polytopes, toric geometry, and the totally non-negative
  {G}rassmannian.
\newblock {\em J. Algebraic Combin.}, 30(2):173--191, 2009.

\bibitem[Rie06]{RietschNagoya}
K.~Rietsch.
\newblock A mirror construction for the totally nonnegative part of the
  {P}eterson variety.
\newblock {\em Nagoya Math. J.}, 183:105--142, 2006.

\bibitem[Rie08]{Rietsch}
K.~Rietsch.
\newblock A mirror symmetric construction of {$qH\sp \ast\sb T(G/P)\sb {(q)}$}.
\newblock {\em Adv. Math.}, 217(6):2401--2442, 2008.

\bibitem[Sco06]{Scott}
J.~Scott.
\newblock Grassmannians and cluster algebras.
\newblock {\em Proc. London Math. Soc. (3)}, 92(2):345--380, 2006.

\bibitem[SW05]{SpeyerWilliams}
David Speyer and Lauren Williams.
\newblock The tropical totally positive {G}rassmannian.
\newblock {\em J. Algebraic Combin.}, 22(2):189--210, 2005.

\bibitem[SW18]{ShenWeng}
L.~Shen and D.~Weng.
\newblock Cyclic sieving and cluster duality for grassmannians.
\newblock arXiv:1803.06901 [math.RT], 2018.

\bibitem[Tal08]{Talaska}
Kelli Talaska.
\newblock A formula for {P}l\"ucker coordinates associated with a planar
  network.
\newblock {\em Int. Math. Res. Not. IMRN}, 2008, 2008.

\bibitem[Tei03]{Tei03}
Bernard Teissier.
\newblock Valuations, deformations, and toric geometry.
\newblock In {\em Valuation theory and its applications, {V}ol. {II}
  ({S}askatoon, {SK}, 1999)}, volume~33 of {\em Fields Inst. Commun.}, pages
  361--459. Amer. Math. Soc., Providence, RI, 2003.

\bibitem[TW13]{TalaskaWilliams}
Kelli Talaska and Lauren Williams.
\newblock Network parametrizations for the {G}rassmannian.
\newblock {\em Algebra Number Theory}, 7(9):2275--2311, 2013.

\bibitem[Wil]{Maplecode}
Lauren Williams.
\newblock Maple code, available by request.

\bibitem[Wil13]{HWilliams:KMcluster}
Harold Williams.
\newblock Cluster ensembles and {K}ac-{M}oody groups.
\newblock {\em Adv. Math.}, 247:1--40, 2013.

\bibitem[Yon03]{Yong}
Alexander Yong.
\newblock Degree bounds in quantum {S}chubert calculus.
\newblock {\em Proc. Amer. Math. Soc.}, 131(9):2649--2655, 2003.

\end{thebibliography}

\end{document}